\documentclass[a4paper,12pt,oneside]{amsart}

\usepackage{amsmath,amsfonts,amssymb,amsthm,amsopn,amsxtra}

\newtheorem{Thm}{Theorem}[section]    
\newtheorem{Lem}[Thm]{Lemma}
\newtheorem{Pro}[Thm]{Proposition}
\newtheorem{Cor}[Thm]{Corollary}

\newtheorem{Th}{Theorem}

\theoremstyle{definition}
\newtheorem{Def}[Thm]{Definition}
\newtheorem{varrem}[Thm]{}
\newtheorem{Ex}[Thm]{Example}
\newtheorem{Exs}[Thm]{Examples}

\theoremstyle{remark}
\newtheorem{Rem}[Thm]{Remark}
\newtheorem{Rems}[Thm]{Remarks}

\DeclareMathOperator{\ad}{ad}
\DeclareMathOperator{\id}{id}
\DeclareMathOperator{\Aut}{Aut}
\DeclareMathOperator{\nil}{nil}
\DeclareMathOperator{\im}{im}
\DeclareMathOperator{\GL}{GL}
\DeclareMathOperator{\rk}{rk}

\newcommand{\R}{\mathbb{R}}        
\newcommand{\N}{\mathbb{N}}        
\newcommand{\Z}{\mathbb{Z}}        
\newcommand{\C}{\mathbb{C}}        
\newcommand{\Q}{\mathbb{Q}}        
\newcommand{\T}{\mathbb{T}}        %
\newcommand{\Cal}{\mathcal}
\newcommand{\fr}{\mathfrak}
\begin{document}
\title[]{On the Structure of Groups with Polynomial Growth III}
\author{Viktor  Losert}
\address{Fakult\"at f\"ur Mathematik, Universit\"at Wien, Strudlhofg.\ 4,
  A 1090 Wien, Austria}
\email{Viktor.Losert@UNIVIE.AC.AT}
\date{4 February 2020}

\begin{abstract}
We show that a compactly generated locally compact group of polynomial growth
having no non-trivial compact normal subgroups can be embedded as a co-compact
subgroup into a semidirect product of a connected, simply connected,
nilpotent Lie group and a compact group. There is also a uniqueness statement
for this extension.
\end{abstract}
\maketitle

\baselineskip=1.3\normalbaselineskip
\setcounter{section}{-1}

\section{Introduction and Main results} 
\smallskip
Let $G$ be a locally compact (l.c.), compactly generated group.
$\lambda$ denotes a Haar measure on $G$ and  $V$ a compact neighbourhood
of the identity $e$, generating $G$. The group $G$ is said to be of {\it
polynomial growth}, if there exists $d\in \N$ such that
$\lambda (V^n)=O(n^d)$ \,for $n\in \N$\,. The group $G$ is called
{\it almost nilpotent}, if it has a nilpotent subgroup $H$ such that 
$G/H$ is compact. A classical result of Gromov \cite{Gr} asserts that a
finitely generated discrete group has polynomial growth
if and only if it is almost nilpotent. Any almost nilpotent group has
polynomial growth, but it is well known that the converse is
no longer true in the non-discrete case (see \cite{Lo2}\;1.4.3 for explicit
examples). Nevertheless, it turns out that there are
very close relations between the two classes and this will be the main
object of the present paper.

If $G$ is any compactly generated l.c.\;group of polynomial growth, it has a
maximal compact normal subgroup $C$ \,(\cite{Lo2}\;Prop.\,1).
Therefore, we will formulate the main theorems for groups having no
non-trivial compact normal subgroups. $G/C$ is always a Lie group
(\cite{Lo1}\;Th.\,2). Any compactly generated Lie group $G$ of polynomial
growth has a maximal nilpotent normal subgroup $N$, the (non-connected)
nilradical of $G$, denoted by $N = \nil (G)$ \,(\cite{Lo2}\;Prop.\,3). In
the discrete case, this is
called the Fitting subgroup (\cite{Se}\;p.\,15).
\begin{Th}  \label{th1}
Let $G$ be a compactly generated l.c.\;group of polynomial growth having no
non-trivial compact normal subgroups. $N = \nil(G)$ shall be
its (non-connected) nilradical. Then there exists a closed subgroup $L$ of
$G$ such that $G = NL$ \,and \,$L/\nil(L)$ is compact \,(in particular,
$L$ is almost nilpotent).
\end{Th}
For discrete polycyclic groups there is a similar result about nilpotent
almost-supplements for the Fitting subgroup (\cite{Se}\;Sec.\,3C; see our 
Remark\;\ref{rem31} for further discussion).
\begin{Th}  \label{th2}
Let $G$ be a compactly generated l.c.\;group of polynomial growth having no
non-trivial compact normal subgroup. Then $G$ can be embedded as a closed
subgroup into a semidirect product \,$\widetilde G = \widetilde N\rtimes K$
such that $K$ is compact, $\widetilde N$ is a connected, simply connected
nilpotent Lie group, $K$ acts faithfully on $\widetilde N$ and
\,$\widetilde G/G$ is compact.
\end{Th}
Then $\widetilde G$ is also a Lie group; but $G$ need not be normal in
$\widetilde G$ \,(see Example\;\ref{ex412}\,(c), (f)\,). Thus, although $G$
need not be almost nilpotent, it is always contained as a co-compact subgroup
in an almost nilpotent (and almost connected) group $\widetilde G$\,. For
$G$ connected, this was shown in \cite{Ab2}\;Th.\,3.6 (see also
Remark\;\ref{rem411}\,(b)\,).

It follows that any group $G$ as in Theorem\;\ref{th2} has a faithful linear
representation (Corollary\;\ref{cor36}), $G$ is isomorphic to a distal
linear group (as considered in \cite{Ab1}). $\widetilde G$~is isomorphic to
a real-algebraic linear group which is (for $\widetilde G$ minimal) an
algebraic hull of
$G$ in the sense of \cite{Ra}\;Def.\,4.39 \,(see Remark\;\ref{rem411}\,(a) for
further discussion).
\vskip 1mm
It turns out that the minimal extensions $\widetilde G$ as above (or
more specifically, with $K$ chosen minimal) are
determined uniquely up to isomorphism.
\begin{Th}  \label{th3}
Let $G,\;\widetilde G,\; \widetilde G'$ be l.c.\;groups,
$j\!: G \to \widetilde G,\;j'\!: G \to \widetilde G'$
shall be continuous, injective homomorphisms such
that $j(G),\:j'(G)$ are closed, $\widetilde G/j(G),\;\widetilde G'/j'(G)$
compact, $\widetilde G = \widetilde N \rtimes K\,,\ \,
\widetilde G' = \widetilde N' \rtimes K'$ with $K,\,K'$
compact, $\widetilde N,\,\widetilde N'$ connected,
simply connected nilpotent, $\widetilde Nj(G)$ dense in $\widetilde G$ \,
and $K'$ acting faithfully on $\widetilde N'$.
\\
Then there exists a unique continuous homomorphism
\,$\Phi\!: \widetilde G \to \widetilde G'$ such that
\,$\Phi \circ j = j'$. \,$\Phi$ is surjective iff \ $ \widetilde N'j'(G)$
is dense in $\widetilde G'$.
\,$\Phi$ is injective iff  $K$ acts faithfully on $\widetilde N$\,.
\vspace{-2mm}
\end{Th}\noindent
Thus, if \,$\widetilde G,\,\widetilde G'$ are given as in Theorem\;\ref{th2}
and \,$\widetilde Nj(G),\,\widetilde N'j'(G)$ are both dense (which can
always be attained by minimizing $K,K'$), then $\Phi$ is an isomorphism.
Due to this uniqueness, we call a group $\widetilde G$ as in
Theorem\;\ref{th2} with $\widetilde NG$ dense in $\widetilde G$,
the {\it algebraic hull} of $G$ and
\,$\widetilde N = \nil(\widetilde G)$ the {\it connected
nil-shadow} of $G$.\vspace{3mm plus .5mm}

Theorem\;\ref{th1}\;and\;\ref{th2} are based on the splitting techniques 
introduced by Malcev and developed further by Wang, Mostow and Auslander
\,(see also \cite{Au}). We build upon \cite{Wan} and extend it in
Section\;2 for
our purpose (see \ref{di21}, Remark\;\ref{rem37} and Remarks\;\ref{rem411} for
further discussion). Section\;3 contains the proofs of
Theorem\;\ref{th1}\;and\;\ref{th2}.
In Section\;4 the proof of Theorem\;\ref{th3} is given, based on various
structural properties of subgroups of semidirect products like those
appearing in Theorem\;\ref{th2}. Examples\;\ref{ex412} contains various
examples for the algebraic hull and related objects.
\medskip
\section{Notations and auxiliary results} 
\smallskip
\begin{varrem}\label{di11}
If $\Cal B$ is a group acting on $G$ by automorphisms, then $G$ is said to
be an \linebreak $FC^-_{\Cal B}$-group if the orbits
\,$\{\alpha (x)\! : \alpha \in \Cal B\}$ are relatively compact in $G$ for
all $x \in G$\,. For $\Cal B$ the inner automorphisms,
$G$ is called an $FC^-$-group. $G$ is called a {\it generalized
$\overline{FC}$-group} if there exists a series 
\,$G = G_0 \supseteq  G_1 \supseteq\dots\supseteq  G_n = (e)$ \,of closed
normal subgroups of $G$ such that
$G_i/G_{i + 1}$ is an $FC^-$-group and compactly generated for
$i = 0,\dots,n - 1$ \,(see \cite{Lo2}\;1.2.1 for further
discussion and references). Any compactly generated group of polynomial
growth is a generalized $\overline{FC}$-group. As worked
out in \cite{Lo2}, generalized $\overline{FC}$-groups have some nice
algebraic properties, the class contains all discrete polycyclic
groups, connected solvable groups and compact groups (thus it should allow
unified formulations for some results of \cite{Ra}\;Ch.\,III
that are developed there separately for the discrete and the connected case).
Conversely, every generalized $\overline{FC}$-group can be built
up from members of these subclasses.
\end{varrem}
\begin{varrem}\label{di12}
We refer to \cite{Ba} and \cite{War} for basic results on the algebraic
theory of nilpotent groups. If $G$ is a connected nilpotent
Lie group, then $G$ is simply connected iff it is torsion free
(\cite{Va}\;Th.\,3.6.1). If $\fr g$ denotes the Lie algebra of $G$, then 
the exponential function \,$\exp\!:\fr g \to G$ is always
surjective and if $G$ is simply connected, then \,$\exp$
defines a homeomorphism.

If $N$ is any compactly generated, torsion free nilpotent group, it can
always be embedded as a closed subgroup into a connected, simply
connected nilpotent Lie group $N_{\R}$ such that $N_{\R}/N$ is compact
\,(but $N$ need not be normal). $N_{\R}$ is called the
{\it (real) Malcev\,-\,completion} of $G$
(see \cite{Ba}\;Ch.\,4, \cite{War}\;Sec.\,11,12,
\cite{Au}\;Ch.\,II and \cite{Ra}\;Ch.\,II). $N_{\R}$ is determined
uniquely up to isomorphism. If $\varphi\!: N \to G$ is any continuous
homomorphism into a connected, simply connected nilpotent Lie group~$G$\,,
it has a unique extension \,$\varphi_{\R}\!: N_{\R} \to  G$\,.
\end{varrem}
\begin{varrem}\label{di13}
{\it Semidirect products:} \;$G = H \rtimes K$ means that $H,K$ are closed
subgroups of the l.c.\,group $G$, \,$H$ normal, \,$G = HK,\
\,H \cap K = \{e\}$
\,(``internal product"). In most cases we will follow \cite{HR}\;(2.6)
\,(i.e., the left factor is normal). The restrictions of the inner
automorphisms define a continuous action of $K$ on $H$, we write
\,$k\circ h = k\,h\,k^{-1}$.
Conversely, if l.c.\;groups $H, K$ are given and a continuous action of
$K$ on $H$ \,(i.e., a continuous homomorphism \,$K \to \Aut(H)$ -
compare \cite{Ho}\;III.3) one can define the (``external") semidirect
product $H\rtimes K$ by considering the cartesian product $H\times K$ of
the topological spaces with group multiplication 
\,$(h_1,k_1)\,(h_2, k_2) = \linebreak
(h_1 (k_1 \circ h_2), k_1 k_2)$. Then $H$ is
isomorphic to the closed normal subgroup 
\,$\{(h,e) : h \in H\}$, similarly for $K$. For $\sigma$-compact l.c.\;groups
both viewpoints are equivalent (respectively, they lead to isomorphic
groups) and we will not distinguish further on.\vspace{-1mm}
\end{varrem}
Next, we give a result on combining two group extensions
(``pasting of two groups along a common subgroup"). This is probably
known, but we could not find a reference.\vspace{-3mm}
\begin{Pro}  \label{pro14}
Let $G, H_1$ be l.c.\;groups such that $H = G\cap H_1$
(with induced topology
from $G$) is a closed normal subgroup of $G$ and a subgroup of $H_1$
\,for which the inclusion \,$H\to H_1$ is continuous. Assume that $G$
acts continuously on $H_1$ by automorphisms (see \ref{di13}) such that
\,$x\circ h = x\,h\,x^{-1}$ whenever
\,$(x,h)\in (H\times H_1)\cup (G\times H)$. \\
Then there exists a l.c.\;group $G_1$ and continuous homomorphisms
\,$j\!: G \to G_1\,,\linebreak
j_1\!: H_1 \to G_1$ such that $j_1(H_1)$ is a closed normal subgroup of 
$G_1\,,\ j = j_1 \text{ on } H\,,\linebreak
G_1 = j_1(H_1)\,j(G),\
j(H)=j_1(H_1)\cap j(G),\ j_1(x\circ h) = j(x)\,j_1(h)\,j(x)^{-1} 
\text{ for all }\linebreak
x \in G\,,\,h \in H_1$\,. \;If $H$ is closed in $H_1$\,,
then $j(G)$ is closed in $G_1$\,.\vspace{-2.5mm}
\end{Pro}\noindent
If $H$ is closed in $H_1$\,, and the topologies of $G$ and $H_1$ coincide
on $H$\,, 
it will result from the proof that $j$ defines a topological isomorphism of
$G$ and $j(G)$, \,similarly for~$j_1$\,. For this reason, we will skip
$j\,,\,j_1$ in general and consider $G,\,H_1$ as subgroups of $G_1$\,. Then
the properties amount to \,$G_1 = H_1\,G\,,\ H_1 \triangleleft G_1\,,\
x \circ h = x\,h\,x^{-1} \text{ for all }\linebreak x \in G,\;h \in H_1$\,.
\begin{proof}
Put \,$G^* = H_1 \rtimes G$ \,(with respect to the given action, see
\ref{di13})
and \,$H^* = \{(x^{-1},x) : x \in H\}$. Then by easy computations, it 
follows from the properties of the action that $H^*$ is a closed normal
subgroup of $G^*$ \,(e.g., the subgroup property follows from
$(x^{-1},x)\,(y^{-1},y) = (x^{-1} (x \circ y^{-1}), xy) =
(y^{-1} x^{-1}, xy) \text{ \,for } x, y \in H$, since by assumption, 
$x \circ y^{-1} = x\,y^{-1} x^{-1} \text{ for } x \in H$). Put 
\,$G_1 = G^*/H^*,\ j(x)=(e,x)H^* \text { for } x \in G,\
j_1(h) = (h,e)H^* \text{ for } h \in H_1$\,.
This satisfies the properties stated above.
\vspace{-2mm}\end{proof}
\begin{varrem}\label{di15}
See \cite{HR}\; for basic properties of l.c.\;groups and \cite{Va} for Lie
groups. $e$ will always denote the unit element of a group $G$. $Z(G)$
stands for the centre of $G$, $\Aut(G)$ will denote the group of
topological automorphisms of $G$ with its standard topology 
(see \cite{HR}\;(26.3)). $G^0$ denotes the connected component of
the identity.
\end{varrem}
\smallskip
\section{The splitting technique} 
\smallskip
\begin{varrem}\label{di21} In this section, we use the following setting.
$G$ shall be a
(not necessarily connected) Lie group whose topological commutator group
${[G, G]}^-$ is compactly generated, nilpotent, torsion free and such
that $G/G^0$ is nilpotent (\;$^0$ denoting the connected component of
the identity). In particular, $G$ is an extension of a nilpotent group
by an abelian group and therefore solvable.

In addition, we consider a fixed closed subgroup $H \supseteq {[G,G]}^-$
such that $H$ is compactly generated, nilpotent and torsion free. Then
$G/H^0$ is nilpotent (observe that $G^0/{[G,G]}^{-0}$ is central in
$G/{[G,G]}^{-0}$, hence $G/{[G,G]}^{-0}$ is nilpotent). In principle, the
proofs could also be done without specifying such an $H$, but this
approach makes it easier to use the results of \cite{Wan}. If $G$ is
compactly generated and has no non-trivial compact normal subgroups,
one can always take $H=N=\nil(G)$ (the nilradical - see 2.8), then
$\Aut_H (G)$, defined below, coincides with $\Aut (G)$. In the terminology of
\cite{Au}, $H$ is a CN-group. If $H$ is connected and open in $G$ and
$G/H$ is finitely generated, this coincides with the class of solvable
groups $G$ considered in \cite{Wan}\;sec.\,6 (contained in the class
of $\Cal S$-groups defined in \cite{Wan}\;sec.\,10). More generally,
if $H$ is connected and
$G/H$ compactly generated, one gets the $\epsilon$-category of \cite{To}.
The main results will be Proposition\;\ref{pro215} and
Corollary\;\ref{cor215a} on existence of the splitting
(containing \cite{Wan}\;(10.2)\,) and Proposition\;\ref{pro220} on uniqueness
up to conjugacy.

$\Aut(G)$ will denote the group of topological automorphisms of $G$
\,(with its
standard topology, \cite{HR}\;(26.3)\,). For $x\in G$, $\iota_x(y) = xyx^{-1}$
denotes the corresponding inner automorphism of $G$,
\,$\iota\!:G\to\Aut (G)$ is a homomorphism,
\,$\iota_{\theta(x)}=\theta \circ \iota_x \circ \theta^{-1}$ for
$\theta\in\Aut(G)$. As in \cite{Wan}\;p.\,2, we say that
$\theta\in\Aut(G)$ is $\it{unipotent}$, if there exists an integer $n > 0$
such that $(\ad \theta)^n$ is the identity on $G$, where
$(\ad \theta)(x) = \theta (x)\,x^{-1}$ \ (if $G$ is connected, this is
equivalent to the statement that $d\mspace{.5mu}\theta - \id$ is nilpotent on
the Lie algebra of~$G$\;--\;recall that $G$ is solvable). $H/H^0$ is finitely
generated, nilpotent and torsion free, $G/H$ abelian.
By well known results (compare \cite{War}\;9.3,\;9.5) nilpotency of the
group $G/H^0$ is equivalent to unipotency of the automorphisms of $H/H^0$
induced by $\iota_x\ (x \in G)$. We put 
$\Aut_H (G) = \{\theta\in\Aut(G): H\ \text{is}\ \theta\text{-invariant}\}$
\ (\,$\theta|H$ will denote the
restriction of the mapping) and (extending \cite{Wan}\;p.\,8)\vspace{1mm}
\begin{equation*}\begin{split}
\Aut_1(G) = \{\theta \in \Aut_H (G) : \;
\theta \ &\text{induces the identity on}\ G/H \ \\[-1mm]
&\text{and a unipotent automorphism of} \ H/H^0\}\ .
\end{split}\end{equation*}
Clearly, this depends on $H$, so we will sometimes write more
precisely $\Aut_{1, H}(G)$. Note that if $H$ is not connected,
$\Aut_1(G)$ need not be a subgroup, but it is always
$\Aut_H(G)$-invariant. The assumptions on $G, H$ imply that
$\iota_x\in\Aut_1(G)$ for all $x\in G$. For $\theta\in\Aut(G)$, we write
\,$G_{\theta} = \{x \in G : \theta(x) = x\}$. If $G$ is connected, we
call \,$\theta \in \Aut(G)$ $\it{semisimple}$ if the corresponding
linear transformation $d\theta$ of the Lie algebra $\fr g$ of
$G$ is semisimple (i.e., it diagonalizes after suitable extension of
the base field). Recall that any $\theta\in\Aut(\R ^n)$ has a
unique decomposition
\,$\theta = \theta_s \circ \theta_u = \theta_u\circ\theta_s$\,,
where $\theta_s$ is semisimple, $\theta_u$
unipotent \,(multiplicative Jordan decomposition -- see
\cite{Bo1}\;VII,\,Th.\,1,\,p.\,42). $\theta_s$ is a polynomial of $\theta$,
hence any
\text{$\theta$-invariant} subspace is also $\theta_s$-invariant. If
$\theta$ is an automorphism of a Lie algebra, the same is true for
$\theta_s, \theta_u$ (an easy consequence of
\cite{Bo2}\;VII,\,Prop.\,12,\,p.\,16). This carries over to automorphisms of
connected, simply connected Lie groups.\vspace{1.5mm}

If $G, H$ are given as above, we can consider the Malcev completion
$H_{\R}$ of $H$ and by \ref{di12} and Proposition\;\ref{pro14}, we can
consider $G$ and $H_{\R}$ as
closed subgroups of a (uniquely determined) Lie group $G_{\R}$ such that
$H_{\R}$ is normal in $G_{\R}\,,\ G\cap H_{\R} = H$ and
$G_{\R} = H_{\R}G$\,. Then $G_{\R}/H_{\R}\cong G/H$\,, in particular, the pair
$G_{\R}, H_{\R}$ satisfies again our general requirements and
$G_{\R}/G$ \,(being homeomorphic to $H_{\R}/H$) is compact \,(but be aware
that $G_{\R}$ may have non-trivial compact normal subgroups, even if
$G$ does not, see also Corollary\;\ref{cor35}; furthermore, $G_\R$ depends
on $H$ and it need not be connected). Any $\theta \in \Aut_H(G)$ has
a unique extension
$\theta_{\R} \in \Aut_{H_{\R}} (G_{\R})$\,, $\theta \in \Aut_{1, H} (G)$
implies $\theta_{\R} \in \Aut_{1, H_{\R}} (G_{\R})$\,. If $\theta$ is
unipotent, the same is true for~$\theta_{\R}$\,.
\vspace{-4mm}\end{varrem}
\begin{Lem}\label{le22}
For $\theta \in \Aut_1(G)$, the following statements are
equivalent:
\begin{enumerate}
\item[(i)] \ $\theta|\,G^0$ is semisimple, \;$G = G^0\,G_{\theta}$\,.
\item[(ii)] \ $\theta|\,H^0$ is semisimple, \;$G = H^0\,G_{\theta}$\,.
\item[(iii)] \ $\theta_{\R}|H_{\R}$ is semisimple, \;
$G_{\R} = H\,(G_{\R})_{\theta_{\R}}$\,.
\end{enumerate}
\vspace{-3mm}\end{Lem}
\begin{proof} \;
(i) $\Rightarrow$ (ii)\,: it will be enough to show that
$G^0 \subseteq H^0 G_{\theta}$\,. The assumption \linebreak
$\theta\in\Aut_1(G)$
implies $(\ad \theta) (G^0) \subseteq H^0$, consequently $d \theta$
induces the
identity on $\fr g/\fr h$ \,(where $\fr h$ denotes the Lie
algebra of $H$). $d\theta$ being semisimple, it follows that
$\fr g = \fr g_{\theta} + \fr h$ \,(where
$\fr g_{\theta} = \{X \in \fr g : d\theta (X) = X\}$). Clearly,
$\fr g_{\theta}$ is the Lie algebra of $G_{\theta}$ and it follows
(as in \cite{Va}\;L.\,3.18.4) that $H^0 G_{\theta}^0$ is open in $G$.
\\[0mm plus .1mm]
(ii) $\Rightarrow$ (iii)\,: We have $H = H^0H_{\theta}$\,. It is easy to
see that $(H_{\theta})_{\R} \subseteq (H_{\R})_{\theta_{\R}}$, thus
$H_{\R} = H^0 (H_{\R})_{\theta_{\R}}$\,. Then $G_{\R} = G H_{\R}$
implies $G_{\R} = H^0 (G_{\R})_{\theta_{\R}}$\,. In addition, we get a
decomposition of the Lie algebra $\fr h_{\R}$ of $H_{\R}$ into a sum
(similar as above) and then semisimplicity of
$d\theta|\,\fr h\ (= d\theta_{\R}|\,\fr h)$ implies semisimplicity of
$d\theta_{\R}$\,.
\\[0mm plus .1mm]
(iii) $\Rightarrow$ (i)\,: We have
\,\,$\ad\theta_{\R}(H_{\R})\subseteq H^0$ \,and this implies that
$d\theta_{\R}$ induces the identity on
$\fr h_{\R}/\fr h$\,. As in the first step, we get that
$H_{\R} = H^0 (H_{\R})_{\theta_{\R}}$, hence $H = H^0 H_{\theta}$ and
$G_{\R} = H^0 (G_{\R})_{\theta_{\R}}$\,. Since
$(G_{\R})_{\theta_{\R}} \cap G = G_{\theta}$\,, this gives
\,$G = H^0 G_{\theta} \subseteq  G^0 G_{\theta}$\,. We get a surjective
homomorphism from $G^0 \cap G_{\theta}$ to $G^0/H^0$ and since
$G^0 \cap G_{\theta}$ is \linebreak
$\sigma$-compact, this is an open mapping
(\cite{HR}\;Th.\,5.29). Hence the mapping
has to remain surjective on $(G^0 \cap G_{\theta})^0 = G_{\theta}^0$
\,and it follows that $G^0 = H^0 G_{\theta}^0$\,. As in the second step,
this gives a decomposition of $\fr g$ and semisimplicity of $d\theta$\,.
\end{proof}
\begin{Def}  \label{def23}
$\theta \in \Aut (G)$ is called {\it semisimple} if it satisfies
condition (i) of Lemma\;\ref{le22}.

\noindent Note that if $H$ is any subgroup of $G$ as in \ref{di21} such that
$\theta\in \Aut_{1,H}(G)$ holds, then $\theta$~satisfies
Lemma\;\ref{le22}\,(ii),\,(iii)
as well. In particular (by (iii)), $\theta_{\R} \in \Aut (G_{\R})$ is
again semisimple. But the converse is not true in general (take e.g.,
$G = \Z,\ \theta (n) = -n$, then $\theta_{\R}$ is semisimple but
$\theta$ is not). If $H$ is connected and open, $G/H$ finitely
generated ($G, H$ as in \ref{di21}), $\theta \in \Aut_1(G)$, our
Definition is equivalent to that of \cite{Wan}\;sec.\,6. If \,(for general
$G, H$ as in \ref{di21}) \,$\theta', \theta'' \in \Aut_1(G)$ are commuting
semisimple automorphisms, it follows as in the proof of
\cite{Wan}\;(8.8) \,(see also Corollary\;\ref{cor26} below with
$L = G_{\theta}$) \,that $\theta' \circ \theta''$ is again semisimple and
$\theta' \circ \theta'' \in \Aut_1(G)$.
\end{Def}
\begin{Lem}   \label{le24}
Assume that \,$G,H$ are given as in \ref{di21}, $\theta \in \Aut_1(G)$
and let $\rho$ be the semisimple part of \;$\theta|H^0$. Then there exists
a unique semisimple automorphism $\theta_s \in \Aut_1(G)$ which
extends $\rho$ and commutes with $\theta$\,.
\\
If $H$ is connected, we have
\,$G_{\theta_s} = \{x \in G : \theta (x)\,x^{-1} \in H_{\rho}\,\}$.
For general $H$, we have
\,$G_{\theta_s} = \{x \in G : (\ad \theta)^n (x) = e
\text{ for some } n > 0\}$ and
\,$\theta_s = (\theta_{\R})_s|\,G$\,. In particular
$G_{\theta} \subseteq G_{\theta_s}$\,.
\end{Lem}
\begin{proof}
First, assume that $H$ is connected. To prove uniqueness, it will be
enough (by Lemma\;\ref{le22}\,(ii)\,) to verify the first formula
for $G_{\theta_s}$ (then $G_{\theta} \subseteq G_{\theta_s}$
follows as well in this case). If $x \in G_{\theta_s}$\,, then (using
$\theta \circ \theta_s = \theta_s \circ \theta$ and
$\theta \in \Aut_1 (G)$\,) we get
\,$\theta (x) x^{-1} \in G_{\theta_s} \cap H = H_{\rho}$\,. For the converse,
take $x \in G$\,, then we have (Lemma\;\ref{le22}\,(ii)\,)
\,$x = yz$ with $y \in H,\ z \in G_{\theta_s}$\,.
If $\theta (x) x^{-1} \in H_{\rho}$\,, we can (since
$\theta (x) x^{-1} = \theta (y) \theta (z) z^{-1} y^{-1}$) apply
\cite{Wan}\;(5.1) with $w = y,\ v = \theta (z) z^{-1}$ \;(note a misprint
in \cite{Wan}: it should read $\theta (w) vw^{-1}$ instead of
$\rho (w)v w^{-1}$) \,and conclude that
$y \in H_{\rho} \subseteq G_{\theta_s}$\,. This gives
$x \in G_{\theta_s}$\,. Furthermore,
concerning the second formula for $G_{\theta_s}$\,, \cite{Wan}\;(5.1)
shows that for $w\in H,\ (\ad\theta)(w)\in H_{\rho}$ implies
$w\in H_{\rho}$\,. Thus (by induction) \,$(\ad \theta)^n (x) = e$ for some
$n > 0$ \,implies
\,$(\ad\theta)(x)\in H_{\rho}$\,, i.e., $x \in G_{\theta_s}$\,. The other
inclusion follows from the fact that $\theta|H_{\rho}$ is
unipotent.
\\[.3mm plus .8mm]
Concerning existence of $\theta_s$\,, this follows from \cite{Wan}\;(8.1)
if $H$ is open (and still connected): \ he assumes that $G/H$ is
finitely generated, but (recall that $G/H$ is abelian) any finite
subset of $G$ is contained in an open subgroup $G_1$ of $G$ such
that $G_1/H$ is finitely generated; since uniqueness has already
been proved, this allows to define the automorphism $\theta_s$
unambiguously on all of $G$; since $H$ is open and
$\theta_s$-invariant, continuity holds automatically. If $H$ is not
open, we can refine the topology to make it open (observe that on $H$
the two topologies coincide). Then the argument above produces a
unique extension $\theta_s$ of $\rho$ for the refined topology. To
prove continuity of $\theta_s$ for the original topology, we may
assume that $G$ is $\sigma$-compact (since $\theta \in \Aut_1(G)$,
any subgroup $G_1$ containing $H$ is automatically
$\theta$-invariant, the same for $\theta_s$). The description of
$G_{\theta_s}$ that was demonstrated above shows that
$G_{\theta_s}$ is closed in the original topology. By
Lemma\;\ref{le22}\;(ii) and \cite{HR}\;Th.\,5.29, $G$~is
topologically isomorphic to a quotient of the semidirect product
$H\rtimes G_{\theta_s}$\,. On $H \rtimes G_{\theta_s}$\,, the mapping 
$\theta (y, z) = (\rho (y), z)$ is a group automorphism (since
$\theta_s$ is known, to be a group automorphism) and $\theta$ is clearly
continuous, hence the same is true for the induced mapping $\theta_s$
on the quotient group. This finishes the proof when $H$ is
connected.
\\[0mm plus .3mm]
If $H$ is not connected, we consider the extension $\theta_{\R}$ to
$G_{\R}$ (see \ref{di21}). Put
$\rho_{\R} = (\theta_{\R}|H_{\R})_s$\,. Since $H^0$ is
$\theta_{\R}$-invariant,
it is also $\rho_{\R}$-invariant. Hence, by uniqueness of Jordan
decomposition, $\rho = \rho_{\R}|H^0$. Since $\theta \in \Aut_1 (G)$,
it induces a unipotent transformation on $H/H^0$. Its extension to
$H_{\R}/H^0$ coincides \,(by uniqueness) with the transformation induced
by \,$\theta_{\R}|H_{\R}$\,. Thus \,$\theta_{\R}|H_{\R}$ induces a unipotent
transformation on $H_{\R}/H^0$, hence $\rho_{\R}$ induces the identity
on $H_{\R}/H^0$ which 
implies
\,$H_{\R}\subseteq  H^0\,(H_{\R})_{\rho_{\R}}$\,. Let $\theta_s$ be any
extension of $\rho$ as in the Lemma. From
\,$\theta_s\in\Aut_1(G)$, it follows as above that
\,$(\theta_s)_{\R}|H_{\R}$ induces the identity on $H_{\R}/H^0$. As an
easy consequence, $\theta_{\R} \circ (\theta_s)^{-1}_{\R}$ is
unipotent on $H_{\R}$\,, thus uniqueness of the Jordan decomposition implies
$(\theta_s)_{\R}|H_{\R} = \rho_{\R}$\,. As observed after
Definition\;\ref{def23},
$(\theta_s)_{\R} \in \Aut_{1, H_{\R}}(G_{\R})$ is semisimple, hence
uniqueness in the connected 
case implies
\,$(\theta_s)_{\R} = (\theta_{\R})_s$\,, thus
$\theta_s = (\theta_{\R})_s|\,G$\,. This
proves uniqueness in the general case. \;For existence, it suffices to
show that $G$ is $(\theta_{\R})_s$-invariant. But semisimplicity
implies \,$G_{\R} = H_{\R}\,(G_{\R})_{(\theta_{\R})_s}$ and we already
know that \,$H_{\R} = H^0\,(H_{\R})_{\rho_{\R}}$\,, thus
\,$G_{\R} = H^0\,(G_{\R})_{(\theta_{\R})_s}$ which implies invariance of
any subgroup containing $H^0$ \,(and also that
\,$(\theta_{\R})_s|\,G \in \Aut_1(G)$\,). Finally, since
\,$G_{\theta_s} = G \cap (G_{\R})_{(\theta_{\R})_s}$ and
\,$\ad\theta = (\ad\theta_{\R})|\,G$\,,
the formula for $G_{\theta}$ follows from the connected case. 
\vspace{-4mm plus.5mm}
\end{proof}
\begin{varrem}	       \label{di25}
For $\theta \in \Aut_1 (G)$, we write \,$s(\theta) = \theta_s$ \,
(defined by Lemma\;\ref{le24}), $\theta_u = \theta \circ \theta^{-1}_s$.
It follows easily \pagebreak
that $\theta_u \in \Aut_1(G)$ is unipotent,
$\theta = \theta_s \circ \theta_u = \theta_u \circ \theta_s$\,,
and (by the corresponding result for operators on vector spaces and
Lemma\;\ref{le24}) this is the only such decomposition in $\Aut_1(G)$ for
which the factors commute with $\theta$\,. Lemma\;\ref{le22}\,(i) implies that
$d\theta_s$ is the semisimple part of $d \theta$ (on the Lie algebra
$\fr g$). Combined with the formula for $G_{\theta_s}$\,, it follows
that $\theta_s$ (and hence $\theta_u$ as well) does not depend on
the choice of $H$, as long as there exists {\it some} $H$ for which
$\theta \in \Aut_{1, H} (G)$ \ (note that $\theta \in \Aut_{1,H}(G)$ holds
iff for $G_1 = G \rtimes \Z$ with the action defined by $\theta$, the
pair $G_1, H$ satisfies the assumptions of \ref{di21}; in particular, by
\ref{di21}, existence of such an $H$ can be characterized by the conditions
that the closed subgroup generated by $[G,G]$ and $(\ad \theta) (G)$
should be compactly generated, nilpotent and torsion free and $\theta$
should induce a unipotent transformation on $G/G^0$; if $G$ is
compactly generated and has no non-trivial compact normal subgroup,
one can always take $H = N$, as defined in Remark\;\ref{rem28}). By
uniqueness, we have
\,$s(\psi\circ\theta\circ\psi^{-1}) = \psi\circ s(\theta)\circ\psi^{-1}$
\,for $\psi \in \Aut (G),\ \theta \in \Aut_1(G)$ \;(note
that $\psi \circ \theta \circ \psi^{-1} \in \Aut_{1, \psi (H)} (G)$\,).
In particular, if $\psi \in \Aut (G)$ commutes with $\theta$, it
commutes also with $\theta_s, \theta_u$ \,(see also \cite{Wan}\;(8.6)).
For $\theta = \iota_x\ \,(x \in G)$ we just write $s (x)\ (= s(\iota_x)$).
Note that in this case the inclusion $G_{\theta} \subseteq G_{\theta_s}$
implies that $\sigma = s(x)$ satisfies $\sigma (x) = x$\,.
Furthermore, we put\vspace{-1mm}
\[\Cal S =\{ s(x) : x \in G\}\,.\vspace{-1mm}\]
\vspace{.5mm}
The example after Definition\;\ref{def23} shows (for $G = \R^n, L = \Z^n$)
that if $\theta$
is semisimple on $G$ and $L$ is a general $\theta$-invariant subgroup,
then $\theta|L$ need not be semisimple in the sense of
Definition\;\ref{def23}. \,Furthermore, if $\theta$ is given by the matrix 
$\left(\begin{smallmatrix}
\,0 & \;0 & \ 1 \\
\,1 & \;0 & \ 1 \\
\,0 & \;1 & -1
\end{smallmatrix}\right)$,
then $L = \Z^3$ is not invariant under $\theta_{s}$\,.\vspace{.5mm}

Observe that if $G, H$ are as in \ref{di21} and $L$ is a closed subgroup of
$G$ such that $L\cap H^0$ is connected, then $L, L\cap H$ satisfy
again the assumptions of \ref{di21} \,(using that $(L\cap H)^0 = L\cap H^0$
and algebraically $L/(L \cap H^0) \cong  LH^0/H^0 \subseteq G/H^0$
holds; furthermore, since $H$ is nilpotent, any closed subgroup of $H$
is compactly generated - compare \cite{Lo2}\;Prop.\,2). If
$\theta \in\Aut_1 (G)$, $L$ is as above and $\theta$-invariant, then
\,$\theta|L\in \Aut_1(L)$.
\end{varrem}
\begin{Cor}\label{cor26}
Assume that $L$ is a closed subgroup of $G$ such that \,$L \cap H^0$ is
\linebreak connected, $\theta \in \Aut_1 (G)$ and $L$ is $\theta$-invariant.
\;Then $L$ is $\theta_s$-invariant and \linebreak
$\theta_s|\,L = (\theta|L)_s\,,\quad \theta_u|\,L = (\theta|L)_u$\,.
\\
In particular, if $\theta$ is semisimple, then $\theta|\,L$ is semisimple.
\end{Cor}
\begin{proof}
Put \,$\theta_L = \theta|L\,,\;H_L = L \cap H$\,. $(\theta_L)_s$ is
defined by Lemma\;\ref{le24} (see above) and we have
$L_{(\theta_L)_s} = L \cap G_{\theta_s}$\,.
$L^0$ is $\theta$-invariant and it follows from
the properties of Jordan decomposition (\ref{di21}) that $L^0$ is
$\theta_s$-invariant and that $(d \theta_L)_s$ is the
restriction of $(d \theta)_s = d \theta_s$\,. Thus
\,$(\theta_L)_s= \theta_s$ on $L^0$ and Lemma\;\ref{le22}\,(i)
(for $\theta_L$) gives the result.\vspace{-2.5mm}
\end{proof}
\begin{Lem}    \label{le27}
Assume that $\Cal C$ is a subgroup of $\Aut(G)$, \;$\Cal C_1$ is a
normal subgroup of $\Cal C$ such that
\,$\Cal C_1\subseteq \Aut_1(G),\ \Cal C_1$ is nilpotent and
$[\Cal C, \Cal C_1]$
consists of unipotent transformations. \,Then $s$ is a group
homomorphism on $\Cal C_1\,,\ s(\Cal C_1)$ is commutative and
centralizes $\Cal C$\,.\vspace{-2mm}
\end{Lem}
\begin{proof}
We use the notation for commutators \,$[\sigma, \tau] = \sigma
\tau \sigma^{-1} \tau^{-1}$ as in \cite{Lo3}\,p.129. We consider the
ascending central series
\,$(e) = \Cal C^{(0)}\subseteq\dots\subseteq\Cal C^{(k)} = \Cal C_1$
for $\Cal C_1$ \;(i.e.,
$\Cal C^{(i+1)}/\Cal C^{(i)}$ is the centre of
$\Cal C_1/\Cal C^{(i)}$) and put $\Cal C^{(k+1)} = \Cal C$\,. Take
$\sigma \in \Cal C_1$\,. By induction, we want to show that if
$\tau \in\Cal C^{(i)}$, then $\sigma_s$ commutes with $\tau$. This is
trivial
for $i = 0$, so we assume that the statement holds for $i-1$ (where
$i \geq  1$). We have $\tau' = [\sigma^{-1}, \tau] \in \Cal C^{(i-1)}$
(for $i = k+1$ use that $\sigma \in \Cal C_1$), hence
$\tau'$ commutes with $\sigma_s$ and by assumption, $\tau'$ is
unipotent. Observe that $\tau \sigma \tau^{-1} = \sigma \tau' =
\sigma_s \sigma_u \tau'$. Let $\fr h$ be the Lie algebra of $H$.
By \cite{Wan}\;(2.2), $\Cal C_1$ induces a triangular group of
transformations on $\fr h$ and then the same is true for the
unipotent parts of these transformations and the group generated by
them. It follows that the group generated by
$\{\xi_u : \xi \in\Cal C_1\}$ contains just unipotent transformations
on $H^0$, in
particular, $\sigma_u \tau'$ is unipotent on $H^0$. Consequently,
$\sigma_s|H^0$ is the semisimple part of \,$(\tau \sigma\tau^{-1})|H^0$ and
then uniqueness in Lemma\;\ref{le24} implies
\,$\sigma_s = s(\tau \sigma \tau^{-1}) = \tau \sigma_s \tau^{-1}$,
providing the induction step.
\\[1.3mm]
In particular, $\sigma_s$ commutes with $\Cal C_1$, hence (see
\ref{di25}), for any $\tau \in \Cal C_1$, it commutes also with
$\tau_s$ and $\tau_u$\,. Thus
\,$\sigma \tau = \sigma_s \tau_s \sigma_u \tau_u$ and
(recall that by \ref{def23} \,$\sigma_s \tau_s$
is semisimple) as above, we get
\,$s(\sigma \tau) = s(\sigma)\,s(\tau)$\,.
\end{proof}
\begin{Rem}    \label{rem28}
As in the previous proof (using the Lie algebra $\fr h_{\R}$ of
$H_{\R}$ instead of~$\fr h$), it follows from \cite{Wan}\;(2.3) that
the group $\Cal N$ generated by $\{(\iota_x)_u : x \in G\}$ is nilpotent,
contained in $\Aut_1 (G)$ and consists of unipotent transformations.
In particular,
$N = \nil (G) = \{x \in G : \iota_x\text{ unipotent}\}$ is a nilpotent
characteristic subgroup of $G$
containing $H$. It is the biggest nilpotent normal subgroup of $G$ (in
particular closed). We call it the (non-connected) $\it{nilradical}$
of $G$. If $G$ is compactly generated, existence of $N$ follows also
from \cite{Lo2}\;Prop.\,3, see also \cite{Wan}\;(9.1).

Take $x \in G$ and put $\sigma = s(x)$. It is an easy consequence that
for $y \in G,\linebreak  s(y) = \sigma$ holds iff
\;$y \in (N \cap G_{\sigma})\,x$\,.

If $\Cal C \subseteq \Aut_H(G)$, we put
\,$G_{\Cal C} =\bigcap_{\theta \in \Cal C}\,G_{\theta}$\,. Note that
$G_{\Cal C}\cap H^0$ is connected (since the fixed points correspond to
a linear
subspace of the Lie algebra $\fr h$ of $H$). Hence (see \ref{di25}),
$G_{\Cal C},\,G_{\Cal C}\cap H$ satisfy the conditions of \ref{di21}. \,We
put \,$L_{\Cal C} = \nil (G_{\Cal C})$\,. Then
\,$L_{\Cal C} \cap H^0 = G_{\Cal C} \cap H^0$. It is easy to see that
if $\langle{\Cal C}\rangle$
denotes the subgroup generated by $\Cal C$, then
$G_{\Cal C} = G_{\langle{\Cal C}\rangle}$\,.
\end{Rem}
\begin{Cor}	  \label{cor29}
Let $L$ be a nilpotent subgroup of $G$, $\sigma \in \Aut(G)$ such that
$L$ is $\sigma$-invariant and \,$\sigma(x)\,x^{-1}\in N$ for all $x\in L$\,.
Then $\sigma$ commutes with $s(L)$\,.\vspace{-2mm}
\end{Cor}\noindent
Observe that $\sigma(x)\,x^{-1} \in N$ holds for every $x$,
whenever $\sigma$ induces the
identity on $G/N$\,, in particular if $\sigma \in \Aut_1(G)$\,.
\vspace{-1mm}
\begin{proof}
Put \,$\Cal C_1 = \{\iota_x : x \in L\}$ and let $\Cal C$ be the
group generated by $\Cal C_1$ and $\sigma$\,. Then $\Cal C_1$ is
nilpotent, contained in $\Aut_1(G)$ and normal in $\Cal C$ (since
$L$ is $\sigma$-invariant).  $[G,G] \subseteq H \subseteq N$ and
$\sigma (x) x^{-1} \in N$ for $x \in L$ imply that
\,$[\Cal C,\Cal C_1] \subseteq \{\iota_y : y \in N\}$, \linebreak
hence $[\Cal C,\Cal C_1]$ consists of
unipotent transformations. Now Lemma\;\ref{le27} shows that
$\Cal C$~centralizes \,$s(\Cal C_1) = s(L)$\,. 
\end{proof}
\begin{Lem}	   \label{le210}
Assume that $\Cal C$ is a subgroup of $\Aut_1(G)$ containing only
semisimple transformations. If $\Cal C_0$ is a normal subgroup of
$\Cal C$ such that
\,$G_{\Cal C_0} \cap H^0 = G_{\Cal C} \cap H^0$, then
\,$G_{\Cal C_0} = G_{\Cal C}$\,.\vspace{-2mm}
\end{Lem}
\begin{proof}
Take $\theta \in \Cal C$\,. Normality of $\Cal C_0$ implies that
$G_{\Cal C_0}$ is $\theta$-invariant. By Corollary\;\ref{cor26} and
Lemma\;\ref{le22}, \,$\theta|\,G_{\Cal C_0} \cap H^0 = \id$ \,implies
\,$\theta|G_{\Cal C_0} = \id$\,.
\end{proof}
\begin{Cor}	  \label{cor211}
Let $\Cal C$ be a commuting subset of $\Aut_1(G)$ consisting of
semisimple transformations. Then we have
\,$G = H^0 G_{\Cal C} = NG_{\Cal C}$ and there exists a finite subset
$\Cal C_0$ of $\Cal C$ such that \,$G_{\Cal C_0} = G_{\Cal C}$\,.\vspace{-2mm}
\end{Cor}
\begin{proof}
Choose a finite subset $\Cal C_0$ of $\Cal{C}$ so that
\,$\dim(G_{\Cal C_0} \cap H^0)$ is minimal. Then Lemma\;\ref{le210} \,(applied
to the groups generated by $\Cal C_0$ and $\Cal C$) implies
\,$G_{\Cal C_0} = G_{\Cal C}$\,. The equation \,$G = H^0 G_{\Cal C_0}$
follows from Lemma\;\ref{le22}\,(ii) by induction on the cardinality of
$\Cal C_0$ \,(recall that $G_{\theta},\,G_{\theta} \cap H$ also
satisfy the assumptions of \ref{di21} and for
$\theta \in \Cal C_0\,,\linebreak G_{\theta}$ is $\Cal C_0$-invariant and
the restrictions are
semisimple by Corollary\;\ref{cor26}) -- compare \cite{Wan}\;(8.8).
\vspace{-2mm}\end{proof}
\begin{Lem}	   \label{le212}
Let $\Cal C$ be a subset of $\Aut(G)$ satisfying \,$G = N G_{\Cal C}$\,.
Then the following statements are equivalent: 
\begin{enumerate}
\item[(i)] $\sigma \in s (G_{\Cal C})$.
\item[(ii)] $\sigma \in \Cal S$ and it commutes with $\Cal C$\,.
\item[(iii)] $\sigma \in \Cal S$ and $G_{\Cal C}$ is
$\sigma$-invariant.
\end{enumerate}\vspace{-3mm}
\end{Lem}\noindent
In particular, by Corollary\;\ref{cor211}, this applies to any commuting subset
$\Cal C$ of $\Aut_1(G)$ consisting of semisimple transformations.
\begin{proof}
(i) $\Rightarrow$ (ii) follows from \ref{di25}. \
(ii) $\Rightarrow$ (iii) is trivial.
\\
(iii) $\Rightarrow$ (i): We put $H = {[G,G]}^-$, then
$\Cal C\subseteq \Aut_H(G)$, hence by Remark\;\ref{rem28},
$G_{\Cal C},\,G_{\Cal C} \cap H$ satisfy the assumptions of \ref{di21}. If
$G_{\Cal C}$ is
$\sigma$-invariant, then by Corollary\;\ref{cor26}, \,$\sigma|G_{\Cal C}$ is
semisimple and then by Lemma\;\ref{le22}\,(ii)
\,$G_{\Cal C} = (G_{\Cal C} \cap H^0) (G_{\Cal C} \cap G_{\sigma})$\,.
This implies \,$G = N (G_{\Cal C}\cap G_{\sigma})$\,.
Take $x \in G$ such that \,$\sigma = s(x)$\,. Then
$x\in G_{\sigma}$ and $x = zy$ with
$y \in G_{\Cal C} \cap G_{\sigma}\,,\ z \in N$\,. It follows
that $z \in  N\cap G_{\sigma}$\,,
hence by Remark\;\ref{rem28}, $s(x) = s(y)$.
\end{proof}
\begin{Lem}	     \label{le213}
Let $\Cal C$ be a subset of $\Aut(G)$ satisfying \,$G = NG_{\Cal C}$\,.
Then the following statements are equivalent:
\begin{enumerate}
\item[(i)] $\sigma \in s(L_{\Cal C})$\,.
\item[(ii)] $\sigma \in \Cal{S},\ \sigma (t) = t$ for all
$t \in G_{\Cal C}$\,.
\item[(iii)] $\sigma \in \Cal S$, it commutes with $\Cal C$ and
$\sigma (t) = t$ for all $t \in G_{\Cal C}\cap H^0$.
\end{enumerate}\vspace{-3mm}
\end{Lem}
\begin{proof}
Again we put $H = {[G,G]}^-$.
\\
(i) $\Rightarrow$ (ii),(iii): By Lemma\;\ref{le212}, $\sigma$ commutes with
$\Cal C$\,. Take $x \in L_{\Cal C}$ such that $\sigma = s(x)$. Then
$\iota_x$ is unipotent on $G_{\Cal C}$ (see \ref{rem28}) and
Corollary\;\ref{cor26}
(for $L = G_{\Cal C}$) implies that $\sigma = s(x)$ is the identity
on $G_{\Cal C}$\,.
\\
(iii) $\Rightarrow$ (ii): If $\sigma$ commutes with $\Cal C$, then
$G_{\Cal C}$ is $\sigma$-invariant. By Corollary\;\ref{cor26},\linebreak
$\sigma|G_{\Cal C}$ is semisimple and by Lemma\;\ref{le22}\,(ii) \,(with
$G_{\Cal C} \cap H^0$ instead of $H^0$), $\sigma$ is the identity on
$G_{\Cal C}$\,.
\\
(ii) $\Rightarrow$ (i): $G_{\Cal C}$ is $\sigma$-invariant, hence by
Lemma\;\ref{le212} \,$\sigma \in s (G_{\Cal C})$. Take $x \in G_{\Cal C}$ such
that $\sigma = s(x)$. Then $G_{\Cal C}$ is $\iota_x$-invariant and
by assumption, $s(x) = \sigma$ is the identity on $G_{\Cal C}$.
Hence Corollary\;\ref{cor26} implies that $\iota_x$ is unipotent on
$G_{\Cal C}$\,, i.e., $x \in L_{\Cal C}$\,.
\end{proof}
\begin{Lem}	 \label{le214}
{\rm (i)} Let $L$ be a nilpotent subgroup of $G$ and put \,$\Cal C = s(L)$.
\,Then $\Cal C$~is an abelian group contained in $\Cal S$ and we have
\,$L\subseteq L_{\Cal C}$\,.
\\
{\rm (ii)} Let $\Cal C$ be any subset of $\Aut (G)$ satisfying
\,$G = N G_{\Cal C}$ and put \,$\Cal C_1 = s (L_{\Cal C})$. \,Then
\,$s(L_{\Cal C_1}) =\Cal  C_1$\,.
\\
If in addition $\Cal C \subseteq \Cal S$ holds, then
$\Cal C \subseteq \Cal C_1$ \,(in
particular, the elements of $\Cal C$ commute and the generated subgroup is
contained in $\Cal S$),
$G_{\Cal C} = G_{\Cal C_1}\,,\ L_{\Cal C} = L_{\Cal C_1}$\,.
\\
{\rm (iii)} If $\Cal C$ is a commuting subset of $\Cal S$ then $L_{\Cal C}$ is
a maximal nilpotent subgroup of \,$G$.\vspace{-2mm}
\end{Lem}\noindent
Thus the maximal nilpotent subgroups of $G$ are all of the form
$L_{\Cal C}$\,, where $\Cal C$ is an abelian group contained in
$\Cal S$. In particular (for $\Cal C=\{\id\}$), $N=\nil(G)$ is a maximal
nilpotent subgroup of $G$ \,(but this follows also directly from the
definition in Remark\;\ref{rem28}).\vspace{-1mm}
\begin{proof}
(i): $H^0$ being nilpotent and torsion free, we can identify
$(L \cap H^0)_{\R}$ with a subgroup of $H^0$ which is $L$-invariant. Thus
$L' = (L \cap H^0)_{\R}\mspace{1mu}L$ \,is still nilpotent. By
Lemma\;\ref{le27} (with
$\Cal C_1 = \{\iota_x : x \in L'\},\ \Cal C' = \Cal C_1$ in
place of $\Cal C$\,), $s$ is a homomorphism on $L'$. Hence
$s(L') = s(L)$ is always a group and (replacing $L$ by $L'$) it is no
restriction to assume that $L \cap H^0$ is connected. Furthermore
(again by Lemma\;\ref{le27}) \,$\Cal C = s(L)$ is commutative.
\\
If $x\in L$, then $\iota_x$ is unipotent on $L$\,, hence by
Corollary\;\ref{cor26}, $\sigma = s(x)$ is the identity on $L$. This
implies $L\subseteq G_{\Cal C}$\,. $\sigma \in \Cal C$ implies that $\sigma$
is the identity on $G_{\Cal C}$\,, hence by Lemma\;\ref{le213}\;(ii),
$\sigma \in s(L_{\Cal C})$. Take $y \in L_{\Cal C}$ such that
$\sigma = s(x) = s(y)$.
By Remark\;\ref{rem28}, \,$y\,x^{-1}\in N\cap G_{\Cal C}\subseteq L_{\Cal C}$
\,and it follows that $x \in L_{\Cal C}$\,.
\\
(ii): By Lemma\;\ref{le213}\,(ii), $G_{\Cal C} \subseteq G_{\Cal C_1}$\,.
Put $\Cal C_2 = s (L_{\Cal C_1})$\,. If $\sigma \in \Cal C_2$\,, then by
Lemma\;\ref{le213}\,(ii), $\sigma$ is the identity on
$G_{\Cal C_1} \supseteq G_{\Cal C}$\,. Thus
$\sigma \in s(L_{\Cal C}) = \Cal C_1$\,, proving
that $\Cal C_2 \subseteq \Cal C_1$\,. By (i) (with $L_{\Cal C}$
in place of $L$), we have $L_{\Cal C} \subseteq L_{\Cal C_1}$\,,
hence
$\Cal C_1 = s (L_{\Cal C}) \subseteq s (L_{\Cal C_1}) = \Cal C_2$
\,which gives $\Cal C_2 = \Cal C_1$\,.
\\
With the additional assumption $\Cal C \subseteq \Cal S$,
Lemma\;\ref{le213}\,(ii) again implies that any $\sigma \in \Cal C$ belongs
to $s(L_{\Cal C}) = \Cal C_1$\,, i.e., $\Cal C \subseteq \Cal C_1$\,.
In particular, by (i), $\Cal C$ is contained in an abelian subgroup of
$\Cal S$\,. \,$\Cal C \subseteq \Cal C_1$ implies that
$G_{\Cal C_1} \subseteq G_{\Cal C}$ and by Lemma\;\ref{le213}\,(ii),
$G_{\Cal C} \subseteq G_{\Cal C_1}$\,. Thus
$G_{\Cal C}= G_{\Cal C_1}$ and then
$L_{\Cal C} = L_{\Cal C_1}$\,.
\\
(iii): Assume that $L_{\Cal C} \subseteq L$, where $L$ is a nilpotent
subgroup of $G$ and put \,$\Cal C' = s(L)$. Then by (ii) (using
$\Cal C \subseteq \Cal S$ and Corollary\;\ref{cor211}),
$\Cal C \subseteq \Cal C_1 = s(L_{\Cal C}) \subseteq s(L) = \Cal C'$\,.
Consequently $L_{\Cal C'} \subseteq  G_{\Cal C'} \subseteq G_{\Cal C}$
and by
(i), $L_{\Cal C} \subseteq L \subseteq L_{\Cal C'}$\,. Since
$L_{\Cal C}$ is normal in $G_{\Cal C}$ (Remark\;\ref{rem28}), it follows
that $L_{\Cal C}$ is $L$-invariant. Take $x \in L$, then ($L$ being
nilpotent) $\iota_x$ is unipotent on $L_{\Cal C}$\,. Recall that
$L_{\Cal C} \cap H^0 = G_{\Cal C} \cap H^0$ and that $\iota_x$
induces a unipotent transformation on $G/H^0$. Combined, we see that
$\iota_x$ is unipotent on $G_{\Cal C}$\,. Thus $x \in L_{\Cal C}$\,.
This proves that $L = L_{\Cal C}$\,.
\end{proof}
\begin{Pro}	  \label{pro215}
Let $\Cal C_0$ be a commuting subset of $\Cal S$ such that the dimension
of \,$G_{\Cal C_0}\cap H^0$ is minimal (among all such subsets). Put
\,$L= G_{\Cal C_0}$\,, $\Cal C = s(L)$\,. Then the following properties
hold.
\\
{\rm (i)} $L$ is a maximal nilpotent subgroup of $G$ \,(in particular, $L$ is
closed), $L \cap H^0$ is connected, $G = H^0L$\,,
\,$L = L_{\Cal C} = L_{\Cal C_0} = G_{\Cal C}$\,. The dimension of
\,$L \cap H^0$ is minimal among the maximal nilpotent subgroups of $G$\,.
\\
{\rm (ii)} $\Cal C$ is a subgroup of $\Aut_1(G)$ consisting of semisimple
transformations. It is a maximal commuting subset of
$\Cal S\,,\ \Cal C_0 \subseteq \Cal C$\,.
\\
{\rm (iii)} $\beta (x y) = s(y)$ \,(where $x\in H^0,\ y\in L$) defines a
continuous surjective group homomorphism
$\beta\!: G \to \Cal C\,,\ \ker \beta = N\,,\ \beta (x)\,\iota^{-1}_x$ is
unipotent for all $x\in G$\,.
\end{Pro}
\begin{proof}
(i): If $\Cal C_1$ is any commuting set with
\,$\Cal C_0 \subseteq\Cal C_1 \subseteq \Cal S$\,, then minimality of
$\dim (G_{\Cal C_0}\cap H^0)$ implies
$G_{\Cal C_1} \cap H^0 = G_{\Cal C_0}\cap H^0$ \,
(recall that $G_{\Cal C_1} \cap H^0$ is always connected).
Take $x \in G_{\Cal C_0}$ and put $\sigma = s(x)$. By Lemma\;\ref{le212},
$\sigma$ commutes with $\Cal C_0$\,. Put
\,$\Cal C_1 = \Cal C_0\cup \{\sigma\}$. Then $\sigma$ is the identity on
$G_{\Cal C_1}\supseteq G_{\Cal C_0} \cap H^0$. \,$G_{\Cal C_0}$
being invariant under $\sigma$ and $\iota_x$\,, Corollary\;\ref{cor26} and
Lemma\;\ref{le22}\,(ii) imply
that $\sigma$ is the identity on $G_{\Cal C_0}$ and then that
$\iota_x$ is unipotent on $G_{\Cal C_0}$\,. Thus $x \in L_{\Cal C_0}$\,.
This proves that $G_{\Cal C_0} = L_{\Cal C_0}$\,.
By Lemma\;\ref{le214}\,(iii), $L$ is a maximal nilpotent
subgroup of $G$, Corollary\;\ref{cor211} shows that $G = H^0 L$\,.
Lemma\;\ref{le214}\,(ii) implies that
$G_{\Cal C_0} = G_{\Cal C}\,,\ L_{\Cal C_0} = L_{\Cal C}$\,. The
minimality statement about $\dim (L\cap H^0)$
results also from Lemma\;\ref{le214}.
\\[.5mm]
(ii): By Lemma\;\ref{le214}\,(i), $\Cal C$ is an abelian group contained in
$\Cal S$\,, $\Cal C_0 \subseteq \Cal C$ by Lem\-ma\;\ref{le214}\,(ii). If
$\Cal C_1$ is as at the beginning, the reasoning as above gives
\,$L_{\Cal C_1} = G_{\Cal C_1} \subseteq G_{\Cal C_0} =
L_{\Cal C_0}$\,. Then maximality of $L_{\Cal C_1}$
(Lemma\;\ref{le214}\,(iii)\,) implies $L_{\Cal C_1} = L_{\Cal C_0}$\,,
hence
\,$\Cal C_1\subseteq s(L_{\Cal C_1}) = s (L_{\Cal C_0}) = \Cal C$\,.
This proves maximality of $\Cal C$\,.
\\[.1mm plus .3mm]
(iii): By Lemma\;\ref{le27}, $s$ \,is a homomorphism on $L$ and clearly $s(y)$
is the identity for $y \in L \cap H^0 \subseteq N$\,. Since $H^0$ is
normal in $G$ and by (i), $G = H^0 L$\,, it follows easily that $\beta$
is well defined on $G$ and a surjective group homomorphism.
\\[.2mm plus .3mm]
If $\fr h$ denotes the Lie algebra of $H$, then by \ref{di25} and
Lemma\;\ref{le24}, $d(s(y)| H) =\linebreak
(d(\iota_y|H))_s$\,. If $y = \exp(Y)$
\,(where $Y \in \fr g$), then by \cite{Va}\;(2.13.6)\;and\;Th.\,2.13.2,
$d(\iota_y|H) = \exp (\ad_{\fr h}Y)$ \,(where $\ad_{\fr h}$
denotes the adjoint representation of $\fr g$ on~$\fr h$).
Furthermore, uniqueness of the Jordan decomposition implies
\,$\bigl(\,\exp(\ad_{\fr h}Y)\bigr)_s = \linebreak
\exp\bigl(\,(\ad_{\fr h}Y)_s\bigr)$ \,
(recall that $(\ad_{\fr h}Y)_s$ is also the semisimple part in the
additive Jordan decomposition of the operator $\ad_{\fr h}Y$ on $\fr h$.
It follows (using also \cite{Ho}\;Th.\,IX.1.2) that the mapping
$t\mapsto s\bigl(\exp (tY)\bigr)|\,H^0$ \,(from $\R$ to
$\Aut (H^0)$\,) is continuous, hence by \cite{Va}\;Th.\,2.11.2, the
mapping $y \mapsto s(y)|H^0$ from $G^0$ to $\Aut(H^0)$ is continuous.
Since by (i) $L =L_{\Cal C}$\,, we know that $s(y)$ is the identity on
$L$ for $y \in L$\,. As in the proof of Lemma\;\ref{le24}, $G$ is isomorphic
to a quotient of a semidirect product $H^0\rtimes L$ and it follows
easily that $y \mapsto s(y)$ from $L^0$ to $\Aut(G)$ in continuous (hence
the same is true on $L$) and then that $\beta$ is continuous.
\\
By definition (see also Remark\;\ref{rem28}) $z = xy \in\ker\beta$ \,(where
$x\in H^0,\ y \in L$) iff $y \in N$ and this is equivalent to $z \in N$.
Again by Remark\;\ref{rem28}, $\beta(x)\,\iota^{-1}_x$ is unipotent for all
$x\in G$\,.
\vspace{-1mm}\end{proof}
\begin{Cor}    \label{cor215a}	 
Let \,$\Cal C\,,\beta$ be as in Proposition\;\ref{pro215}.
Assume that $\Cal C'$ is a subgroup of $\Aut_H(G)$
with \,$\Cal C \subseteq Z(\Cal C')$. Put
\,$G' = G \rtimes \Cal C'\,,\ N' = \{(x, \beta  (x^{-1})) : x\in G\,\}$.
\\
Then $N'$ is nilpotent, closed and normal in
$G',\ \,G'= N'\,\Cal C',\ \,N'\cap \Cal C' = (e)$. Thus
\,$G'= N'\rtimes\Cal C'$ holds
algebraically and in fact topologically as well. If $\Cal C'$ consists
of semisimple transformations, then \,$N'=\nil(G')$\,.
\end{Cor}
\begin{proof}
Take $\sigma\in\Cal C'$. Since it commutes with $\Cal C$, the group
$G_{\Cal C} = L_{\Cal C}$ is $\sigma$-invariant. By \ref{di21} and \ref{di25},
$s(\sigma(x)) = \sigma\circ s(x)\circ\sigma^{-1}$, hence
$s(\sigma(x)) = s(x)$ \,for $x \in L$\,. This implies 
$\beta \circ \sigma = \beta$ and then a short computation shows that $N'$ is a
normal subgroup of $G'$ (evidently closed).
\\
Commutativity of $\Cal C$ implies $[N',N'] \subseteq \ker\beta = N$. For
$x \in G$, the restriction of $\iota_{(x, \beta (x^{-1}))}$
to $G$ equals $\iota_x \circ \beta (x^{-1})$ which is unipotent and belongs
to $\Aut_1(G)$ (in fact, even to the nilpotent group $\Cal{N}$
of Remark\;\ref{rem28}). Hence $N'$ is nilpotent by \cite{War}\;(9.3).
It follows easily from continuity of $\beta$ that the isomorphism of
$G\rtimes \Cal C'$ and $N\rtimes \Cal C'$ is a homeomorphism.
Assume that $N''\supsetneqq N'$ is a nilpotent
subgroup of $G'$. Then there exists
$\sigma \in N'' \cap  \Cal C'$ with
$\sigma \not=\id$\,. Since $N\subseteq N',\ \sigma$ should
be unipotent on $N$\,. If $\sigma$ is semisimple, this would imply
$\sigma | N = \id$ \,and Lemma\;\ref{le22}\,(ii) would give
$\sigma = \id$\,. The remaining properties are clear.
\end{proof}
\begin{Rem}     \label{rem215b}       
It is easy
to see that $[G',N']\subseteq H$ \,(for $\Cal C'=\Cal C$ even $[G',G']=[G,G]$
holds). If $P'$ denotes the 
group of compact elements of $N'$, then one gets
\,$P' = \{(x, \beta (x^{-1})) : x \in P\}$, where $P$ denotes the group of
compact elements of $L_{\Cal C}$\,.
One has
$P'\subseteq Z(G')$. $P'$ is non-trivial (hence $N'$ is not torsion free)
whenever $G$ has non-trivial
compact (necessarily abelian) subgroups (even when $G$ has no non-trivial
compact normal subgroups). If $H$ is connected, one can show (similarly
as in \cite{Wan}\;(9.2)\,) that there exists a closed torsion free subgroup
$N''$ of $N'$ with $N''\supseteq H$ and $N'=N''P'$\,. Then
$G'=N''\rtimes(P'\times\Cal C')$  \;($P'$ acting trivially on $N''$;
but in general the complementary group $N''$ is not unique).\\
Take for example, $G=\C\rtimes\T$
with $\T=\R/\Z\,,\,t\circ z=e^{2\pi it}z$ for
$t\in\T\,,\,z\in\C\,,\,\Cal C=\Cal C'=\iota(\T)\cong\T$\,. Then
$N=\C\,,\,L_{\Cal C}=\T=P\cong P'\,,\,N'\cong\C\times\T$\,.
\\
If $\Cal C'\subseteq\Aut_1(G)$ is
locally compact, abelian, $\Cal C'\supseteq\Cal C$\,, then $G',H$ satisfy
again the assumptions of \ref{di21} and any semisimple $\sigma\in\Cal C'$
defines a semisimple automorphism of $G'$. 
\end{Rem}
\begin{Lem}	   \label{le216}   
If $\sigma \in\Aut_1(G)$ is semisimple, $x \in G$ and
\,$\sigma(x)\,x^{-1} \in Z(G)$ holds, then
\,$x \in (Z(G) \cap H^0)\,G_{\sigma}$\,.
Under the additional assumption $\sigma \in \Cal S$, we get
$x \in G_{\sigma}$, i.e., $\sigma(x) = x$\,.\vspace{-2.5mm}
\end{Lem}\noindent
$Z(G)$ denotes the centre of $G$. Thus for $\sigma \in\Cal S$, \,
$x \in G_{\sigma}$ is equivalent to \,
$[\iota_x, \sigma] = \id$\,.
\begin{proof}
By Lemma\;\ref{le22}\,(ii), $x = vy$ with $v \in H^0$, $y \in G_{\sigma}$\,.
Then $\sigma (x) x^{-1} = \sigma (v) v^{-1}$ and the first statement
follows from \cite{Wan}\;(5.6) \,(with $N_1 = Z(G)\cap H^0$\,). If
$\sigma \in \Cal S$, then by Corollary\;\ref{cor26}, $\sigma$ is the identity
on $Z(G)$, i.e., $Z(G)\subseteq G_{\sigma}$\,.
\end{proof}
\begin{Lem}	\label{le217}	
Assume that $\sigma \in \Aut(H^0)$ is semisimple,
$\theta \in \Aut(H^0)$, $\theta \circ \sigma^{-1}$ unipotent,
\,$[\theta, \sigma] \in\iota (H^0)$. \,Then
\,$H^0 = (\ad \theta)^2 (H^0)\,H^0_{\sigma}$\,.
\end{Lem}
\begin{proof}[Proof (compare \cite{Wan}\;(5.2))]
First we treat the special case
$w \in Z(H^0)$. We have $Z(H^0)\cong \R^n$ \,(written additively).
$\theta|\,Z(H^0)$ is given by a matrix $A$\,, \,$\sigma|\,Z(H^0)$ by a
semisimple matrix $A_1$\,, where $A, A_1 \in \GL(\R^n)$. Then \,
$H^0_{\sigma}\cap Z(H^0) \cong \ker (A_1 -I)$, \,
$\ad \theta\,(Z (H^0)) \cong \im (A- I)$. Since \,$[\theta,\sigma]$ is the
identity on $Z(H^0)$, it follows 
that $A, A_1$ commute, hence $A_1$ is the semisimple part of $A$. In
particular,\linebreak
$\R^n = \im (A_1 - I) \oplus \ker (A_1 - I),\ \,\im (A_1- I)$
is $A$-invariant and $A - I$ is invertible on \,$\im (A_1 - I)$.
Thus \,$\im (A_1 - I) \subseteq \im ((A - I)^2)$\,, proving
\,$w\in(\ad \theta)^2 (H^0)\,H^0_{\sigma}$\,.
\\[1.7mm]
In the general case, we use induction on \,$\dim H^0$. The special case
covers \linebreak
$\dim H^0\!=\!1$. For the induction step, we apply the
hypothesis to \,$H^0/Z (H^0)$. Thus, given $w\in H^0$, there exist
$u_0 \in (\ad\theta)^2 (H^0),\ v_0 \in H^0$ such that
$\sigma (v_0) v^{-1}_0\in Z (H^0)$ and $w = u_0 v_0 z_0$
for some $z_0 \in Z(H^0)$.
By Lemma\;\ref{le216}, $v_0 = z_1 v_1$ for some\linebreak
$z_1 \in Z (H^0),\ v_1 \in H^0_{\sigma}$\,. From the special case above,
we get
$u_2\in (\ad \theta)^2 (Z(H^0)),\linebreak v_2 \in H^0_{\sigma} \cap Z(H^0)$
such that $z_1 z_0 = u_2 v_2$\,. Then
$u = u_0 u_2\,,\ v = v_1v_2$ will satisfy our requirements.
\end{proof}
\begin{Lem}	    \label{le218}   
Let $\sigma \in \Aut_1(G)$ be semisimple and $x \in G$ such that
\,$\sigma\circ\iota^{-1}_x$ is unipotent. Then there exists
$u\in [G,H^0]$ such that \,$u x u^{-1} \in G_{\sigma}$\,.
\end{Lem}
\begin{proof}
By Lemma\;\ref{le22}\,(ii), $G = H^0 G_{\sigma}$\,. Write $x = y z$ with
$y \in H^0, z \in G_{\sigma}$\,. Put $\theta = \iota_x$\,. It follows that
\,$[\sigma, \theta] = \iota_{\sigma (x) x^{-1}} = \iota_{\sigma (y)y^{-1}}
\in \iota (H^0)$. By Lemma\;\ref{le217}, there exist
\,$u \in \ad \theta(H^0)\subseteq [G, H^0],\ v \in G_{\sigma} \cap H^0 =
H^0_{\sigma}$
\,\,such that \,$y = \ad \theta (u)\,v = [u, x]^{-1} v$\,. Then
\,$u xu^{-1} = [u, x] x = v z \in G_{\sigma}$\,. 
\end{proof}
\begin{Lem}	 \label{le219}	  
For \,$\beta, \Cal C$ as in Proposition\;\ref{pro215}, 
$\sigma \in \Cal C,\ x \in G$, the following statements are equivalent:
\begin{minipage}[t]{5cm}
\begin{enumerate}
\item[(i)] $s(x) = \sigma$.\vspace{1mm}
\item[(ii)] $x \in G_{\sigma}\,,\ \beta (x) = \sigma$\,.
\end{enumerate}
\end{minipage}\vspace{-1mm}
\end{Lem}
\begin{proof}
By Proposition\;\ref{pro215}, $\sigma = s(y)$ for some
\,$y \in L \subseteq G_{\sigma}$\,. By Remark\;\ref{rem28}, $s(x) = \sigma$
\,iff \,$x \in (N \cap G_{\sigma})\,y$\,. Since by
Proposition\;\ref{pro215}\,(iii) \,$\ker\beta = N$, this proves our claim.
\end{proof}\noindent
For $\alpha \in \Aut(G),\ u \in G$, we write
$u \alpha u^{-1} = \iota_u\circ\alpha\circ \iota_{u^{-1}}$.
\begin{Pro}	  \label{pro220}     
Let $\Cal C$ be the subgroup of $\Cal S$ constructed in
Proposition\;\ref{pro215}.
If $\Cal C_1$ is any commuting subset of $\Cal S$, there
exists $u \in [G,H^0]$ such that \,\,$u \Cal C_1 u^{-1} \subseteq \Cal C$\,.
Putting \,$\Cal C_2 = \Cal{C} \cap \Cal C_1$\,, one can take
\,$u \in [G_{\Cal C_2}, H^0 \cap G_{\Cal C_2}]$\,.
\\
In particular,
\,$\Cal S\, =\, \bigcup\, \{u \Cal C u^{-1} : u \in [G,H^0]\,\}$.
\end{Pro}
\begin{proof}
First, we assume that $\Cal C_1$ is finite. If
$\Cal C_2 =\Cal C_1$\,, there is nothing to prove. So, take
$\sigma_1 \in \Cal C_1 \setminus  \Cal C_2$ and $x_1 \in G$
with $\sigma_1=s(x_1)$ \,(which entails $x_1 \in G_{{\sigma}_1}$).
Since $\sigma_1$ commutes with $\Cal C_2\,,\ G_{{\sigma}_1}$ is
invariant under $\Cal C_2$ and
the restrictions of the transformations are semisimple by
Corollary\;\ref{cor26}. Hence by Corollary\;\ref{cor211},
\;$G_{\sigma_1} =
(H^0\cap G_{\sigma_1})\,(G_{\Cal C_2} \cap G_{\sigma_1})$ \
(if $\Cal C_2 = \emptyset$, we put $G_{\Cal C_2} = G$). By
Remark\;\ref{rem28}, it follows that we may assume that
$x_1 \in G_{\Cal C_2}$\,. Consider $\beta $ as in
Proposition\;\ref{pro215} and put $\sigma = \beta  (x_1)$. Then 
$\sigma \in \Cal{C}$, hence $G_{\Cal C_2}$ is $\sigma$-invariant. Applying
Lemma\;\ref{le218} \linebreak to $G_{\Cal C_2}$
and the restriction of $\sigma$ \,(use also
Proposition\;\ref{pro215}\,(iii)\,), there exists \linebreak
$u \in [G_{\Cal C_2}, H^0 \cap G_{\Cal C_2}]$ such that
$u x_1 u^{-1} \in G_{\sigma}$\,. Then $u \in [G, H^0]$
and by Lemma\;\ref{le219}, $u \sigma_1 u^{-1} = s(u x_1 u^{-1}) =\sigma$\,.
Since $u \in G_{\Cal C_2}$, we have 
$u \sigma' u^{-1} = \sigma'$ for $\sigma' \in \Cal C_2$\,. Thus
$u \Cal C_1 u^{-1} \cap \Cal C$ strictly contains $\Cal C_2$
and, repeating this argument, we can reach our goal after finitely many steps.
\\
In the general case, there exists by Corollary\;\ref{cor211} a finite subset
$\Cal C_1'$ of $\Cal C_1$ such that
$G_{\Cal C_1'}= G_{\Cal C_1}$\,.
By the special case treated above, we may assume that
$\Cal C_1' \subseteq \Cal C$\,. Then
$G_{\Cal C}\subseteq  G_{\Cal C_1}$\,.
Take $\sigma \in \Cal C_1$\,, then $\sigma$ is the identity on
$G_{\Cal C_1}$\,, hence by Lemma\;\ref{le213}, 
$\sigma \in s(L_{\Cal C}) = \Cal C$ \;(Proposition\;\ref{pro215}\,(i)\,). Thus
$\Cal C_1 \subseteq \Cal C$\,.
\end{proof}

\begin{Rems}	  \label{rem221}      
(a) \ In Lemma\;\ref{le212}, the implication (i)$\Rightarrow$(ii) holds
for general subsets $\Cal C$ of $\Aut(G)$ (same proof). Furthermore, 
by Corollary\;\ref{cor29}, the elements of $s(G_{\Cal C})\cup \Cal C$ always
commute with those of $s(L_{\Cal C})$.
\vskip1mm
\item[(b)] With some further arguments the following conditions give other
characterizations of the elements $\sigma \in s(L_{\Cal C})$ \;
($\Cal C$ as in Lemma\;\ref{le213}):\vspace{-1.5mm}
\begin{enumerate}
\item[(iv)] $\sigma \in \Cal S$, \,it commutes with 
$s(L_{\Cal C})$ \ and \
$\sigma (t) = t \text{ \,for all } t \in G_{\Cal C} \cap H^0$.
\item[(v)] $\sigma \in \Cal S,\ \sigma (t) = t$ for all
$t \in L_{\Cal C}$\,.
\end{enumerate}\vspace{-1.5mm}
In particular, it follows that $s(L_{\Cal C})$ is a maximal commuting
subset (group) in 
$\{\sigma \in \Cal S : \,\sigma|\,G_{\Cal C} \cap H^0 = \id\,\}$.
\item[(c)] In Lemma\;\ref{le213}, the implications
(i)$\Rightarrow$(ii), \,(i)$\Rightarrow$(iii)$\Rightarrow$(ii) hold for
arbitrary subsets $\Cal C$ of $\Aut_H (G)$ \,(same proof).
\item[(d)] $s(G_{\Cal C})$ is not a group, unless
$G_{\Cal C} = L_{\Cal C}$\,,
i.e. $G_{\Cal C}$ is nilpotent ($\Cal C$ any subset
of $\Aut_H(G)$\,). Indeed, assume that $s(G_{\Cal C})$ is a group, take
$\sigma \in s (G_{\Cal C})$. If $x \in G_{\Cal C}$\,,\linebreak
then by \ref{di25},
$\iota_x \circ \sigma \circ \iota_{x^{-1}} \in s(G_{\Cal C})$, hence 
$\iota_{\sigma (x) x^{-1}} = [\iota_x, \sigma] \in s (G_{\Cal C})$. Since
\linebreak
$\sigma (x) x^{-1}\in H^0,\ \iota_{\sigma (x) x^{-1}}$
is unipotent, hence it must be the identity. Consequently
$\sigma (x) x^{-1} \in Z(G)$ and by Lemma\;\ref{le216}, we get 
$\sigma (x) = x \text{ for all } x \in G_{\Cal C}$. Then
Lemma\;\ref{le213}\,(ii) implies $\sigma \in s(L_{\Cal C})$, i.e., 
$s(G_{\Cal C}) = s(L_{\Cal C})$. Now take $x \in G_{\Cal C}$, then by
Remark\;\ref{rem28}, there exists $y \in L_{\Cal C}$ such that
$yx^{-1} \in N \cap G_{\Cal C} \subseteq L_{\Cal C}$, hence
$x\in L_{\Cal C}$\,.

As a special case, if $\Cal C_0$ is a commuting subset of $\Cal S$ and
\,$\dim (G_{\Cal C_0} \cap H^0)$ is not minimal (compare
Proposition\;\ref{pro215}), then $G_{\Cal C_0} \neq L_{\Cal C_0}$\,, in
particular:  $s(G_{\Cal C_0})$ is not a group
\;(if $\Cal C$ is a maximal commuting set with
$\Cal C_0 \subseteq \Cal C \subseteq \Cal S$\,, then by
Proposition\;\ref{pro220},
$G_{\Cal C_0} \cap H^0 \neq G_{\Cal C} \cap H^0$ and by
Proposition\;\ref{pro215},
$L_{\Cal C} = G_{\Cal C} \subseteq G_{\Cal C_0}$\,, but by 
Lemma\;\ref{le214}\,(iii), $L_{\Cal C}$ is not strictly contained in
$L_{\Cal C_0}$).
\item[(e)] In Lemma\;\ref{le217}, the assumption
\,$[\theta,\sigma]\in\iota(H^0)$
\,can be replaced by assuming the existence of a normal series
\,$(e) = H_0 \vartriangleleft H_1\dots\vartriangleleft H_k = H^0$, where
$H_i$ are closed, connected and invariant under 
$\theta, \sigma$ and $[\theta, \sigma]$ induces the identity on
$H_i/H_{i-1}\linebreak  (i = 1,\dots,k)$. The same argument shows that
$H^0 = (\ad\,\sigma)(H^0)\,H_{\sigma}^0$ \,(for any semisimple automorphism
$\sigma$ of $H^0$), compare \cite{Lo3}\;L.\,5.4.\vspace{0mm plus.2mm}
\item[(f)] For
$G = \R \times \T,\ H = \R \times(0),\ \sigma (x,y) = (x,y + x)$,
one has $G_{\sigma} = \Z \times \T$, thus 
$G = H^0 G_{\sigma}$ but $\sigma \notin\Aut_H(G)$ and $G_{\sigma} \cap H^0$
is not connected (compare Remark\;\ref{rem28}). Alternatively,
one could take $H_1 = (e)$. Then $\sigma \in\Aut_{H_1} (G)$ but
$G \not= H_1^0 G_{\sigma}$\,.\vspace{0mm plus.4mm}
\item[(g)] For $L = L_{\Cal C}$ with $\Cal C$ as in Proposition\;\ref{pro215}
one can show that $N_G(L)=L$\,, i.e. $L$ is ``self-normalizing". If
$G$ is a connected solvable Lie group with Lie algebra $\fr g$\,, then $L$
is connected. If $\fr l$ denotes the Lie algebra of $L$\, then $\fr l$ is a
Cartan subalgebra of $\fr g$ and conversely. \cite{Br}\;3.1 gives an
explicit construction of the nil-shadow of $G$ \,(based on
\cite{DTR}\;Sec.\,3). His mapping $T$ coincides
with $\beta$ of our Proposition\;\ref{pro215} and his multiplication $*$
corresponds to the multiplication on $N'$ arising in our
Corollary\;\ref{cor215a}.
\item[(h)] The Example before Corollary\,\ref{cor26} shows that Jordan
decomposition does not always work for automorphisms of $\Z^3$ and by
duality one also gets counter-examples for $\T^3$. This limits the
possibility to weaken the assumptions of \ref{di21}. If $G$ is a
l.c. group such that $[G, G]$ and $G/G^0$ are nilpotent, let $K$ be the
group of compact elements in
${[G, G]}^-$ \,(in the compactly generated case, this is just the maximal
compact normal subgroup). If $K^0$ is central in $G$ (or more generally, if
$\iota_x|K^0$ is unipotent for all $x\in G$\,) one can extend most of
the results of this section (extending similarly the definition of
$\Aut_1(G)$\,). This applies in particular when $G$ is any connected (but not
necessarily simply connected) solvable l.c. group (then, by Iwasawa's theorem,
every compact normal subgroup of $G$ is contained in the centre,
\cite{HM}\;Th.\,9.82).
We will not make use of this generalization, but see also Remark\;\ref{rem37}
and Section\;4.
\vspace{-2mm plus .2mm}\end{Rems}
\begin{Ex}	 \label{ex222}		
Assume that $G_0\,, H_0$ satisfy the assumptions of \ref{di21} and put
\linebreak
$G = G_0 \times G_0\,,\ H = H_0\times H_0\,,\ \sigma (x,y) = (y,x)$.
Then $\sigma \in \Aut_H(G),\ d \sigma$ is semisimple, but
if $G_0 \neq H_0$\,, then $\sigma \notin \Aut_1(G)$. We have 
$G_{\sigma} = \{(x,x) : x\in G_0\},\linebreak
L_{\sigma} = \{(x,x) : x \in N_0\}$
\,(where $N_0 = \nil (G_0)$\,). Hence if 
$G_0 \neq H_0$\,, then $G \neq H G_{\sigma}$\,, i.e.,
Corollary\;\ref{cor211} does not hold in this case. $L_{\sigma}$ is not
maximal nilpotent (since it is strictly contained in $N_0 \times N_0$), i.e.,
Lemma\;\ref{le214}\,(iii) does not hold. If $\Cal C_0$~is a
subset of $\Aut_1(G_0)$ as in Proposition\;\ref{pro215} and for 
$\tau\in\Aut(G_0),\ \tilde\tau(x,y) = (\tau(x),y),\
\widetilde{\Cal C} = \{\sigma\}\cup\{\tilde\tau : \tau \in \Cal C_0\}$,
then $L_{\widetilde{\Cal C}} = \{(x,x) : x\in L_{\Cal C_0}\}$ and for
${\widetilde{\Cal C}}_1 = s(L_{\widetilde{\Cal C}})$, we get
$L_{\widetilde{\Cal C}_1} = \{(x,y) : x, y \in L_{\Cal C_0}\}$, thus
$\widetilde{\Cal C}_1$ is strictly contained in 
$s (L_{\widetilde{\Cal C}_1})$, i.e., Lemma\;\ref{le214}\,(ii) does not hold.
\end{Ex}
\smallskip
\section{Proofs of Theorem\;\ref{th1}\;and\;\ref{th2}} 
\vspace{1.5mm plus .3mm}
\begin{proof}[{\bf Proof of Theorem\;\ref{th1}}]
Let $R$ be the (non-connected) radical, $N$ the nilradical of~$G$
\,(\cite{Lo2}\;Prop.\,3). By \cite{Lo2}\;Prop.\,5, there exists a closed
subgroup $G_1$ of $R$ with finite index and such that
\,$[G_1,G_1] \subseteq N\subseteq G_1$\,. Since $R/R^0$ is a discrete,
finitely generated
group of polynomial growth, it has a nilpotent subgroup of finite index.
Hence, we can assume in addition that $G_1/R^0$ is 
nilpotent and (by easy arguments as in \cite{Lo2}) that $G_1$ is normal in
$G$. It follows that $G_1, N$ satisfy the assumptions of 
\ref{di21}. Choose $\Cal C$ as in Proposition\;\ref{pro215} and put
\,$L_1 = L_{\Cal C} = (G_1)_{\Cal C}$\,. Then by Proposition\;\ref{pro215},
$L_1$~is nilpotent and $G_1 = N^0 L_1$\,. Let $L$ be the normalizer
of $L_1$ in $G$. Proposition~\ref{pro220}
implies \,$G = [G_1,N^0]\,L \subseteq N^0L$\,.
\\[.5mm]
We claim that $L \cap G_1 = L_1$\,. Put
\,$L' = L \cap G_1\,,\;L'' = L \cap N^0$. Then $L' = L_1 L '',$ $L_1$ and
$L''$ are nilpotent normal subgroups of $L$\,, hence $L'$ is nilpotent
(\cite{Ra}\;L.\,4.7). Since $L_1$ is maximal (Proposition\;\ref{pro215}), we
conclude that $L_1 = L'$, proving our claim. $L_1$ being nilpotent and
normal in
$L$\,, it follows that $\nil(L)\supseteq L_1$\,. Since
$L G_1 = G$\,, we have
\,$L/L_1 = L/(L \cap G_1) \cong L G_1/G_1 = G/G_1$ and this is compact
by (\cite{Lo2}\;Prop.\,4), finishing our proof.\vspace{-2mm}
\end{proof}
\begin{Rem}      \label{rem31}
The argument shows that in fact \,$G = N^0L$\,. The same proof works if $G$ is
a generalized $\overline{FC}$-group without non-trivial
compact normal subgroups under the additional assumption that $G/G^0$ has
polynomial growth \;(by the standard properties of 
\cite{Gu}, this assumption is equivalent to $R/R^0$ having polynomial
growth -- recall that $G^0/R^0$ is compact). In particular, 
the additional assumption is satisfied, if $N$ is connected
(by \cite{Lo2}\;Prop.\,5).

With some further efforts, it can be shown that Theorem\;\ref{th1}
\,(with $G = N^0L$)
is valid for arbitrary compactly generated Lie groups $G$ of polynomial
growth (if\linebreak
$P$~denotes the maximal compact normal subgroup of $N$ and $P^0$
is central in $G$\,, things are easier, using the generalizations mentioned
in Remark\,\ref{rem221}\,(h)\,).
However, it does not hold for arbitrary generalized $\overline{FC}$-groups
(see Example\;\ref{ex32} below). If $G$~is a generalized
$\overline{FC}$-group and a Lie group, one can show the existence of a
closed subgroup $L$ such that $L/\nil(L)$ \,is compact and 
$N L$ is an open subgroup of finite index in $G$ \;(in the discrete case,
i.e., $G$ is a finite extension of a polycyclic group, this is
\cite{Se}\;Cor.\,2,\;p.\,48, where $\nil(L)$ is called an almost-supplement
for $\nil(G)$\,).\vspace{-1mm}
\end{Rem}
\begin{Ex} \label{ex32}
The conclusion of Theorem\;\ref{th1} does not hold in general for discrete
torsion free polycyclic groups (in particular not for arbitrary
generalized $\overline{FC}$-groups). Take $A = \Z^n$, let $\alpha ,\beta$ be
two commuting automorphisms of $A$ such that 
\,$\im(\alpha - \id) + \im(\beta - \id) \neq A$ \,and choose $v_0 \in A$ not
belonging to the left side. We consider 
\,$G =(A\rtimes \Z)\rtimes \Z$ with the first action defined by $\alpha$ and
for the second one, the ``affine'' action arising from
$\beta$ on $A$ and $1\circ (0,1) = (v_0,1)$. Altogether,
$G \cong \{(v, k, l) : v \in A,\ k, l \in \Z\,\}$ and for $l, k' > 0$, the
multiplication is 
\,$(v, k, l)\,(v', k', l') = \Bigl(v + \alpha^k\bigl(\beta ^l (v') +
(\beta^{l-1} +\dots +\id)\circ(\alpha^{k'-1} +\dots +\id)\,(v_0)\bigr),
\linebreak k + k', l + l'\Bigr)$.
If for $(k, l) \neq (0,0)\ \;\alpha^k\circ\beta^l-\id$ \,is always injective,
it is easy to see that $N = A$ and any nilpotent subgroup
$B$ of $G$ with $B \nsubseteq A$ satisfies $B \cap A = (e)$. But the
choice of $v_0$ implies that $G$ cannot be written as 
$A \rtimes B$ for some subgroup $B$ of $G$\,.
\\
Explicitly, for $n = 4$, we have $A \cong\Z^2\otimes \Z^2$ and
$\alpha,\beta$ can be found as follows: 
$\alpha = \alpha_0\otimes\id,\ \beta = \id\otimes\alpha_0$\,, where
$\alpha_0-\id$ \,is not surjective and the eigenvalues of 
$\alpha_0$ are not roots of unity, e.g., take $\alpha_0$ given by the
matrix 
$\left( \begin{smallmatrix}
3 \ 1 \\
2 \ 1
\end{smallmatrix} \right)$. Here \,$\im(\alpha_0-\id)$ has index $2$ in $\Z$
and it turns out that 
\,$\{(v, k, l) : (k, l) \in 4\Z \times\Z\,\}$ splits, i.e.,
the conclusion of Theorem\;\ref{th1} holds for this subgroup.
\end{Ex}\vspace{1mm}

We add here a further structural property, partially extending
\cite{Lo2}\;Prop.\,6.\vspace{-2mm}
\begin{Pro}  \label{pro33}
Let $G$ be a generalized $\overline{FC}$-group without non-trivial compact
normal subgroups. Then the nilradical $N$ is a maximal
nilpotent subgroup of $G$\,.
\\
Thus, if $x \in G$ and the induced automorphism of $N$ is unipotent, then
$x \in N$. In particular, the centralizer $C_G(N)$ equals
the centre $Z(N)$.\vspace{-2mm}
\end{Pro}
\begin{proof}
We start with the statement on the centralizer $Z_1 = C_G(N)$. Clearly,
\linebreak $Z_1 \cap N = Z (N)$ and $Z_1$ is normal in $G$\,. Let
$R$ be the (non-connected) radical of $G$\,. Then it follows easily from
maximality of $N$ that $Z_1 \cap R \subseteq  N$ (otherwise, 
consider the last non-trivial term of the derived series of the solvable
group $(Z_1 \cap R) N/N$\,). Let $R_1$ be the radical of
$Z_1$\,. It is a characteristic subgroup, hence maximality of $R$ implies
$R_1 \subseteq R$ and it follows that $R_1 = Z(N)$. 
Then by \cite{Lo2}\;Prop.\,4 (applied to the generalized $\overline{FC}$-group
$Z_1$), $Z_1/Z(N)$ is compact. Thus $Z_1$ is a 
$Z$-group in the sense of \cite{GM}, in particular an $FC^-$-group and it
is compactly generated by \cite{Lo2}\;Prop.\,2. Then by 
\cite{GM}\;Th.\,3.20, ${[Z_1,Z_1]}^-$ is compact. Since $G$ has no
non-trivial compact normal subgroups, it follows that $Z_1$ is 
abelian, consequently $Z_1 = R_1 = Z(N)$.
\\
For the general case, assume that $M$ is a nilpotent subgroup of $G$
containing $N$\,. We may assume $M$ to be closed. Then by
\cite{Lo2}\;Prop.\,6, $M/N$ is compact. We use the ascending central series
\,$(e) = N_0 \subseteq \dots\subseteq N_k = N$\,, recall that 
$N_i/N_{i-1}$ is torsion free. For $x \in M$\,, we consider the automorphism
$\varphi_i(x)\in\Aut(N_i/N_{i-1})$ induced by 
$\iota_x\ (i = 1,\dots,k)$. $\varphi_i$ is a continuous homomorphism and
$N \subseteq\ker\varphi_i$\,, hence \,$\im\varphi_i$ is compact. Since
$\iota_x$
is unipotent on $N$, it follows that all $\varphi_i$ are trivial, thus
\,$[M,N_i]\subseteq  N_{i-1}$ for all $i$\,. Let $\Cal H$ be the
set of those continuous automorphisms of $N$ that induce the identity on
$N_i/N_{i-1}$ for $i = 1,\dots, k$\,. Then it is easy to see
that $\Cal H$ is a nilpotent normal subgroup of $\Aut(N)$. For $x\in G$\,,
let $\varphi (x)\in\Aut (N)$ be the restriction of $\iota_x$\,.
Then \,$H = \{x\in G : \varphi(x)\in\Cal H\}$ is a normal subgroup of~$G$
with $M\subseteq H$\,. By the first part of the proof, 
$\ker\varphi = C_G(N) = Z(N)$ and by the definition of $\Cal H$
(take $i = 1$)
\,$N_1 = Z(N) \subseteq Z(H)$. It follows that $H$ is nilpotent and
then maximality gives $H = N$\,, hence $M\subseteq N$\,.\vspace{-2mm}
\end{proof}
\begin{Rem} \label{rem34}
This need not be true when there are non-trivial compact normal subgroups.
Take e.g. a direct product of a compact semsimple group
and a nilpotent group.\vspace{-3mm}
\end{Rem}
\begin{Cor} \label{cor35}
If $G$ is as in Proposition\;\ref{pro33}, then $G_{\R}$ has no non-trivial
compact normal subgroup. $N_{\R}$ is the nilradical of $G_{\R}$\,.
\vspace{-3mm}\end{Cor}
\begin{proof}
We have $G_{\R} = N_{\R}\,G$ \,(see \ref{di21}). Assume that $P$ is a compact
normal subgroup of $G_{\R}$\,. Since $N_{\R}$ is normal in $G_{\R}$ and
torsion free, it follows that $[P,N_{\R}] = (e)$. Take $x\in P$\,,
then $x = uv$ with $u\in N_{\R}\,,\ v \in G$\,. Since
$\iota_v$ coincides with $\iota_{u^{-1}}$ on $N_{\R}$\,, we get that $\iota_v$
is unipotent on $N_{\R}$\,, hence also on $N$\,. Thus 
$v\in N$\,, resulting in \,$P\subseteq P\cap N = (e)$.
\\
It is easy to see that $G_{\R}$ is again a generalized $\overline{FC}$-group.
If $N_1$ denotes its nilradical, then $N_1\supseteq N_{\R}$
and (by maximality) $N_1\cap G = N$, giving $N_1 = N_{\R}$\,.
\vspace{-2mm}\end{proof}
\begin{proof}[{\bf Proof of Theorem\;\ref{th2}}]\vspace{0mm plus .2mm}
\item[(a)] Assume that the nilradical $N$ is connected and that $G/N$ is
compact. Then $G$ is an almost connected Lie-group. Let $K$ be a
maximal compact subgroup. By \cite{Ho}\;Th.\,XV.3.7, we have \,$G = NK$\,.
Since $N$ is torsion free, $N\cap K$ is trivial, thus
\,$G\cong N\rtimes K$\,. Furthermore, $K\cap C_G(N)$ is easily seen to be
normal in $G$, hence it must also be trivial, proving 
faithfulness of the action of $K$. So we may take $\widetilde G = G$ in this
case (in fact, this argument just needs that $N$ is some connected
nilpotent normal subgroup for which $G/N$ is compact, but then it is not hard
to see that necessarily $N =\nil(G)$ holds).
\vspace{-.5mm plus .2mm}
\item[(b)] For the general case, it will be enough (using (a)\,) to show the
existence of a Lie group $\widetilde G$ without non-trivial compact 
normal subgroups, having $G$ as a closed subgroup such that \,$\widetilde G/G$
is compact, $\widetilde N = \nil(\widetilde G)$ is connected and
\,$\widetilde G/\widetilde N$ is compact. We take up the notations from the
Proof of Theorem\;\ref{th1} above. Put \,$K_1 = \overline{\Cal C}$\,. Then
$K_1$ is a compact abelian subgroup of $\Aut(G_1)$ \;(recall that
$G_1 = N^0L_1$ and each $\sigma\in\Cal C$ is the identity on $L_1$\,,
thus it suffices to consider the restrictions to $N^0$; let $\fr n$ be the
Lie algebra of $N^0$, then $\Aut(N^0)\cong\Aut(\fr n)$ --
using \cite{Va}\;Th.\,2.7.5 and \cite{Ho}\;Th.\,IX.1.2; by
\cite{Lo1}\;Th.\,1, for $x \in G_1$\,,
the eigenvalues of $d\iota_x$ have modulus $1$, hence the
same is true for the eigenvalues of $(d\iota_x)_s$ and this equals
$d\iota_{s(x)}$ by \ref{di25}\,). Put \,$G_2 = G_1\rtimes K_1$ and
$N_2 =\nil(G_2)$. By Corollary\;\ref{cor215a}, we have 
\,$G_2 \cong N_2\rtimes K_1$ and \,$G_2/G_1\ (\cong K_1)$ is compact.
\\[0mm plus .1mm]
For $x \in L$, we define
\,$\varphi(x)\,(y,\sigma) = (xyx^{-1},\iota_{x,1}\circ\sigma\circ
\iota_{x,1}^{-1})$, where $\iota_{x,1}$
denotes the restriction of $\iota_x$ to $G_1$ \,(compare \cite{Au}\;p.\,243).
Recall that $L_1\triangleleft L$ and \,$\Cal C = s(L_1)$\,,
thus \,$\iota_{x,1}\circ\Cal C\circ\iota_{x,1}^{-1} = \Cal C$\,.
Easy computations show that \,$\varphi(x)\in\Aut(G_2)$ and that
$\varphi$ is a homomorphism on $L$\,. For \,$x\in G_1\cap L\ \,(= L_1)$,
we get \,$\varphi(x) = \iota_{x,2}$\,, the corresponding inner automorphism
of $G_2$\,. Then for $u \in G_1\,,\ x\in L$\,, we extend the definition by
\,$\varphi(ux) = \iota_{u,2}\circ\varphi (x)$. Again one can 
check that \,$\varphi\!: G\,(= G_1L)\to\Aut(G_2)$\linebreak is well
defined and a continuous homomorphism. For $z\in G$, the 
restriction of $\varphi(z)$ to $G_1$ is $\iota_{z,1}$ and for
$z\in G_1$\,, we have $\varphi(z) = \iota_{z,2}$\,. This allows to 
apply Proposition\;\ref{pro14}, there exists a locally compact group
$G_3$ (in fact a Lie group) 
having $G_2\,,G$ as closed subgroups
with $G_2$ normal, $G_2\cap G = G_1\,,\  G_2\,G = G_3$
\,(consequently, $G_3=GK_1$). Since
$G/G_1$ is compact, it follows that $G_3/G_2$ is
compact and then that $G_3/N_2$ and $G_3/G$ are compact.
In particular,
by \cite{Gu}\;Th.\,I.4, $G_3$ has polynomial growth. Put
$N_3 = \nil(G_3)$\,. Then $G_2\triangleleft G_3$ implies
$N_2\subseteq N_3$\,. Let $P_3$ be the maximal compact normal
subgroup of $G_3$ (\cite{Lo2}\;Prop.\,1) and put \,$G_4 = G_3/P_3$\,. Then
$G\cap P_3$ is trivial and we get an embedding of $G$ into
$G_4$\,. For $N_4 = \nil(G_4)$, we have
$N_4\supseteq N_3P_3/P_3$\,.
Since $N_4$ is torsion free, we can
finish the construction by putting \,$\widetilde N = (N_4)_{\R}\,,\
\widetilde G = (G_4)_{\R} = G_4 (N_4)_{\R}$ \,(see \ref{di21}) which has
the required properties \,(it has no non-trivial compact normal subgroups by
Corollary\;\ref{cor35}; alternatively we could factor once 
more by the maximal compact normal\linebreak subgroup;
see also the comment to Corollary\;\ref{cor48} for further explanations).
\end{proof}
\vspace{-3.5mm plus .3mm}
\begin{Cor} \label{cor36}
Let $G$ be a compactly generated group of polynomial growth without
non-trivial compact normal subgroups. Then $G$ has a faithful
finite dimensional representation, i.e., for appropriate $n > 0$ there exists
an injective continuous homomorphism 
\,$\pi\!: G\to \GL (\R^n)$ such that $\pi(G)$ is closed and
$\pi(N)$
\,(where $N = \nil(G)$) consists of upper triangular unipotent matrices.
\\
Moreover, there exists such a faithful representation $\pi$ having the
additional property
that the eigenvalues of $\pi(x)$ are of modulus $1$ for all $x\in G$\,.
\vspace{-2mm}
\end{Cor}\noindent
In \cite{Ab1} a (real or complex) linear group is called {\it distal}, if
every eigenvalue of its elements has absolute value $1$.\vspace{-1.5mm}
\begin{proof}[Proof {\rm(Compare \cite{Wan}\;Th.\,3)}]
By Theorem\;\ref{th2}, we can assume that $G = \widetilde N\rtimes K$,
where $K$ is
compact, $\widetilde N$ is connected nilpotent torsion free and
the action of $K$ on $\widetilde N$ is faithful. Thus $K$ can be considered as
a closed subgroup of $\Aut(\widetilde N)$. Now we use the Birkhoff
embedding theorem (\cite{Wan}\;p.\,16, \cite{Au}\;p.\,239). It gives a
faithful representation $\pi$ of $\widetilde N$ by upper triangular unipotent
matrices together with a representation of $\Aut(\widetilde N)$ \,(where
$\pi(\sigma)$ is semisimple for $\sigma$ semisimple), combining to
a representation of $\widetilde N\rtimes\Aut(\widetilde N)$. Note that
in particular $\pi(K)$ acts faithfully on $\pi(\widetilde N)$. By compactness,
the (complex) eigenvalues of $\pi(x)$ have modulus 1 for $x\in K$\,. Fixing
$x\in K$\,, $\pi(\langle x\rangle \widetilde N)$ triangulizes (over $\C$) by
\cite{Wan}\;(2.2). Hence the eigenvalues of $\pi(xy)$ have modulus $1$ for
$y\in \widetilde N$\,.
\end{proof}\vspace{-3mm plus .3mm}
\begin{Rem} \label{rem37}
By some additional arguments, one can also prove an analogue of
Theorem\;\ref{th2} for
a generalized $\overline{FC}$-group $G$ without non-trivial compact normal
subgroups. One gets an embedding into some group
$\widetilde G = \widetilde N\rtimes S$ such that $\widetilde N$ is
connected, simply connected, nilpotent and $S$ is an almost connected
$SIN$-group (i.e., by \cite{GM}\;Th.\,2.9, $S\cong V\rtimes K$,
where $V\cong\R^n,\ K$ is compact and $K^0$ acts trivially on $V$\,) such
that the action of $S$ on $\widetilde N$ is semisimple.
$G$ becomes a closed subgroup, but one can no longer expect that
$\widetilde G/G$ is compact.  As in Corollary\;\ref{cor36}, one
gets again faithful finite dimensional representations.
\vspace{-1mm plus .4mm}

In \cite{Wi2}\;Th.\,3 (for discrete groups) an intermediate type of
embeddings is studied. By splitting only the compact part $K$\,, a
generalized $\overline{FC}$-group $G$ without non-trivial compact normal
subgroups can be embedded as a closed, co-compact subgroup into a
semidirect product $\widetilde S\rtimes K$ where $\widetilde S$ is a
connected, simply connected and super-solvable Lie group (i.e.,
$\widetilde S$ has a faithful representation by real triangular matrices),
$K$ compact, acting faithfully on $\widetilde S$\,. \cite{Wi2}\;Cor.\,4
gives for $G$~discrete a uniqueness result, similar to our Theorem\;\ref{th3},
for this type of embedding. This should extend to general $G$\,.
\vspace{-1mm plus .6mm}

Such embeddings are related to the semisimple splittings of \cite{Au}
(see also\linebreak \cite{Se}\;Ch.\,7; by considering closures in
the automorphism group, we arrive at somewhat bigger groups in the
non-discrete case). To point out the differences, note
that in Theorem\;\ref{th2} we get $\widetilde N = \nil(\widetilde G)$ to be
connected which entails that $GK$ can be a proper subset of $\widetilde G$
(not even a subgroup in general) and $G$ need not be normal in
$\widetilde G$. Since we want $K$ to be compact, $G\widetilde N$ will in
general be only a dense subgroup of 
$\widetilde G$. On the other hand, we do not require that $G/N$ is torsion
free. See also Remarks\;\ref{rem411} for further discussion.
\vspace{0mm plus .2mm}

In \cite{To} rather general splitting results are stated. But the handling
of the definitions is not always consistent and the presentation is rather
intransparent. Therefore, we have decided to rely on the earlier version of
Wang \cite{Wan} as a basis of our exposition. We have tried to avoid too much
use of results from algebraic groups (this might also give some shorter
arguments in Section\;4).\vspace{0mm plus .7mm}

With almost the same proof Theorem\;\ref{th2} extends to the case where only
$N$ (instead of $G$) has no non-trivial compact normal subgroup. But if
$C$\,, the maximal compact subgroup of $G$\,, is non-trivial, then for
any embedding as in  Theorem\;\ref{th2} (with $\widetilde N$~simply connected),
$C$ (that embeds into $K$\,, see also Proposition\;\ref{pro47}\,(c)\,) will
act trivially on $\widetilde N$\,. Thus the
action of $K$ on~$\widetilde N$ will no longer to be faithful.
If $G$ is a compactly generated l.c. group of polynomial growth satisfying
the assumptions of \ref{di21}, one can (after first passing to $G_\R=H_\R G$)
use the modification sketched in Remark\;\ref{rem215b} to get an embedding
(as a closed subgroup) into a group $\widetilde N\rtimes K$ where
$\widetilde N$ is simply connected, nilpotent, $K$ compact abelian. This contains
some further examples there $C$ is non-trivial.

More generally, using the
generalizations mentioned in Remark\,\ref{rem221}\,(h)\,), one can extend a
large part of the proof of Theorem\;\ref{th2} (up to $G_3$) to the case
where $G$ is a compactly generated l.c. group of polynomial growth with
maximal compact subgroup $C$ and there exists a closed normal subgroup $H$
such that $G/H$ is compact and $\iota_x|C$ is unipotent for all $x\in H$\,.
But there are examples of (non torsion free) compactly generated nilpotent
Lie groups that cannot be embedded into a connected nilpotent group. Hence
the last step of the argument will fail in general and this produces only
a certain analogue of the groups $G_{an}$ described in
Proposition\,\ref{pro38}.
\vspace{.4mm plus 1.5mm}

If $G$ is any compactly generated Lie group of polynomial
growth, the following properties can be shown to be equivalent:\vspace{-2.5mm}
\begin{enumerate}
\item[(a)] $G$ has a faithful (continuous) finite dimensional representation.
\vspace{.2mm}
\item[(b)] $G$ has a closed normal subgroup $H$ such that $G/H$ is compact
and $H$ has no non-trivial compact normal subgroup.
\item[(c)] $G$ has a closed normal subgroup $H$ such that $G/H$ is compact
and\linebreak ${[G^0H,H]}^{-\,0}$ is torsion free.
\item[(d)] $R$ (the non-connected radical of $G$) has a subgroup $R_1$ of
finite index such that ${[R_1,R_1]}^-$ is nilpotent and torsion free.
\end{enumerate}\vspace{-2.5mm}
Then the group $R_1$ in (d) can be chosen to be $G$-invariant
and the group $H$ in (b) can be found so that
$[R_1, R_1] \subseteq H \subseteq R$\,. There
exists a faithful finite dimensional representation $\pi$ such that $\pi (G)$
is closed and distal. There is also an embedding of $G$ as a closed subgroup
of some $\widetilde G\cong \widetilde N \rtimes K$ such that $K$ is compact,
$\widetilde N$ connected nilpotent torsion free, $\widetilde G/G$ compact.
But the action of $K$ on $\widetilde N$ need not be faithful and $N$ need not
be contained in $\widetilde N$ (recall that in case $\pi (N)$ consists of
unipotent transformations, $N$ must be torsion free).

A classical case where this is satisfied are finitely generated (discrete)
\pagebreak
groups of polynomial growth (and more generally, extensions of
polycyclic groups by finite groups, not necessarily torsion free). Here one
can even get a faithful representation by integer-valued matrices
(\cite{Se}\;Th.\,5,\;p.\,92).\vspace{-.5mm plus 1mm}

For a compactly generated Lie nilpotent group $G$ one gets that
$G$ has a faithful (continuous) finite dimensional representation iff
${[G,G]}^{-\,0}$ (the identity component of the topological commutator group)
is torsion free \,(this extends the
characterization of \cite{Ho}\;Th.\,XVIII.3.2 in the connected case; for
connected $G$ it follows that $[G,G]$ must already be closed, but this need
not be true in the non-connected case).
\vspace{-5.5mm plus .5mm}\end{Rem}
The construction used in the proof of Theorem\;\ref{th2} gives also a smaller
almost nilpotent extension.
\vspace{-2.5mm plus.5mm}
\begin{Pro}  \label{pro38}
Let $G$ be a compactly generated group of polynomial growth without
non-trivial compact normal subgroups.
\\
{\rm (i)} There exists a Lie group $G_{an}$ containing $G$ as a closed
subgroup such that \,$G_{an}/G\,,
G_{an}/N_{an}$ are compact \,(where 
$N_{an} = \nil (G_{an})$) and $G_{an} = \overline{N_{an}G}$\,.\\ 
Furthermore, $G_{an}$ has no non-trivial compact normal subgroup,
$G \cap N_{an}=N,\linebreak\
[N_{an}, N_{an}] \subseteq N$ \,(in particular, $N$ is normal
in $G_{an}$) and there exists a compact connected abelian subgroup $K_1$ of
$G_{an}$ such that $G_{an} = GK_1$ and \,$K_1\cap C_{G_{an}}(G) = \{e\}$.
$N_{an}K_1\cap G$ is a $G_{an}$-invariant subgroup of finite index in $R$
(radical of $G$).
\\[1mm]
{\rm (ii)} If $G$ is almost connected, then $N_{an} = \widetilde N$ is
connected and $G_{an} = \widetilde G$ coincides with the group of
Theorem\;\ref{th2}.\vspace{-2.5mm}
\end{Pro}
\begin{proof}
We take $G_{an} = G_4$\,, as constructed in the proof of Theorem\;\ref{th2}.
Then (i)
follows (replacing $G_1$ by some subgroup of finite index, one can always 
achieve that $K_1$ is connected, see also Corollary\;\ref{cor48} and the
comment there).
\\
If $G$ is almost connected (since it is a Lie group, this means that $G/G^0$ 
is finite), we can take $G_1 = R^0$. Then Corollary\;\ref{cor215a} shows that
$N_2$ is connected. $N_2$~being co-compact in $G_3$\,, the same is true
for its image $N_2 P_3 / P_3$ in $G_4$\,. Since $N_2 P_3 / P_3$
is connected and $N_4$ torsion free, it follows (e.g. \cite{Ra}\;Rem.\,2.6)
that $N_4$ equals $N_2 P_3 / P_3$\,. Hence 
$\widetilde G = G_4$ (alternatively, one could use Corollary\;\ref{cor48};
then Proposition\;\ref{pro44}(a) implies that
$\widetilde N \cap G^0 K_1$
is a co-compact connected subgroup of~$\widetilde N$, hence
$\widetilde N \subseteq  G^0 K_1$). 
\end{proof}
\smallskip
\section{Subgroups of semidirect products} 
\smallskip
As a preparation for the proof of Theorem\;\ref{th3}, we start with two
technical
lemmas and then we collect some properties of subgroups of the semidirect
products that arise in Theorem\;\ref{th2} (in particular, consequences of the
assumption ``\,$\widetilde G/G$ compact"). This allows to give an explicit
description of the decomposition of Theorem\;\ref{th1}
(Corollary\;\ref{cor48}) and to
identify (Proposition\;\ref{pro47}) various constituents that came up in
\cite{Lo2}.

To make the induction arguments easier, we consider now also nilpotent Lie
groups $N$ that are not torsion free. Analogously to Definition\;\ref{def23},
$\sigma \in\Aut(N)$
is called {\it semisimple}, if the corresponding transformation $d\sigma$ of
the Lie algebra $\fr n$ of~$N$ is semisimple and $N = N^0 N_{\sigma}$\,.
If $N'$ is a closed normal
$\sigma$-invariant subgroup of~$N$, it is easy to see that the induced
automorphism on $N/N'$ is again semisimple.

\begin{Lem}	 \label{le41}
Let $N$ be a nilpotent Lie group, $\Cal K$ a connected subgroup of
$\Aut(N)$ consisting of semisimple transformations and let $N_1$ be a
closed connected $\Cal K$-invari\-ant subgroup of $N$. Then for
$\sigma\in\Cal K$, we have
\,$(\ad\sigma)(N) \cap N_1 = (\ad\sigma) (N_1)$.
\end{Lem}
\begin{proof}
Semisimplicity implies \,$(\ad\sigma)(N) = (\ad\sigma)(N^0)$, thus we may
assume\linebreak
$N$~connected. Let $\widetilde N$ be the universal covering group of
$N$, \,$\pi\!: \widetilde N\to N$ the canonical projection,
$\Gamma = \ker \pi$,
\,$\widetilde N_1 = \pi^{-1} (N_1)^0$. Denote by $\widetilde{\Cal K}$
the group of lifted automorphisms $\tilde\sigma$\,. Consider $x \in N$ with
$\sigma (x)\,x^{-1} \in N_1$ and take $\tilde x \in \widetilde N$ with
$\pi (\tilde x) = x$\,. Then
\,$\tilde\sigma (\tilde x)^{-1} \tilde x^{-1} \in \pi^{-1} (N_1) =
\Gamma \widetilde N_1$\,.
Let $\widetilde P$ be the analytic subgroup of $\widetilde N$ generated
by~$\Gamma$ \,(we have $\Gamma \subseteq Z(\widetilde N)$;
if this is written additively, $\widetilde P$ is just the vector subspace
generated by $\Gamma$\,). Then $M = \widetilde P\widetilde N_1$ is an
analytic subgroup of $\widetilde N$,
\,$\tilde\sigma (\tilde x)\,\tilde x^{-1} \in M$.
By \cite{Wan}\;(5.6), we have \,$\tilde x = \tilde y\,z$ with
$\tilde y\in M$, $z\in\widetilde N_{\tilde\sigma}$\,.
Since $\widetilde{\Cal K}\;(\cong\Cal K)$
is connected, it has to be trivial on $\Gamma$,
hence also on $\widetilde P$, thus we can assume that
$\tilde y\in\widetilde N_1$\,. Put $y = \pi(\tilde y)$, then $y\in N_1$\,,
$\sigma(y)\,y^{-1} = \sigma(x)\, x^{-1}$.
\end{proof}

\begin{Cor} \label{cor42}
Let $N$, $\Cal K$ be as above.
\\
{\rm (a)} If $\sigma \in \Cal K$ and $N_1$ is any closed
$\Cal K$-ivariant subgroup, then $\sigma| N_1$ is again semi\-simple.
\\
{\rm (b)} If $\Cal K$ is commutative, then $N = N^0 N_{\Cal K}$\,.
\\
{\rm (c)} If $\Cal K$ is commutative, $N'$ is any closed normal
$\Cal K$-invariant subgroup, then\linebreak
$N_{\Cal K} N'/N'= (N/N')_{\Cal K}$\,.
\vspace{-2mm}
\end{Cor}
\begin{proof}
(a): If $x \in N_1$\,, then $\sigma(x)\,x^{-1} \in N^0_1$ \,($\Cal K$ is
connected), hence by
Lemma\;\ref{le41}, $x = y\,z$  with $y \in N^0_1$\,, $z \in N_{\sigma}$\,.
\\
(b),(c): As in the proof of Corollary\;\ref{cor211}, there is a finite subset
$\Cal K_0$ of $\Cal K$ such that $N_{\Cal K} = N_{\Cal K_0}$\,. 
Then an easy induction argument (see also \cite{Wan}\;(8.8)\,)
proves our claim.\vspace{-2mm}
\end{proof}

\begin{Lem}	 \label{le43}
Let $\Cal G$ be a triangular group of automorphisms of $N = \R^n$ and
assume that $\Cal G$ is nilpotent and that the closure $\Cal K$ of
\,$s(\Cal G)$ \,(the semisimple parts) is connected. Then any closed
$\Cal G$-invariant subgroup $H$ of $N$ is $\Cal K$-invariant.\vspace{-.5mm} 
\end{Lem}
\begin{proof}
By Lemma\;\ref{le27}, $\Cal K$ is commutative and centralizes $\Cal G$\,.
If $\Cal N$ denotes the set of unipotent transformations in $\Cal{GK}$\,, 
then $\Cal N$ is a subgroup.
For $\sigma \in \Cal G$, $s(\sigma)$ is a polynomial in $\sigma$\,, hence
the result follows immediately when $H$ is connected. Thus (passing to
$N/H^0$), we may assume
that $H$ is discrete and that it generates $N$ as a real vector space.
\\[-.2mm]
First, we claim that $N_{\Cal K} \cap H$ is non-trivial. For
$\sigma \in \Cal G$, we consider its restriction to the $\Q$-vector space
$H_{\Q}\ (= \Q H)$. It has a Jordan decomposition (for the base field~$\Q$,
\cite{Bo1}\;VII,\,Th.\,1,\,p.\,42)
and by uniqueness, the real extensions of the components have to coincide with
$\sigma_s\,, \sigma_u$\,. Thus $H_{\Q}$ is $s(\Cal G)$-invariant and
$\Cal N$-invariant. $\Cal N$~is still triangular and it follows easily
that $N_{\Cal N}\cap H_{\Q}$ is non-trivial, hence $N_{\Cal N} \cap H$ is
non-trivial. Since $N_{\Cal N}\cap H$ is $\Cal G$-invariant and
$\Cal N$ acts trivially on this group, it is also $\Cal K$-invariant.
$\Cal K$ being connected, $N_{\Cal N} \cap H$ discrete, it follows that
$\Cal K$ acts trivially on $N_{\Cal N}\cap H$, i.e.,
$N_{\Cal N} \cap H \subseteq N_{\Cal K}$\,.
\\
As observed in the proof of Corollary\;\ref{cor42}, there is a finite subset
$\Cal K_0$ of $\Cal K$ such that $N_{\Cal K_0} = N_{\Cal K}$\,. For 
$\sigma \in \Cal K$\,, the projection to $N_{\sigma}$ obtained from the
primary
decomposition is a rational polynomial in $\sigma$ (\cite{Va}\;Th.\,3.1.1).
Combined, this gives a projection $p$ to $N_{\Cal K}$ which is a rational
polynomial in elements of $\Cal K$\,. Put $M = (\id-p) (N)$\,, defining 
a complementary subspace for $N_{\Cal K}$\,, invariant under
$\Cal G$ and $\Cal K$\,. Since $H_{\Q}$ is $p$-invariant, it follows that 
$(\id-p)(H_{\Q})\subseteq H_{\Q}$\,, hence \,$H_1 = (\id-p)(H_{\Q}) \cap H$
generates $M$ as a real vector space. Thus, if $M$ would be non-trivial, the
same would be true for $H_1$ and the argument above would imply that
$N_{\Cal K} \cap H_1$ is non-trivial which is impossible. It follows that
$N_{\Cal K} = N$\,, finishing the proof.\vspace{-.3mm}
\end{proof}

\begin{Pro}  \label{pro44}
Let $\widetilde N$ be a compactly generated nilpotent Lie group, $K$ an
abelian connected locally compact group with a continuous semisimple action
on~$\widetilde N$. Let $G$ be a subgroup of
\,$\widetilde G = \widetilde N \rtimes K$ such that
$N = G \cap \widetilde N$ is closed in $\widetilde N$\,, $\widetilde G/G$ is
compact and \,$\widetilde N G$ dense in $\widetilde G$\,. Then the following
properties hold.
\\
{\rm (a)} Put
\,$\widetilde M =
\{x \in \widetilde N\! :\,k \circ x = x \text{ \,for all } k \in K\}$\,,
$\widetilde L = \widetilde M \times K$\,, $L = G \cap \widetilde L$\,.
\\
Then $G$ and $N$ are $K$-invariant, \;$G = N^0L$\,,
\,$\widetilde G = N^0\widetilde L$\,,
\,$\widetilde N = N^0\widetilde M$ \;and
\,$\widetilde N/(\widetilde N \cap GK)\;
(\cong \widetilde M/(\widetilde M\cap GK)\,)$ is compact.
\\
{\rm (b)} If \,$Z(\widetilde N)\mspace{2mu}GK$ is dense in $\widetilde G$,
then
\,$[\widetilde G,\widetilde G] \subseteq N$ \;(in particular, \,$G$ is normal
in~$\widetilde G$\,, $\widetilde M/(\widetilde M\cap N)$ is abelian).
\\[-.3mm]
{\rm (c)} If $H$ is a connected, closed $G$-invariant subgroup of
$\widetilde N$ and $\widetilde N$ is torsion free, then $H$ is normal in
\pagebreak $\widetilde G$\,.
\\
{\rm (d)} If $H$ is any closed $G$-invariant subgroup of $\widetilde N$\,,
then $H$ is $K$-invariant \,(in particular,
if \,$Z(\widetilde N)\mspace{2mu}GK$ is
dense in $\widetilde G$\,, then $H$ is normal in $\widetilde G$\,). 
\vspace{-2mm}
\end{Pro}
\noindent
By a semisimple action on $\widetilde N$ (denoted by $\circ$), we mean that
each transformation shall be semisimple (see also Lemma\;\ref{le46}).
We do not require $G$ to be
closed. By \cite{HR}\;(5.24\,b), compactness of $\widetilde G/G$ is equivalent
to compactness of the Hausdorff space $\widetilde G/\overline G$\,.
\vspace{-2mm}
\begin{proof}
$(\alpha)$ \;First, we assume that $\widetilde N$ is abelian and that
\,$G \cap\widetilde N$ and $\widetilde M$ are trivial. We claim that this
implies $\widetilde N$ to 
be trivial. Connectedness of $K$ easily implies (use the dual action) that any
continuous action on a compact abelian group is trivial. Hence $\widetilde N$
must be torsion free and, replacing $\widetilde N$ by $\widetilde N_{\R}$\,,
we can assume that $\widetilde N$ is connected, i.e.,
$\widetilde N\cong \R^n$, written additively.
Let \,$K' = \widetilde N G\cap K$ be the projection of $G$ to~$K$\;%
$(\cong \widetilde G/\widetilde N)$. Triviality of $G \cap \widetilde N$
gives a mapping \,$c\!: K' \to \widetilde N$
such that $G = \{\,c(x)\,x : x \in K'\}$ and $c$ is a crossed homomorphism,
i.e., $c (xy) = c(x) + x \circ c(y)$ for $x,y \in K'$. Then
commutativity of
$K$ leads to \,$x \circ c(y) - c (y) = y \circ c(x) - c(x)$. Triviality of
$\widetilde M$ implies that for each $v \in \widetilde N\setminus \{0\}$
there exists $x \in K$ such
that $x \circ v \neq v$\,. Assume that $\widetilde N$ is non-trivial, i.e.,
$n > 0$\,. Considering the root space decomposition for the extended action of
$K$ on $\C^n$, it is easy to see (using that $K$ is connected) that there
exists $x_0 \in K$ for which all roots are different from~1, i.e., 
$\alpha (v) = x_0 \circ v - v$ is an isomorphism on $\widetilde{N}$
(compare \cite{Bo2}\;p.\,28). By assumption, $K'$ is a dense subgroup of $K$,
hence we can assume that $x_0 \in K'$. Putting
\,$v_0 = -\alpha^{-1}\bigl(c(x_0)\bigr)$, it follows that
\,$c(x) = v_0 - x \circ v_0 \text{ for all } x \in K'$. This would imply
that \,$G = v_0 K' (-v_0)$ is conjugate to $K'$, contradicting compactness
of \,$\widetilde G/G$.
\vspace{-1.5mm plus .2mm}
\item[$(\beta)$] \;Now we assume just that \,$\widetilde{N}\cap G$ is
trivial and claim
that $\widetilde M = \widetilde N$ and that $GK$ is abelian. Observe that if
$N'$ is a closed $\widetilde G$-invariant subgroup of $\widetilde N$, we can
consider \,$\widetilde G_1 = (\widetilde N/N') \rtimes K$ and then the image
$G_1$ of $G$ in this quotient
satisfies again the assumptions of the Proposition (recall that
$\widetilde{N}\cap G$ is trivial). First, if $\widetilde N$
is abelian, we take $N' = \widetilde M$\,. Then by
Corollary\;\ref{cor42}\,(c), $(\widetilde N/N')_K$
is trivial, hence $(\alpha)$ implies $\widetilde M = \widetilde N$\,. In the
general case, we consider $N' = {[\widetilde N, \widetilde N]}^-$. Then
the abelian case (and again Corollary\;\ref{cor42}\,(c)\,) gives
\,$\widetilde N = {[\widetilde N,\widetilde N]}^-\,\widetilde M$\,. But it is
well known (using induction on the nilpotency-class) that this implies
$\widetilde M = \widetilde N$\,. Clearly
\,$[\widetilde G, \widetilde G]\subseteq \widetilde N$, hence triviality of
$\widetilde{N}\cap G$ implies that $G$ is abelian, thus $GK$
is abelian \;(in fact, a little further argument shows that
\,${[\widetilde N,\widetilde N]}^-$ must be compact in this case).
\\[.8mm plus .7mm]
In steps $(\gamma),(\delta)$ the Proposition will be proved by
induction on
the nilpotency-class~$n$ of $\widetilde N$. The result is trivial when
\pagebreak
$\widetilde N$ is trivial, hence we assume now that (a) holds for
$G_1, \widetilde G_1, \widetilde N_1$\,, when $\widetilde N_1$ has
nilpotency-class smaller than $n$\,.
\vspace{0mm plus .4mm}
\item[$(\gamma)$] \;We will now prove (c) and (d). Consider
$N' = Z(\widetilde N),\ \widetilde N_1 = \widetilde N/N'$ and
$\overline G_1$ the closure of the image of $G$ in
$\widetilde G_1 = \widetilde N_1 \rtimes K$\,. The action of
$\widetilde G$ on $\widetilde N$ by inner automorphisms induces an action of
$\widetilde G_1$ on $\widetilde N$
and on $\tilde{\fr n}$ (the Lie algebra of $\widetilde N$). The inductive
assumption gives a decomposition \,$\overline G_1 = N_1^0L_1$ with 
\,$L_1 \subseteq \widetilde M_1 \times K$\,. For $x = y \, z \in L_1$
with $y \in \widetilde M_1\;(\subseteq \widetilde N/N'),\ z \in K$, the
corresponding operators on~$\tilde{\fr n}$ implement
(by uniqueness) the Jordan decomposition for the
operators from $L_1$ (see also Corollary\;\ref{cor45}).
\\[-.2mm]
Assume that $H$ is connected with Lie algebra $\fr h$\,. Then $\fr h$ is
$L_1$-invariant, hence it is also invariant under the semisimple parts.
Thus $\fr h$ is $K'$-invariant, where $K'$ denotes the projection of $L_1$
to $K$. Since $K'$ contains the projection of $G$ to $K$, it follows that
$\fr h$ is $K$-invariant. This proves (d) when $H$ is connected.
\\[.3mm plus .5mm]
For (c) obvserve that by (a) (for $\widetilde G_1$), \,\,$\overline G_1$ is
$K$-invariant, hence $\bigl(Z(\widetilde N)\mspace{2mu}GK\bigr)^-$~is a
subgroup of
$\widetilde G$, \,$\fr h$ is invariant under this subgroup and
\,$\widetilde N\,\cap\bigl(Z(\widetilde N)\mspace{2mu}GK\bigr)^-$~must
be co-compact in $\widetilde N$ \,(since $\widetilde G/G$ is compact).
Then $\fr h$ is
$\widetilde N$-invariant by \cite{Ra}\;Th.\,2.3 Cor.2 (after extending the
action of $\widetilde N$ on $\tilde{\fr n}$ to $\widetilde N_{\R}$\,, using
\cite{Ra}\;Th.\,2.11). This proves (c).
\\[0mm plus.5mm]
For the general case of (d) we first assume that $\widetilde N$ 
is torsion free. Then by (c),\linebreak
$H^0$~is normal.
Replacing $\widetilde N$ by $\widetilde N/H^0$, we can assume that $H$ is
discrete.
In addition (replacing $\widetilde N$ by $\widetilde N_{\R}$), we
may assume that $\widetilde N$ is connected. Let $D$ be the additive
subgroup of $\tilde{\fr n}$ generated by $\log H$\,. By
\cite{Ra}\;Th.\,2.12
(see the detailed version on p.\,34), $D$ is discrete and clearly
$G$-invariant. Considering
$\widetilde N_1 = \widetilde N/Z(\widetilde N)$ as above, it follows from
Lemma~\ref{le43} (and the inductive assumption) that $D$ is $K$-invariant.
$K$~connected
implies that $K$ acts trivially on $D$, hence it is trivial on $H$.
\\[.6mm plus.6mm]
If $\widetilde N$ is not torsion free, let $P$ be the maximal compact normal
subgroup. Then
(passing to $\widetilde N/P$), it follows that $PH$ is $K$-invariant. By
Corollary\;\ref{cor42}\,(a), $K$ acts semisimply on $PH$ and $P$\,.
Since $P^0$ is abelian, $K$ acts trivially on $P^0$.
$PH$ is isomorphic to a quotient of $P \rtimes H$, hence
$(PH)^0 = P^0 H^0=H^0 P^0$.
For $x\in H$ we have by  Corollary\,4.2(b), $x=y\,x_0$ with
$y\in H^0,\,k\circ x_0=x_0$ for all $k\in K$\,. Then we get
$(k\circ x)x^{-1}=(k\circ y)y^{-1}$, hence $k\circ x\in H$\,.
\vspace{-1mm plus .5mm}
\item[$(\delta)$] \;Next, we prove (a) and (b). First, we assume that
\,$Z(\widetilde N)\mspace{2mu}GK$ is dense in $\widetilde G$\,. By~(d),
$N$ is normal in $\widetilde G$\,. Thus, taking
\,$N' = N$\,, $\widetilde G_1 = \widetilde N/N' \rtimes K$, it follows from
$(\beta)$ and Corollary\;\ref{cor42}\,(c) that
\,$\widetilde N = N\widetilde M$ and that ($G_1$ denoting
the image of $G$ in~$\widetilde G_1$) \,$G_1K$ is abelian. This implies
that $\widetilde G_1$ is abelian, thus
\,$[\widetilde G, \widetilde G] \subseteq N$, proving (b).
Furthermore, by  (a) and (b) of Corollary\;\ref{cor42},
$N = N^0(\widetilde M\cap N)$, thus
$\widetilde N = N^0\widetilde M$\,.\pagebreak
\\[-1mm plus.2mm]
In the general case, it follows from the induction hypothesis (see
$(\gamma)$\,)
that \,$\widetilde G_2 = {\bigl(Z (\widetilde N)\mspace{2mu}GK\bigr)}^-$ is a
$K$-invariant subgroup of $\widetilde G$\,, containing $G$\,. Put
\,$\widetilde N_2 = \widetilde G_2 \cap \widetilde N$, then
$\widetilde G_2 = \widetilde N_2 \rtimes K$ and from the
special case above, we get that $G$ and $N$ are $K$-invariant and that
$\widetilde N_2 \subseteq N^0\widetilde M$\,. Now, if $\widetilde N$
is torsion free, we
can (passing to $\widetilde N_{\R}$) assume that it is connected as well.
By (c), $N^0$ is normal in $\widetilde N$, hence $N^0\widetilde M$ is
a connected subgroup of $\widetilde N$\,. $\widetilde N_2$ is clearly
co-compact in $\widetilde N$, hence by \cite{Va}\;Th.\,3.18.2,
\,$N^0\widetilde M = \widetilde N$\,. If $\widetilde N$
is not torsion free, let $P$ be its maximal compact normal subgroup. Then
the previous argument, applied to $\widetilde N/P$ \,(and combined with
Corollary\;\ref{cor42}\,(c)\,), gives
\,$\widetilde N = (PN)^0\mspace{1mu}\widetilde M$\,.
As noted in $(\gamma)$, $(PN)^0=N^0P^0$ and $P\subseteq\widetilde M$\,, leading
again to $\widetilde N = N^0\widetilde M$\,. The remaining properties follow
easily.\vspace{-2mm}
\end{proof}

\begin{Cor} \label{cor45}
Let $G, \widetilde G$ be as in Proposition\;\ref{pro44} and assume that $G$ is
closed in~$\widetilde G\,,\ N = G \cap \widetilde N$ torsion free, $K$ a Lie
group. Then $G, N\:(=\!\!H)$ satisfy the assumptions of \ref{di21}. $L$ is
nilpotent, $\iota_G(GK)\subseteq \Aut_1(G)$\,. For $x = y\,z \in L$ with
$y \in \widetilde M,\;z \in K$, we have
\,$\iota_G(y) = \iota_G(x)_u\,,\ \iota_G(z) = \iota_G(x)_s$\,.
\\
Put $K' = K \cap \widetilde N G$. Then \,$\Cal C_0 =\Cal C = \iota_G(K')$
satisfy the properties of Proposition\;\ref{pro215}. One gets
\,$L = G_{\Cal C}= L_{\Cal C}$\,, $\Cal C = s(L)$ \,and
\,$\overline{\Cal C}=\iota_G(K)$.
\vspace{-2.5mm}
\end{Cor}\noindent
Taking $G = \widetilde G$, similar
statements hold for $\widetilde G$ when $\widetilde N$ is torsion free.
As before, $\iota_G(x)$ denotes the restriction of the
inner automorphism
$\iota_{\widetilde G}(x)$ to $G$\,.\vspace{-1mm}
\begin{proof}
\,$L < \widetilde M\times K$ is clearly nilpotent, $[G,G]\subseteq N$ and
$G/N^0 \cong L/(L \cap N^0)$ \,(by Proposition\;\ref{pro44}\,(a)\,). The
remaining properties are easy. $L = L_{\Cal C} = G_{\Cal C}$ \,follows from
density of $K'$ in $K$. Maximality of $\Cal C$ in $s(G)$ follows from
Remark\;\ref{rem221}\,(d).
\end{proof}

\begin{Cor} \label{cor45a}
Let \,$G, \widetilde G$ be as in Proposition\;\ref{pro44} and assume that
$\widetilde N$ is torsion free.
Then
\;$[\widetilde N,\widetilde N]\cap[G,G]\
(\subseteq\,[\widetilde N,\widetilde N]\cap N)$ is a co-compact subgroup
(not necessarily closed) of
$[\widetilde N,\widetilde N]$\,.
\vspace{-1mm}\end{Cor}
\begin{proof}
Proposition\;\ref{pro44}\,(b) gives \,$[GK, GK] \subseteq N$\,.
If $N^0$ is central in $G$\,, it follows similarly as in the proof of
Proposition\;\ref{pro44}\,(c) that
$N^0$ is central in $\widetilde G$\,. In the general case (using that
$[G, N^0]$ is normal in $\widetilde G$ by Proposition\;\ref{pro44}\,(c)\,),
this implies that $[\widetilde G, N^0]=[G, N^0]$ holds.
Thus we can factor by
$[\widetilde N,N^0]$\,. Then $N^0$ is
central in $\widetilde N$\,. By Proposition\;\ref{pro44}\,(a), we get that
\,$[\widetilde N,\widetilde N]=[\widetilde M,\widetilde M]$ \;and that
\,$\widetilde M \cap GK=\widetilde M \cap LK$ is co-compact in
$\widetilde M$\,. Since $L\subseteq \widetilde M\times K$, it follows
that \,$[L,L]=[LK,LK]=[\widetilde M \cap LK,\widetilde M \cap LK]$ is
co-compact in $[\widetilde M,\widetilde M]$,
giving the desired conclusion.\vspace{-1mm}
\end{proof}

\begin{Lem}	 \label{le46}	     
Let $\widetilde N$ be a compactly generated torsion free nilpotent Lie group,
$K$~a connected compact Lie group with a continuous action on $\widetilde N$,
put \,$\widetilde G = \widetilde N\rtimes K$. Then the following holds.
\\
{\rm (i)} The action of $K$ is semisimple \,(as defined after
Proposition\;\ref{pro44})\,.
\\
{\rm (ii)} If \,$G$ is a closed subgroup of $\widetilde G$ such that
\,$\widetilde N G$ is dense in $\widetilde G$, then
\,$[K,K] \subseteq \widetilde N G^0$\,.
\end{Lem}
\begin{proof}
(i): Connectedness of $K$ implies that the action on
$\widetilde N/\widetilde N^0$ is trivial. For $x \in K$, the restricted
transformation on $\widetilde N^0$ is semisimple (by compactness). Then as in
the step (iii)$\Rightarrow$(i) of the proof of Lemma\;\ref{le22},
semisimplicity on $\widetilde N$ follows easily.
\\[-.5mm]
(ii): This is related to a theorem of Auslander, Wang and Zassenhaus
(compare the proof of \cite{Ri}\;Th.\,4.3). $\widetilde G$ has polynomial
growth, hence
the same is true for $G$ and $G/G^0$ (\cite{Gu}\;Th.\,I.3,\,I.4). Thus by
Gromov's theorem (passing to a nilpotent subgroup of finite index), we can
assume that $G/G^0$ is nilpotent. Put
\,$K' = \linebreak(\widetilde N G) \cap K\,,\
K'' = (\widetilde N G^0)\cap K$\,. We have
\,$K = [K,K]\,Z(K)^0$ and \,$Z_1 =\linebreak\ [K,K]\cap Z(K)^0$
is finite (\cite{HM}\;Cor.\,6.16). Let
\,$\varphi\!: K \to K/Z(K) \cong [K,K] /Z_1$ be the quotient mapping.
$G^0$ maps continuously onto $K''$, hence 
$\varphi (K'')$ is an analytic subgroup of the semisimple group \,$[K,K]/Z_1$
and $K'$-invariant. Since $K'$ is dense in $K$, it follows (considering the
Lie algebra)
that $\varphi (K'')$ is normal and that it is closed (\cite{Va}\;Th.\,4.11.6).
$K'/K''$ is isomorphic to a quotient of \,$G/G^0$, hence it is nilpotent.
$[K, K]$
being semisimple, it follows that $\varphi (K'')$ has finite index in
$\varphi (K')$, hence $\varphi (K')$ is closed as well and dense, giving 
\,$\varphi (K') = \varphi (K'') = [K,K]/Z_1$, thus
\,$[K, K] \subseteq K'' Z_1$\,. $K''$ being an analytic subgroup of $K$, we
arrive (considering the Lie algebra) at \,$[K,K] \subseteq K''$\,.
\vspace{-.5mm}
\end{proof}

\begin{Pro}  \label{pro47}    
Let $\widetilde N$ be a compactly generated torsion free nilpotent Lie group,
$K$ a compact Lie group with a continuous action on $\widetilde N$. Let $G$
be a closed subgroup of \,$\widetilde G = \widetilde N\rtimes K$ such that
\,$\widetilde N G$ is dense in $\widetilde G$ and \,$\widetilde G/G$ compact.
Put \,$N = G \cap \widetilde N$. Then the following properties hold.
\\[-.1mm]
{\rm (a)} If $R_K$ denotes the (non-connected) radical of $K$, then
\,$R = G \cap (\widetilde N \rtimes R_K)$ is the radical of $G$\,.
\\[-.3mm]
{\rm (b)} For \,$K_1 = Z(K^0)^0$, the groups
\,$\widetilde G_1 = \widetilde N\rtimes K_1\,,\
G_1 = G\cap\widetilde G_1$ satisfy the assumptions of
Proposition\,\ref{pro44}. Every
connected semisimple subgroup of $\widetilde G$ is contained in~$G$
\,(in particular, $[K^0, K^0]$ is a Levi subgroup of $G^0$).
$N^0$ is a normal subgroup of $\widetilde{G}$ \,and if
$\widetilde N$ is connected, then $N_{\R}$ is normal as well.
\\ 
{\rm (c)} Assume that the action of $K$ is faithful. Then $G$ has no
non-trivial compact normal subgroups and $N$ is the nilradical of $G$\,. If
$x \in\widetilde G$ normalizes $N$ and acts unipotently on $N$,
then $x \in \widetilde N$\,. We have
\,$C_{\widetilde G}(G)= Z(\widetilde G)\subseteq Z(\widetilde N)$\,.
\\[.1mm] 
{\rm (d)} Assume that $K$ is connected. Then \,$K\cap\mspace{1mu} G$ is a
maximal compact subgroup of~$G$\,. If $C$ is any compact subgroup of
$\widetilde G$\,, there exists $n \in N^0$ such that
\,$n\,Cn^{-1} \subseteq K$,
\,$G$~is $C$-invariant,\pagebreak \,$[G^0C,\widetilde N] \subseteq N^0$.\vspace{-1mm}
\\
{\rm (e)} Assume that $K$ is abelian. Then $[G,G]$ \,($\subseteq N$) is
co-compact in $[\widetilde G, \widetilde G]^-$. Thus if $\widetilde N$ is
connected,
\;$[\widetilde N, \widetilde N] \subseteq [\widetilde G, \widetilde G] =
({[G,G]}^-)_{\R}\subseteq N_{\R}$\,.
\end{Pro}
\begin{proof}
(a) is easy.
\\
To prove (b), we first consider
\,$\widetilde G_2 = \widetilde N\rtimes K^0,\
G_2 = G \cap \widetilde G_2$\,. The subgroup $K^0$ has finite index in
$K$, hence \,$\widetilde G/\widetilde G_2$\,, $G/G_2$ are finite and it
follows easily that $G_2\,,\widetilde G_2$ satisfy again the assumptions of
the Proposition.
By Lemma\;\ref{le46}\,(ii) \,(and since\linebreak
$G^0_2 = G^0$),
\,$[K^0,K^0] \subseteq \widetilde N G^0\subseteq \widetilde G_2$\,.
It follows \,(since $K^0 = [K^0,K^0] K_1$\,, \cite{HM} Cor.\,6.16)
that \,$\widetilde G_1\,G_2 = \widetilde G_2$
and that \,$\widetilde N G_1$ is dense in $\widetilde G_1$\,. Then 
\,$G_2/G_1 \cong \widetilde G_2/\widetilde G_1$ is compact which
implies that $\widetilde G/G_1$ and also $\widetilde G_1/G_1$ are
compact.
This gives the assumptions of Proposition\;\ref{pro44}. Putting
\,$\widetilde M = \{x \in \widetilde N\! :\, k\circ x = x \text{ \,for all }
k \in K_1\}$, it follows from Proposition\;\ref{pro44}\,(a) that
\,$\widetilde N = N^0\widetilde M$\,. Clearly \,$[G_2,G_1] \subseteq N$\,,
thus \,$[G^0,G_1] \subseteq N^0$.
If $H$ is a connected, closed $G$-invariant subgroup of
$\widetilde N$ then by Proposition\;\ref{pro44}\,(c), $H$ is normal in
$\widetilde G_1$\,, thus it is $\widetilde G_1G$-invariant, hence $H$ is
normal in $\widetilde G$\,. This applies to $N^0$ and $N_{\R}$
(which is $G$-invariant by \cite{Ra}\;Th.\,2.11). 
By connectedness, the action of $K_1$ on
\,$\widetilde N/N^0 \cong \widetilde M/(N^0 \cap \widetilde M)$ is trivial.
Thus \,$[\widetilde G_2,K_1] \subseteq N^0$
and then \,$[G^0,G_1 K_1] \subseteq N^0$. By Proposition\;\ref{pro44}\,(a),
$\widetilde M \cap G_1 K_1$ is a co-compact subgroup of
$\widetilde M$.
Applying \cite{Ra}\;Th.\,2.11 to the image of this subgroup in
$\widetilde M/(N^0\cap\widetilde M)$\,, it follows that the induced action
of $G^0$ on $\widetilde M/(N^0\cap\widetilde M)$ is trivial, thus
\,$[G^0, \widetilde N] \subseteq N^0\ (\subseteq  G^0)$.
\\[1mm]
Let \,$\varphi\!:\widetilde G_2 \to K^0 \
(\cong \widetilde G_2/\widetilde N)$ be the quotient mapping and let $C$
be a Levi subgroup of $G^0$. We have shown
that \,$\varphi (G^0)\supseteq [K^0, K^0]$\,. Considering the Lie algebras of
$G^0$ and $K^0$, it follows that $\varphi$ maps the connected radical of $G^0$
to $Z(K^0)^0$, consequently $\varphi (C) = [K^0,K^0]$\,.
Clearly, $[K^0,K^0]$ is a Levi subgroup of~${\widetilde G_2}^0$\,.
By \cite{Va}\;Th.\,3.18.13, there exists 
$x \in \widetilde N^0$ such that $\iota_x(C)\subseteq [K^0,K^0]$\,. We
have shown above that $G^0$ is $\widetilde N$-invariant, hence 
$\iota_x(C)\subseteq G^0$ and we may assume that $C\subseteq [K^0, K^0]$\,.
Then \,$C = \varphi (C) = [K^0,K^0] \subseteq G^0$ and again by
\cite{Va}\;Th.\,3.18.13, 
any Levi subgroup of $\widetilde G^0_2$ is contained in
$N^0[K^0,K^0]\subseteq G^0$. Note further that \,$[K^0,K^0] \subseteq G^0$
implies \,$\bigl[[K^0,K^0], \widetilde N\bigr] \subseteq N^0$, hence
\,$[K^0,\widetilde N]\subseteq N^0$.
\\[0mm plus 1mm]
For (e), we can factor by ${[G,G]}^-$ ($\subseteq\widetilde N$\,, being
normal in $\widetilde G$ as above). Then $G$~is abelian. By
Corollary\;\ref{cor45a},
$\widetilde N$ is abelian. Easy calculations (using that $K$ is abelian)
show that $G\cap\,\widetilde M\!\rtimes\!K$ acts trivially on $G_1K_1$\,,
hence (recall that $\widetilde M \cap\,G_1 K_1$ is co-compact in
$\widetilde M$) by \cite{Ra}\;Th.\,2.11 it acts trivially on
$\widetilde M$ and it follows that\linebreak
$\widetilde G$~is abelian.
\\[0mm plus .5mm]
Next, we come to (d). Here we assume $K^0 = K$\,, thus
\,$\widetilde G_2 = \widetilde G$\,. We can (replacing $\widetilde N$ by
$\widetilde N_{\R}$) also
assume that $\widetilde N$ is connected and then $\widetilde G$ is connected
as well. We have shown above that $[G^0K,\widetilde N]\subseteq N^0$. If
$C$ is a compact subgroup of $G$, it follows from \cite{Ho}\;Th.\,XV.3.1
(clearly $K$ is a maximal compact subgroup of $\widetilde G$) that there
exists $x\in \widetilde N$ such that $xCx^{-1} \subseteq K$\,. Then
$[K, \widetilde N] \subseteq N^0$ implies $C \subseteq N^0K$.
Repeating the argument with $N^0 \rtimes K$ instead of $\widetilde G$\,,
it follows that we can assume that $x \in N^0$. Then
$x C x^{-1} \subseteq  K \cap G$\,. If (as above)
$\varphi$ denotes the projection to $K$, we have
\,$\varphi (C) = \varphi (xCx^{-1})$
and maximality of $K \cap G$ follows. $C \subseteq N^0K$ gives 
\,$[G^0C, \widetilde N] \subseteq N^0$. \;$G_1$ is $K_1$-invariant by
(b) and Proposition\;\ref{pro44}\,(a). Furthermore, since $K$ is connected,
$K = [K,K]\,K_1$\,, hence (b) implies
\,$G = G_1[K,K]$ and it follows that $G$ is $K$-invariant and also
invariant under  $x K x^{-1} \text{ for any } x \in G$\,.
\\[.4mm plus 1.5mm]
Finally, we prove (c). Again,
we can assume that $\widetilde N$ is connected.
Then by (b), $N_{\R}$~is normal in $\widetilde G$\,.
Assume that $x \in\widetilde G$ normalizes $N\,,\ x = y\,z$ with 
$y \in \widetilde N\,,\ z \in K$. If $x$ acts unipotently
on $N$, then the same is true on $N_{\R}$ \;(considering the commutator
series of $N_{\R}$ and using \cite{Ra}\;Th.\,2.3\;Cor.\,1, this can be reduced
to the abelian case which is easy). From \cite{Wan}\;(2.2),(2.3) \,(applied to the
automorphisms of $N_{\R}$ defined by the elements of
$\langle x\rangle \widetilde N$\,), we conclude that
$\iota_z$ is both semisimple and unipotent on $N_{\R}$ and it follows that
$z$ centralizes $N_{\R}$\,. $C_{\widetilde G} (N_{\R})$ is normal in
$\widetilde G$. By (b) and Proposition\;\ref{pro44}\,(a),
$K_1 = Z (K^0)^0$ is faithful on $N$, thus
$K_1 \cap C_{\widetilde G} (N_{\R})$ is trivial and we get ($K_1$ is
normal in $K$) that $z \in C_K (K_1)$\,.
By (e), we have
\,$[\widetilde G_1, \widetilde G_1] \subseteq N_{\R}$\,.
Put $\widetilde G_3=\widetilde N\rtimes C_K (K_1)$. Then
$[\widetilde G_3\cap G\,,G_1]\subseteq N$\,.
By assumption, $\widetilde N G$ is dense in $\widetilde G$ and since 
$\widetilde N C_K (K_1)\ (\supseteq\widetilde N K^0)$ is open, it follows that
\,$\widetilde G_3 = \widetilde G_1(\widetilde G_3\cap G)$\,.
This gives $[\widetilde G_3,G_1]\subseteq N_{\R}$\,. Then $z \in C_K (K_1)$
implies $[z,G_1K_1] \subseteq N_{\R}$\,. Now Proposition\;\ref{pro44}\,(a) and
\cite{Ra}\;Th.\,2.11 give 
$[z, \widetilde N] \subseteq N_{\R}$\,. Since $\iota_z$ is semisimple on
$\widetilde N$, it follows that $z$ centralizes $\widetilde N$, thus
(by faithfulness) $z = e$\,. This proves that $x \in \widetilde N$\,.
\\[1mm plus .7mm]
It follows that $C_{\widetilde G}(G)\subseteq C_{\widetilde G}(N)\subseteq
\widetilde N$.
Take $x \in C_{\widetilde G}(G)\cap N^0$. Applying
Corollary\;\ref{cor26} to the automorphisms defined by $G_1$\, (and using
Corollary\;\ref{cor45}) , it follows that $x$
commutes with $K_1 \cap (G_1\widetilde N)$ which
is dense in $K_1$\,. We get that $x$ commutes with $K_1$\,. By
Proposition\;\ref{pro44}\,(d), $C_{\widetilde G}(G)$ is $K_1$-invariant,
$[\widetilde N,K_1]\subseteq N^0$ by Proposition\;\ref{pro44}\,(a).
Then semisimplicity (using Lemma\;\ref{le46} and Corollary\;\ref{cor42}\,(a)) 
implies that $C_{\widetilde G}(G)$ commutes with $K_1$\,.
By Proposition\;\ref{pro44}\,(a),
$\widetilde N/(\widetilde N \cap G_1 K_1)$ is compact, hence by
\cite{Ra}\;Th.\,2.11,
$C_{\widetilde G}(G)\subseteq Z(\widetilde N)$. Since $G\widetilde N$ is
dense in
$\widetilde G$, we arrive at $C_{\widetilde G}(G)\subseteq Z(\widetilde G)$\,.
\\[-.2mm plus .3mm]
Next, assume that $C$ is a compact normal subgroup of $G$\,. Clearly
$C \cap \widetilde N$ must be trivial, thus $C$ centralizes $N$\,.
It follows that \,$C \subseteq \widetilde N$, hence $C$ is trivial.
\\[.8mm plus.3mm]
Let $N_1$ be the nilradical of $G$\,.
Clearly $N_1 \supseteq N$. Take
$x \in N_1$\,. Then $x$ normalizes $N$ and acts unipotently, hence
$x\in\widetilde N$\,. 
Thus $x \in \widetilde N \cap G = N$, proving that $N_1 = N$\,.\pagebreak 
\end{proof}
\vspace{-3mm plus .5mm}
\begin{proof}[{\bf Proof of Theorem\;\ref{th3}}]
\item[$(\alpha)$] \,In $(\alpha) - (\gamma)$, we assume that
$j'(G)\widetilde N'$ is dense in $\widetilde G'$.
This is no restriction
(reducing $K'$) when proving existence of $\Phi$ \,(that the same is true
concerning uniqueness will be seen in $(\delta)$\,).
Furthermore, we assume in the steps $(\alpha) - (\gamma)$ that $K, K'$ are
connected abelian and $K$ acts faithfully on $\widetilde N$. Then by
Proposition\;\ref{pro44}\,(a), $j(G)$ is
$K$-invariant. Since $j$ induces a topological isomorphism between $G$
and $j(G)$, we get a continuous homomorphism
\,$\alpha\!: K \to \Aut (G)$ satisfying 
$j\bigl(\alpha (k) (x)\bigr) = k\,j(x)\,k^{-1} \text{ for all } k \in K\,,\,x \in G$\,.
By Proposition\;\ref{pro44}\,(a), the action of $K$ on $j(G)$ is faithful
\,(since it is faithful on $\widetilde N$; alternatively one can use
Proposition\;\ref{pro47}\,(c)\,), hence $\alpha$ is injective. Similarly, we
get \,$\alpha'\!: K' \to \Aut(G)$\,. We claim that there exists
$n \in N^0$ such that \,$\alpha' (K') = n\,\alpha (K)\,n^{-1}$
\,(as explained before Proposition\;\ref{pro220}, this notation is shorthand
for \,$\iota_n \circ \alpha (K) \circ \iota_n^{-1}$). The claim will be
proved in $(\beta)$ below. Then, replacing
$j'$ by $j'' = j' \circ \iota_n^{-1}$ \,(which replaces $\alpha' (k)$ by
$\alpha'' (k) = \iota_{n} \circ \alpha' (k) \circ \iota_n^{-1}
\text{ for } k \in K'$),
it will be enough to show existence of $\Phi$ under the additional
assumption $\alpha(K) = \alpha'(K')$\,. This will be done in $(\gamma)$.
In $(\delta)$ we will prove uniqueness of $\Phi$ for general $K, K'$ connected
and $K$ acting faithfully, then in $(\epsilon)$ uniqueness for general
$K,K'$ will be shown.
Finally, existence for general $K, K'$ will be treated in $(\varphi)$ and
also the question of surjectivity and injectivity.
\vspace{.5mm plus .9mm}
\item[$(\beta)$] \,We assume that $K, K'$ are connected and abelian,
$K$ acts faithfully on $\widetilde N$. Put
\,$\Cal C_0 = \alpha (K\cap\,\widetilde N j(G))$\,,
\,$\overline{\Cal C} = \alpha (K)$\,,
\,$\Cal C'_0 = \alpha' (K'\cap\mspace{1mu}\widetilde N'j'(G))$\,,
\,$\overline{\Cal C'} = \alpha' (K')$\,.
\,Then Corollary\;\ref{cor45} takes us to the setting of Section 2
\;($j$ transfers $\Cal C$ to
\,$\iota_{j(G)} (K)$, similarly for $\Cal C_0$ and for $j'$).
By Proposition\;\ref{pro220} \,(note that Proposition\;\ref{pro47}\,(c)
implies $j^{-1}(\widetilde N)=j'^{-1}(\widetilde N')=\nil(G)$, i.e., both
instances of Corollary\;\ref{cor45} refer to the same subgroup
$H=N$ of $G$\,), there exists $n\in N^{0}$ such that
\,$\Cal C_0' = n\,\Cal C_0\, n ^{-1}$.
By assumption, $K\cap\,\widetilde N j(G)$ \,(resp.
$K'\cap\mspace{1mu}\widetilde N'j'(G)$) is dense in $K$ (resp. $K'$),
it follows that
\,$\alpha (K) = n\,\alpha'(K')\,n^{-1}$ (actually, $\overline{\Cal C}$
coincides with the closure of $\Cal C_0$).
\vspace{.6mm plus .5mm}
\item[$(\gamma)$] \,Now we prove  existence of the extension $\Phi$ under the
assumption that $K, K'$ are connected abelian and
\,$\alpha(K)=\alpha'(K')$\,. First, 
we want to show that this implies \,$j^{-1}(K) \subseteq {j'}^{-1}(K')$ and
that \,$j'\circ j^{-1} (k) = {\alpha'}^{-1} \circ \alpha (k)$
holds for $k \in K \cap j(G)$\,.
\\[2.5mm]
Take $x \in j^{-1} (K)$\,. Then \,$\iota_x = \alpha (j(x))\in\alpha'(K')$\,.
Put \,$k' = {\alpha'}^{-1}(\iota_x)$\,, then ${k'}^{-1}j'(x)$ centralizes
$j'(G')$\,, hence by Proposition\;\ref{pro47}\,(c),
${k'}^{-1}j'(x) \in Z (\widetilde G')$\,. Thus
\,$j'(x) \in Z(\widetilde G')K'$\,. Since $j^{-1}(K)$ is compact and 
$Z(\widetilde G')\subseteq \widetilde N'$ is torsion free, it follows that
$j'(x)\in K'$, proving \,$j^{-1}(K)\subseteq {j'}^{-1}(K')$\,. Since
\,$\alpha'(j'(x)) = \iota_x$\,,
the second formula follows easily.
\\[0mm plus.2mm]
For $k \in K\,,\;x \in G$\,, we define
\,$\Phi\bigl(j(x)\,k\bigr) = j'(x)\,{\alpha'}^{-1}\!\circ\!\alpha(k)$\,. Then
it follows from the properties above that
\,$\Phi\!:j(G)K\to j'(G)K'$ is well defined and (see the
definition of $\alpha, \alpha'$) that it is a homomorphism. Furthermore,
$\Phi (K) \subseteq K'$. Closedness of $j(G)$ implies that $j(G)K$ is
isomorphic to a quotient of $G \rtimes K$ \,(\cite{HR}\;Th.\,5.21) and
continuity of $\Phi$ follows.
$(j(G)K)\cap\widetilde N$ is a nilpotent normal subgroup of $j(G)K$\,.
Applying Proposition\;\ref{pro47}\,(c) to $\Phi(j(G)K)$ \,(see also
$(\delta)$ below), it follows that
$\Phi(j(G)K\cap\widetilde N)\subseteq\nil(\Phi(j(G)K))=
\Phi(j(G)K)\cap\widetilde N'$.
By Proposition\;\ref{pro44}\,(a),
$(j(G)K)\cap\widetilde N$ is a co-compact subgroup of $\widetilde N$.
By \cite{Ra}\;Th.\,2.11,
\,$\Phi|\bigl(j(G)K \cap \widetilde N\bigr)$ has a unique extension to a
continuous homomorphism \,$\widetilde N\to \widetilde N'$ (again
denoted by $\Phi$). Uniqueness of the extension implies 
\,$\Phi (knk^{-1}) = \Phi(k)\Phi(n)\Phi(k^{-1})$ for
$k \in K,\;n \in \widetilde N$
and then $\Phi$ extends further to a homomorphism
\,$\widetilde G \to \widetilde G'$.
\vspace{.5mm plus.9mm}
\item[$(\delta)$] \,We claim that \,$\Phi(\widetilde N)\subseteq\widetilde N'$
holds for any \,$\Phi,\widetilde G,\widetilde G'$ as in Theorem\;\ref{th3}.
We have \,$\Phi(j(G))=j'(G)$.
Since $j(G)$ is co-compact in $\widetilde G$ and $j'(G)$ is closed, it
follows that $\Phi(\widetilde G)$ is closed in $\widetilde G'$.
Thus (reducing $K'$ temporarily) it also satisfies the assumptions of
Proposition\;\ref{pro47} and Proposition\;\ref{pro47}\,(c) gives
\,$\nil\bigl(\Phi(\widetilde G)\bigr)=\Phi(\widetilde G)\cap\widetilde N'$.
Since (again by Proposition\;\ref{pro47}\,(c)\,)
\,$\widetilde N'=\nil(\widetilde G')$, we conclude that
\,$\Phi(\widetilde N)\subseteq\nil\bigl(\Phi(\widetilde G)\bigr)
\subseteq\widetilde N'$, proving our claim. \,It follows that
$\Phi(\widetilde G)$ is contained in the closure of $\widetilde N'j'(G)$.
Thus we can always assume that $\widetilde N'j'(G)$ is dense in
$\widetilde G'$ when proving uniqueness.
\\[0mm plus .1mm]
Next, we prove uniqueness of the extension for $K$ connected,
$j'(G)\widetilde N'$ dense in $\widetilde G'$
and $K$ acting faithfully on $\widetilde N$.
Assume that 
\,$\Phi_1, \Phi_2\!: \widetilde N\rtimes K \to
\widetilde N' \rtimes K'$ are continuous group homomorphisms satisfying
$\Phi_i \circ j = j'$ for $i = 1,2$. If
$x \in \widetilde G$ normalizes $j(G)$\,, then $\iota_{\Phi_1(x)}$ coincides
with $\iota_{\Phi_2(x)}$
on $j'(G)$. Put \,$\Psi (x) = \Phi_2(x)^{-1}\Phi_1(x)$. Then $\Psi(x)$
commutes with $j'(G)$, hence Proposition\;\ref{pro47}\,(c) implies,
$\Psi(x)\in Z(\widetilde G') \subseteq \widetilde N'$. By
Proposition\;\ref{pro47}\,(d), $j(G)$ is $K$-invariant. It follows that
\,$\Psi\!: j(G)K\to Z(\widetilde G')$ is
a continuous group homomorphism and by assumption, $j(G)\subseteq \ker\Psi$\,.
Since $\widetilde N'$ is torsion free, we get that $\Psi$ must be
trivial, hence \,$\Phi_1, \Phi_2$ coincide on $K$\,.
\\[0mm plus 1.3mm]
By Proposition\;\ref{pro44}\,(a),
$j(G)K$ contains a co-compact subgroup of $\widetilde N$, thus by 
\cite{Ra}\;Th.\,2.11 \,(recall that
$\Phi_i(\widetilde N)\subseteq\widetilde N'$) \,$\Phi_1, \Phi_2$ coincide
on $\widetilde N$\,, proving that \linebreak
$\Phi_1 = \Phi_2$\,.
\vspace{-1mm plus 1mm}
\item[$(\epsilon)$] \,Now, we prove uniqueness of the extension for general
$K,K'$. Consider \,$\Phi_1, \Phi_2$ as in ($\delta$).
Let $K_{\widetilde N}$ be the kernel of the action of $K$ on
$\widetilde N$\,.
This is a compact normal subgroup of
$\widetilde G$, hence \,$\Phi_i(K_{\widetilde N})$ is a compact normal
subgroup of $\Phi_i(\widetilde G')$. By faithfulness of the action on
$\widetilde N'$, it follows from Proposition\;\ref{pro47}\,(c) (applied to
$G=\Phi_i(\widetilde G')$) that $\Phi_i(K_{\widetilde N})$ must be trivial,
hence $K_{\widetilde N}\subseteq \ker\Phi_i$ holds for $i=1,2$.
\\[1.mm plus .6mm]
Passing to $K/K_{\widetilde N}$ (and composing $j$ with the quotient
mapping), we can now assume that $K$ acts faithfully on $\widetilde N$.
Put
\,$G_2=j^{-1}(\widetilde N \rtimes K^0)\cap
{j'}^{-1}(\widetilde N\rtimes K'^0)$\,.
Then $G_2$ is a closed subgroup of
$G$ with finite index. It follows that $\widetilde N j(G_2)$ is dense in
$\widetilde N \rtimes K^0$ \,(a connected group has no proper closed
subgroups of finite index) and similarly for $j'(G_2)$\,. Thus, we can
apply $(\delta)$ and conclude that $\Phi_1, \Phi_2$ coincide on
$\widetilde N \rtimes K^0$\,. Density of $\widetilde N j(G)$ in
$\widetilde G$ implies that
\,$\widetilde G = (\widetilde N \rtimes K^0)\,j(G)$\,, consequently
$\Phi_1 = \Phi_2$\,.
\vspace{-.5mm plus 1.1mm}
\item[$(\varphi)$] \,We show existence of the extension for general $K, K'$.
Consider $K_{\widetilde N}$ as in $(\epsilon)$.
Similarly we get $j'(j^{-1}(K_{\widetilde N})\cap G)=\{e\}$\,,
hence $K_{\widetilde N}\cap j(G)=\{e\}$\,. Thus we can pass to
$K/K_{\widetilde N}$ \,and assume that $K$ acts faithfully on $\widetilde N$.
Put 
\,$K_1 = Z(K^0)^0,\;K'_1 = Z({K'}^0)^0,\;
G_1 = j^{-1}(\widetilde N \rtimes K_1),\;
G'_1 = {j'}^{-1}(\widetilde N\rtimes K'_1),\;G_2 = G_1\cap G'_1$\,.
$G_1, G'_1$ are closed normal subgroups of $G$\,.
We have $R^0_K = K_1$ (\cite{Va}\;Th.\,4.11.7), hence by
Proposition\;\ref{pro47}\,(a), $G_1$ has finite index in the radical
$R$ of $G$ and the same is true for $G'_1$\,. It follows that $G_1/G_2$
is finite and from Proposition\;\ref{pro47}\,(b), we get (similarly as in
$(\epsilon)$\,) that
$\bigl(\widetilde N j(G_2)\bigr) \cap K_1$ is dense in $K_1$\,, analogously
for $\bigl(\widetilde N'j'(G_2)\bigr)\!\cap\!K'_1$\,. Thus we can apply the
connected abelian case (\,($\alpha$)-($\gamma$)\,) to $G_2$ with $K, K'$
replaced by $K_1, K'_1$\,. This gives a homomorphism
\,$\Phi\!:\widetilde N\rtimes K_1\to\widetilde N'\rtimes K'_1$
satisfying \,$\Phi \circ j = j'$ on $G_2$\,. Uniqueness of the extension
(shown in $(\epsilon)$) implies that 
$\Phi\bigl(j(x)\,y\,j(x)^{-1}\bigr) = j'(x)\,\Phi(y)\,j'(x)^{-1}$ for all
$x \in G,\;y \in \widetilde N\rtimes K_1$\,. Now take $x \in G_1$\,,
then $j(x)\in\widetilde N\rtimes K_1$ and it follows that
\,$z = \Phi\bigl(j(x)\bigr)^{-1} j'(x)$ commutes with
$\Phi(\widetilde N \rtimes K_1)\supseteq j'(G_2)$\,.
hence by Proposition\;\ref{pro47}\,(c), $z \in \widetilde N'$. This implies
$j'(x)\in\widetilde N \rtimes K'_1$\,, hence $x \in G'_1$\,. This 
shows that $G_2=G_1\subseteq G'_1$\,.
\\[.1mm plus .2mm]
It follows from density of $\widetilde N j(G)$ in $\widetilde G$ and
Proposition\;\ref{pro47}\,(b) that
$\widetilde G =\linebreak
(\widetilde N \rtimes K_1)\,j(G)$\,. On
$j(G)$ we put $\Phi = j' \circ j^{-1}$. Then, by the properties
above, the two definitions of $\Phi$ agree on 
$j(G)\cap(\widetilde N \rtimes K_1)$ 
and they can be combined to give a continuous homomorphism 
\,$\widetilde G_1 \to \widetilde G'_1$\,.
\\[0mm plus .5mm]
By construction, we always have $K_{\widetilde N}\subseteq \ker\Phi$
\,(see also $(\epsilon)$\,) and by $(\delta)$,
$\Phi(\widetilde G)\subseteq {\bigl(\widetilde N'j'(G)\bigr)}^-$.
If $K$ acts faithfully and $\widetilde N'j'(G)$ is dense in
$\widetilde G'$, we can interchange the
r\^oles of $\widetilde G$ and $\widetilde G'$ and in the usual manner, it
follows from uniqueness of the extension that $\Phi$ is an isomorphism.
For the general case, this implies
\,$\Phi(\widetilde G)= {\bigl(\widetilde N'j'(G)\bigr)}^-$
and \,\,$\ker\,\Phi= K_{\widetilde N}$\,.
\end{proof}
\begin{Cor} \label{cor48}	
Let $G,\widetilde G,G_1, K_1$ be as in Proposition\;\ref{pro47}\,(b),
put \,$\widetilde M = \{x\in\widetilde N\! : 
k \circ x = x \text{ \,for all } k \in K_1\},\
L = G \cap (\widetilde M \rtimes K),\ L_1 = G_1 \cap L$\,. Then the
following properties hold:
\\
$G_1$ is normal in $G$\,, $G/G_1$ is compact,
$[G_1, G_1] \subseteq  N \subseteq G_1\,,\ G = N^0L\,,\linebreak
G_1 = N^0L_1\,,\ L/L_1$ is compact, $L_1$ is nilpotent.\vspace{-2mm}
\end{Cor}\noindent
Thus $L$ satisfies the properties of Theorem\;\ref{th1}. We will exemplify
the constructions in step\,$(b)$ of the proof of Theorem\;\ref{th2} for
this choice of $G_1$\,. By Corollary\;\ref{cor45},
$\overline{\Cal C}=\iota_{G_1}(K_1)$, thus
\,$G_2=G_1\rtimes K_1\subseteq (\widetilde N\rtimes K_1)\rtimes K_1$\,.
Since $K_1$ is abelian, we can interchange the $K_1$-components and use
the representation
\,$G_2=\{(x,\sigma_1,\sigma_2):\linebreak\,(x,\sigma_2)\in G_1\,,
\sigma_1\in K_1\}\subseteq\widetilde N\rtimes(K_1\times K_1)\subseteq
\widetilde N\rtimes(K_1\rtimes K)$, where the action of $K_1\rtimes K$ on
$\widetilde N$ is
given by $(\sigma_1,\sigma_2)\circ x=(\sigma_1\,\sigma_2)\circ x$.
Then $N_2=\{(x,\sigma^{-1},\sigma):\,(x,\sigma)\in G_1\,\}$.
Embedding $G$ to $\{(x,e,\sigma):\,(x,\sigma)\in G\}$ this produces the
action of $G$ on $G_2$ defined in the proof of Theorem\;\ref{th2}
and one can take $G_3=\{(x,\sigma_1,\sigma_2):\,(x,\sigma_2)\in G\,,\linebreak
\sigma_1\in K_1\}$. It
is not hard to see that $P_3=\{(e,\sigma,\sigma^{-1}):\,\sigma\in K_1\}$
\,(the kernel of the action of $K_1\rtimes K$). It follows that $G_4$ can
be identified with the subgroup  \,$G\,K_1$ of $\widetilde N\rtimes K$\,,
and then $N_4$ corresponds to $GK_1\cap\widetilde N$\,.
\begin{proof}
$G_1,\widetilde G_1$ satisfy the assumptions of Proposition\;\ref{pro44}
\,(see also the proof of
Proposition\;\ref{pro47}\,(b)). By Proposition\;\ref{pro44}\,(a), 
$\widetilde N = N^0\widetilde M$ which implies \,$G= N^0L$\,,
$G_1 = N^0L_1$\,. In particular, $G_1L = G$ is closed, giving 
\,$L/L_1 \cong G/G_1$ \,(\cite{HR}\;Th. 5.33).
$L_1 \subseteq \widetilde M \times K_1$ is nilpotent.
The remaining properties are clear.\vspace{1mm}
\end{proof}\noindent
Next, we describe some special cases of Theorem\;\ref{th2}.\vspace{-1mm}
\begin{Cor} \label{cor410}
Let $G, \widetilde{G}$ be as in Theorem\;\ref{th2},
with $\widetilde N G$ dense in $\widetilde G$\,.
\\[.2mm]
{\rm (a)} The following properties are equivalent
\\[.3mm]\hbox{ \ }
{\rm (i)} $K$ is abelian \quad \
{\rm (ii)} $[G, G] \subseteq N$\quad\
{\rm (iii)} the action of $G$ on $(\fr n_{\R})_{\C}$ triangulizes.
\\[2mm]
{\rm (b)} The following properties are equivalent \vspace{-1mm}
\begin{enumerate}
\item[(i)] $G,N$ satisfy the assumptions of \ref{di21}.
\item[(ii)] $K$ is abelian and acts trivially on $N/N^0$.
\item[(iii)] $G$ acts unipotently on $N/N^0$ and the action of $G$ on
$\fr n_{\C}$ triangulizes.
\item[(iv)] $K$ is abelian and $G/G^0$ is nilpotent.
\item[(v)] $K$ is abelian and with $\widetilde M$ as in Corollary\;\ref{cor48}
one has $\widetilde M = \{x\in\widetilde N\! :\linebreak
k \circ x = x \text{ \,for all } k \in K\}$.\vspace{-.8mm}
\end{enumerate}
{\rm (c)} $G/N$ is compact if and only if $\widetilde N/N$ is compact
\;(equivalently: \,$\widetilde N=N_{\R}$\,).
\\[1mm]
{\rm (d)} $\widetilde N$ is abelian if and only if $N$ is an
${FC}^-_{G\,}$-group
and there exists an abelian subgroup $H$ of $G$ such that $NH$ is closed and
$G/(NH)$ is compact.\vspace{-2mm}
\end{Cor}\noindent
As before, $\fr n$ denotes the Lie algebra of $N$, $\fr n_{\R}$ that of the
Malcev completion $N_{\R}$ and $\fr n_{\C}, (\fr n_{\R})_{\C}$ 
denote the complexifications of $\fr n, \fr n_{\R}$\,. The proof will show
that in (d) one can take $H=L_1$ (the group of Corollary\;\ref{cor48}).
Furthermore, the proof of (d) shows that $N$ is an ${FC}^-_{G\,}$-group
iff it is central in $\widetilde N$ and this is equivalent to $N^0$ being
central in $\widetilde N$\,.\vspace{-2mm}
\begin{proof}
(a) (i)$\Rightarrow$(ii) is trivial.
\\
(ii)$\Rightarrow$(iii): the action of $N$ on $\fr n_{\R}$ is
clearly unipotent, thus the same is true on $(\fr n_{\R})_{\C}$ and
(iii) follows from \cite{Wan}\;(2.2).
\\
(iii)$\Rightarrow$(ii): If $x \in [G,G]$, it follows from (iii) that the
automorphism of $N_{\R}$ induced by $\iota_x$ is unipotent, hence the
same is true on $N$. By Proposition\;\ref{pro33}, this implies $x \in N$
(alternatively, one could use (c) of Proposition\;\ref{pro47}).
\\
(ii)$\Rightarrow$(i): If (ii) holds, then the image of $G$ in
$\widetilde G/\widetilde N\ (\cong K)$ is abelian and by assumption, it is
dense.
\\
(b) (i)$\Rightarrow$(iii) follows from nilpotency of $G/N^0$ and using
again \cite{Wan}\;(2.2).
\\
(iii)$\Rightarrow$(ii): The action of $[G,G]$ on $N$ is unipotent, hence
again by Proposition\;\ref{pro33}, $[G,G] \subseteq N$ and from (a) it
follows that $K$ is abelian. $N^0$ is normal by Proposition\;\ref{pro44}\,(c)
combined with Proposition\;\ref{pro47}\,(b). By (a), the action of $G$ on
$(\fr n_{\R})_{\C}$ triangulizes, hence also that on
$(\fr n_{\R}/\fr n)_{\C}$\,. $G$ acts unipotently on $\fr n_{\R}/\fr n$,
hence by \cite{Wan}\;(2.3),
$K$ acts trivially on $\fr n_{\R}/\fr n$ and the same is true on
$N/N^0 \subseteq N_{\R}/N^0$.
\\
(ii)$\Rightarrow$(i): By (a), we have $[G, G]\subseteq N$\,, thus $G/N$
is abelian. Since the action on $N/N^0$ is unipotent, it follows
(\cite{War}\;9.3) that $G/N^0$ is nilpotent (if $K$ is connected one can also
apply Corollary\;\ref{cor45}). \
(iv),(v) are shown similarly.
\\[0mm plus.2mm]
(c) This follows immediately from compactness of \,$\widetilde G/G$ and
$\widetilde G/\widetilde N$.
\\[0mm plus.2mm]
(d) If $\widetilde N$ is abelian, then obviously $N$ is an
$FC^-_G$-group. The subgroup $L_1$ ($\subseteq \widetilde M\times K_1$)
of Corollary\;\ref{cor48} is abelian
as well and it satisfies $G_1 = NL_1$ and $G/G_1$ is compact.
\\[-.2mm]
For the converse, we can (passing to a subgroup of finite index) assume that
$K$ is connected. $N$ is $K_1$-invariant by Proposition\;\ref{pro44}\,(d).
If $N$ is an $FC^-_G$-group, then $N$ must be abelian
(there are no non-trivial unipotent inner automorphisms) and
$L_1 K_1\,\cap \widetilde N$ acts trivially on $N$ by
Corollary\;\ref{cor45}. Thus 
$G_1 K_1 \cap \widetilde N$ commutes with $N$. Since $\widetilde N$
is torsion free and nilpotent, the centralizer of $N$ is a connected subgroup
of $\widetilde N$ (see also Remark\;\ref{rem28}).
By Proposition\;\ref{pro44}\,(a), $G_1 K_1 \cap \widetilde N$ is co-compact
in $\widetilde N$, hence
(\cite{Ra}\;Th.\,2.1(4)\,), $N$ is central in $\widetilde N$\,. Put 
$N'=[K,N]=[K,N^0]\,,
\;\widetilde M' =\{x\in\widetilde N\!: k\circ x = x \text{ for all }k\in K\}$,
\ $\widetilde L' = \widetilde M'\times K$\,.
Then $N'$ is a closed normal subgroup of
$\widetilde G$\,,\;$N' \cap \widetilde M'$ is trivial and
$\widetilde N = N'\widetilde M'$ by \cite{Lo3}\;L.\,5.4 and
Proposition\;\ref{pro47}\,(d) \,(in particular $N'=[K,\widetilde N]$). Thus
\,$\widetilde G=N'\rtimes\widetilde L'$. Let \,$H' = N'H\,\cap\widetilde L'$
be the projection of $H$ to 
$\widetilde L'$. Then $H'$ is an abelian subgroup of $G$ and
\,$(\widetilde M' \cap N)H'$ is co-compact in $\widetilde L'$ and
abelian (observe that $\widetilde M'\cap N$ is central in $\widetilde G$).
Consequently,
$H'' =\bigl((\widetilde M' \cap N)H'K\bigr) \cap \widetilde M'$
\,(i.e., the projection of $(\widetilde M' \cap N)H'$ to $\widetilde M'$) is
a co-compact abelian subgroup of \,$\widetilde M'$. As above, it follows that
$\widetilde M'$ must be
abelian and this implies that $\widetilde N$ is abelian.
\end{proof}\vspace{-2mm plus .2mm}
\begin{Rems} \label{rem411}
(a) \ We want to relate our results to the notions of \cite{Ra}. Let $G$
be a compactly generated group of polynomial growth without non-trivial compact
normal subgroups and let $\pi$ be a continuous faithful
representation of $G$ on $\R^n$.
Denote by $\widetilde G$ the Zariski-closure of
$\pi(G)$ in $\GL(n,\R)$. Let $\widetilde{N}$ be the unipotent radical of
$\widetilde G$\,. Then we have a ``Levi decomposition" (in the sense of
algebraic groups)
\,$\widetilde G = \widetilde N \rtimes K$, where $K$ is a maximal reductive
subgroup of $\widetilde G$ \,(see \cite{Ra}\;p.\,11, \cite{Ab1}\;p.\,296).
Then \,(putting as before $N = \nil(G)$\,) one can show that the following
properties are equivalent:\vspace{-1.3mm}
\begin{enumerate}
\item[(i)] \;$\pi(N)$ consists of unipotent matrices,\ $\pi(G)$ is closed
(for the
Euclidian topology of $\GL(n,\R)$\,) and distal.
\item[(ii)] \;$\pi(N)$ consists of unipotent matrices,\ $\pi(G)$ is closed
(Euclidian topology) and $K$ is compact.
\end{enumerate}
If this holds, it follows that $\widetilde G/\pi(G)$ is compact. Furthermore,
if the action of $K$ on $\widetilde N$ is faithful \,(i.e.,
$C_{\widetilde G}(\widetilde N)\subseteq\widetilde N$\,),
then (i) and (ii) are equivalent to\vspace{-.5mm}
\begin{enumerate}
\item[(iii)] \;$\widetilde G$ is an algebraic hull of $G$ \,(as defined
in \cite{Ra}\;Def.\,4.39).
\end{enumerate}
(Be aware that in \cite{Au}\;p.\,228 the term algebraic hull is used in a
much wider sense).\vspace{2mm plus .3mm}

Thus, in the case of a faithful action, $\widetilde G$ coincides with the
groups considered in Theorem\;\ref{th2}\,and\,\ref{th3}.
To be precise: \,\cite{Ra} considers complex
algebraic groups \,(i.e., the
Zariski closure in $\GL(n,\C)$\,), thus our $\widetilde G$ is the ``real
algebraic hull", i.e., the set of real points of the algebraic hull
in the sense of \cite{Ra}.
In particular, it follows from Theorem\;\ref{th3} that all algebraic hulls
(in the sense of \cite{Ra}) are
isomorphic \,(this has also been shown in \cite{Ra}\;L.\,4.41). Since we are
dealing with groups of polynomial growth, one can
show (similarly as in the proof of \cite{Ra}\;L.\,4.36, using a corresponding
definition of the ``rank" for generalized $\overline{FC}$-groups)
that the condition
``$\pi(G)$ is closed" of (i),(ii) is equivalent to ``$\pi$ is full" in
the sense of \cite{Ra}\;Def.\,4.37, i.e., $\dim(\widetilde N)=\rk(G)$\,.
\vspace{0mm plus .3mm}

The representation (coming from the Birkhoff embedding theorem) that was used
in the proof of Corollary\;\ref{cor36} has the properties leading to (iii).
But
in general, there are also faithful finite dimensional representations of $G$
which satisfy (i) and~(ii), but $K$ does not act faithfully on
$\widetilde N$ (see Examples\;\ref{ex412}\,(d)\,). Let $K_{\widetilde N}$ be
the kernel of the action of $K$ on
$\widetilde N$. Then $K_{\widetilde N}$ is normal in $\widetilde G$ and
by Theorem\;\ref{th3}, $\widetilde G/K_{\widetilde N}$ is isomorphic (as a
locally compact group) to the algebraic hull of $G$.\vspace{.2mm plus .5mm}

In the case of discrete generalized $\overline{FC}$-groups (i.e., finite
extensions of polycyclic groups) another construction of the algebraic hull
(using Hopf algebras and working on arbitrary fields of characteristic zero)
has been described in \cite{Do} (see L.\,4.1.2, Prop.\,4.2.2,\;4.3.2). A
more explicit version in terms of a ``basis" of the group has been given in
\cite{Sa}.
\vspace{.2mm plus .7mm}
\item[(b)] In general, there are further almost nilpotent groups lying between
$G$ and $\widetilde G$\,. The group $G_{an}$ of Proposition\;\ref{pro38} is a
co-compact extension of $G$ that is almost nilpotent and has no non-trivial
compact normal subgroup. For $K_1$ one can take that of
Proposition\;\ref{pro47}\,(b) and for
a given hull $\widetilde G$, the group $G K_1$ does not depend on the
choice of $K$. But in
general, $G$ need not be $K_1$-invariant (in particular, $G$~need not be
normal in $G_{an}$) and $G_{an}$ need not split into a semidirect product of a
nilpotent group and a compact group. $G_{an}$ need not be a minimal almost
nilpotent extension of $G$ \,(see Examples\;\ref{ex412}\,(a),(f)\,).
\vspace{0mm plus .4mm}

When $G$ is connected, simply connected and solvable,
$\widetilde G = G_{an}$ coincides with
the semisimple splitting of \cite{Au}\;p.\,237, $\widetilde N$ is called the
nil-shadow of $G$ (in the notation of \cite{Au}: \ $\widetilde G = R_S\,,\
\widetilde N = M_R\,,\ K_1 = T_R$\,, where $R = G$\,).
Hence in the general case of our Theorem\;\ref{th2}, we call
$\widetilde N$ the {\it connected nil-shadow} of $G$\,. As mentioned in
Remark\;\ref{rem221}\,(g) this coincides with the notions of \cite{Au} and
\cite{Br} for connected, simply connected, solvable Lie groups.

In \cite{Ab2}\;Th.\,3.6, an arbitrary connected Lie group $G$ of
polynomial growth is embedded (as a closed normal subgroup) into a
connected Lie group $H$ such that $H/G$ is compact and $H$ has a co-compact
normal subgroup $M_0$ that is connected and nilpotent.
But in general $M_0$ need
not be simply connected, even when $G$ has no non-trivial compact normal
subgroups (for $G$ solvable with $[G,G]^-$ torsion free, $H$ coincides with
the group $G'$ of Corollary\;\ref{cor215a}, $M_0=N'$, see also
Remark\;\ref{rem215b}).
Thus, this does not always coincide with our algebraic hull.
\vspace{0mm plus .2mm}

In the non-connected case, one can consider splittings where the nilpotent
factor is not necessarily connected. This is related to the ``discrete
semisimple splitting" mentioned in \cite{Au}\;p.\,253,
see also \cite{Se}\;p.\,141.
Let $N^K$ be the closed $K$-invariant subgroup of $\widetilde N$ generated by
$GK \cap \widetilde N$\,.
Then $G \subseteq N^K \rtimes K$ \,(and $N^K$ is minimal to get such
a splitting
for given $K$). But in general, $N^K$ depends on the choice of $K$\,. One can 
show that it is always possible to choose $K$ so that $N^K \rtimes K$ is a
finite extension of $G_{an}$\,. But in general there is no uniqueness result 
corresponding to Theorem\;\ref{th3} \,(see Examples\;\ref{ex412}\,(c);
this aspect is somehow concealed in the
formulation of \cite{Au}\;p.\,254;
compare also \cite{Se}\;Th.\,3,\;p.\,147).\vspace{0mm plus .2mm}

In \cite{Mo}\;Sec.\,2, another construction of the nil-shadow based on
representative functions \,(and working for an arbitrary generalized
$\overline{FC}$-group $G$ without non-trivial compact normal subgroups)
is given. If $\pi$ is any continuous finite dimensional representation
of $G$ and $\widetilde N_{\pi}$ denotes the unipotent radical of the
real Zariski closure of $\pi(G)$\,, then $\widetilde N_{\pi}$ is a
quotient of the nil-shadow $\widetilde N$\,, but the reductive part can
become arbitrarily large (unless $G/\nil(G)$ is finite),
\vspace{1mm plus .5mm} compare the Examples\;\ref{ex412}\,(d).
\item[(c)] \,\cite{Wi2}\;Ex.\,2.3 shows that Proposition\;\ref{pro47}\,(c)
need not hold when $\widetilde N$ is replaced by a general connected, simply
connected and {\it solvable} group.
\item[(d)] \,In \cite{Br}\;Th.\,1.2, it is shown that if $G$ is a compactly
generated l.c.\;group of polynomial growth having no non-trivial compact
normal subgroup, then there exists a co-compact closed subgroup $H$ that
can be embedded (as a closed subgroup) into a connected, simply connected,
solvable Lie group $S$\,. The proof (given in \cite{Br}\;7.1) reduces it
in several steps to a corresponding embedding theorem (\cite{Wan}\;Th.\,3)
for $\Cal S$-groups. He calls $S$ a ``Lie shadow" of $G$\,. It is
necessarily of polynomial growth, but in general not unique
(see \cite{Br}\;p.\,671). It follows from our Theorem\;\ref{th3} that the
algebraic hull of $S$ contains the algebraic hull of $H$ which is
contained in the algebraic hull of $G$\,. In particular, the nil-shadow
of $S$ must coincide with the connected nil-shadow of $G$ (fixing also
the dimension of $S$).
\end{Rems}
\begin{Exs}	\label{ex412}
{\bf (a)} \,We start with the examples given in \cite{Lo2}\;1.4.3. For
\,$G = \C\rtimes\Z$ with the action \,$n\circ z= \alpha^n z$\,, where
$|\alpha| = 1$ and $\alpha$ is not a root of unity, we get
\,$\widetilde G = (\C \times \R) \rtimes K$ with
$K = \{\beta\in\C\!: |\beta| = 1\}\ (= K_1),\
\beta \circ (z,t) = (\beta z,t)$
and the embedding
$(z,n) \mapsto (z,n, \alpha\,n),\ \widetilde N = \C \times \R\,,\
\widetilde M = \R\,,\ L = \Z\,,\ G_{an} = (\C \times \Z) \rtimes K$\,.
\\
For $G = \C^2 \rtimes \R$ with
\,$t \circ (z_1, z_2) = (e^{it\beta_1} z_1, e^{it\beta_2} z_2)$\,,
we get \,$\widetilde G = (\C^2 \times \R) \rtimes K\ (= G_{an})$,
where $K\,(= K_1)$ denotes the closure of 
\,$\{(e^{it\beta_1},e^{it\beta_2}): t \in \R\},\ \,
(\gamma_1,\gamma_2) \circ (z_1,z_2, t) =
(\gamma_1 z_1, \gamma_2 z_2, t)$ and (writing
$\mathbf z = (z_1, z_2)$\,) the embedding
$(\mathbf z,t) \mapsto (\mathbf z,t,(e^{it\beta_1},e^{it\beta_2})),
\linebreak \widetilde N = \C^2 \times \R$\,.
\\
Similarly, for $G = \R^n \rtimes \Z$
with the action $n \circ v = A^n v$\,, where $A \in \GL(n,\R)$ and all
eigenvalues of $A$ have modulus $1$. We consider the multiplicative Jordan
decomposition $A = A_s A_u$\,. We get
\,$\widetilde G = (\R^n \rtimes \R) \rtimes K$\,, where $K$ denotes the
closure of $\{A^n_s\!: n \in \Z\}$ and the actions
are \,$t \circ v = e^{tB}v$ with $B = \log A_u\,,\ C \circ (v,t) = (Cv,t)$
for $C \in K$\,. The embedding is given by 
$(v,n) \mapsto (v,n,A^n_s),\ \widetilde N = \R^n \rtimes \R$\,.
If no root of unity is an eigenvaulue of $A$\,, then 
$\widetilde M = \R\,,\ L = \Z$\,.
Otherwise, $\widetilde M$ includes the eigenspaces of $A_s$ for the roots
of unity and if one of these eigenvalues is different from 1, the action of
$K$ on $\widetilde M$ is non-trivial. If $K^0\,(= K_1)$ is non-trivial
(i.e., $A$ has at least one eigenvalue that is not a root of unity), then 
$N = \R^n,\ G_{an} \subseteq  (\R^n \rtimes \Z) \rtimes K$\,, but if
$K^0 \neq K$ \,(e.g., $A$ has also an eigenvalue that is a root of unity
different from~$1$), the inclusion is proper and $G_{an}$ does not split.
\\[.3mm plus .4mm]
Similarly, for $G = \R^n\rtimes\R$\,. For example, in the case
$G = \C \rtimes \R$ with $t \circ z = e^{it}z$\,, one has
\,$\widetilde G = (\C \times \R) \rtimes K\ (= G_{an})$ with
$K = \{\beta \in \C\!: |\beta| = 1\},\ \beta \circ (z,t) = (\beta z,t)$ and
the embedding 
$(z,t) \mapsto (z,t,e^{it}),\
\widetilde N = \C \times \R\,,\ \widetilde M = \R\,,\ L = \R\,,\
N = \C \times 2 \pi \Z$\,. Thus $G$ is almost nilpotent, but
$G_{an} \neq G$\,, i.e., $G_{an}$ is not minimal.
\vspace{1mm plus .8mm}
\item[\bf (b)]  An example where the action of
$K$ on $\widetilde N$ is not faithful\vspace{.1mm}
(notation of Theorem\;\ref{th3}): take
\,$G = \R\,,\ \widetilde G = \R\times K$ with
$K = \R/\Z\,,\ j(t) = (t, t + \Z)$\,. Here $\widetilde{G}/j(G)$ is compact,
but $\widetilde G$ is not 
isomorphic to the algebraic hull of $G$ \,(which coincides with $G$).
\\[.6mm plus .5mm]
An example where $j(G)$ is not closed: take
\,$G = \Z^2\,,\ \widetilde G = \R\,,\
j(n,m) = n\alpha + m\beta$ \,where $\alpha, \beta \in \R$ are
$\Q$-linearly independent. Then
$j(G)$ is dense in $R$, but not closed, and the algebraic hull of $G$
is $G_{\R} = \R^2$.
\\[.1mm plus .3mm]
These examples can also be used to show that in Remark\;\ref{rem411}\,(a) the
assumptions \,$\pi(N) \subseteq  \widetilde N$ and \,$\pi(G)$ closed cannot be
dropped.
\vspace{.5mm plus .7mm}
\item[\bf (c)] For $G$ almost nilpotent, one has \,$\widetilde N=N_{\R}$ by
Corollary\;\ref{cor410}\,(c), and conversely.
To get examples for the discrete case (where $G$ is a finite
extension of a nilpotent group), put 
$\widetilde N = \R^2,\ \alpha_1 (x_1,x_2) = (-x_1,x_2),\
\alpha_2(x_1,x_2) = (x_1 - x_2),\ K\ (\cong \Z^2_2)$ the subgroup
of $\GL(2,\R)$ generated by $\alpha_1,\alpha_2\,,\ \,
\widetilde G = \widetilde N \rtimes K\,,\ N = \Z^2$ and $G$ shall be the
subgroup of $\widetilde G$ generated by $N$
and $((0,0), \alpha_1)\,,\,((\frac12,0), \alpha_2)$.
$\widetilde N/N$ being compact, it follows that
$\widetilde N\cong N_{\R}\,,\ N = G \cap \widetilde N$
and $\widetilde G$ is the algebraic hull of $G$\,.
Since $K$ is discrete, we have $G_{an}=G$.
Here, $N$ and $G$ are
$K$-invariant, $N^K = \frac12\Z\times\Z$\,. For $\mu = (0,1) \in N\,,\
K^{\mu} = \mu K \mu^{-1}$,
one gets $N^{K^{\mu}} = \{(x, y) \in (\frac12 \Z)^2 : x + y \in \Z\,\}$
and it is easy to see ($N^{K^{\mu}}$ does not split into cyclic
$K^{\mu}$-invariant
subgroups) that $N^{K^{\mu}}\!\rtimes K^{\mu}$ is not isomorphic to
$N^K\!\rtimes K$\,. Thus there are non-isomorphic discrete splittings. $G$ has
index $2$ in both extensions.
\\[0mm plus .4mm]
Observe (using \cite{Ho}\;Th.\,XV.3.1) that for every compact subgroup $C$
of $\widetilde G$ there exists $\mu \in \widetilde{N}$ such that
$\mu^{-1}C\mu\subseteq K$\,, in particular, $K^{\mu}\ (\mu \in \widetilde{N})$
gives all maximal compact subgroups of $\widetilde G$\,.
\\[0mm plus .3mm]
For further examples, consider $\widetilde N = \mathsf{H}\times \R$\,, where
$\mathsf{H}$ denotes the three-dimensional real Heisenberg group. Explicitly,
$\widetilde N = \R^4$
topologically, with multiplication\linebreak
$(x_1, x_2, t_1, t_2)\:
(x'_1, x'_2, t'_1, t'_2) =
(x_1 + x'_1,x_2 + x'_2, t_1 + t'_1 - x_2 x'_1, t_2 + t'_2)$.
Let $N$ be the (discrete) subgroup generated by
$(1,0,0,0)\,,\,(0,1,0,\frac14)\,,\,(0,0,\frac12,\frac12)$.
Writing $\mathbf v = (x_1,x_2,t_1,t_2)$, we get
$N = \{\,\mathbf v: \,x_1,x_2,2 t_1, 4 t_2 \in \Z\,,\
4 t_2 - 4 t_1 - x_2 \equiv 0 \pmod 4 \}$.
Consider $\alpha \in \Aut(\widetilde N)$ defined by 
$\alpha (\mathbf v) = (x_1, -x_2, -t_1, t_2),\
K = \langle\alpha\rangle,\
\widetilde G = \widetilde N\rtimes K$\,. Finally, let $G$ be the subgroup of
$\widetilde G$ generated by $N$ and $((\frac12,0,0,0),\alpha)$. Since
$\alpha (0,1,0,\frac14) \notin N$, we get that $N$ and $G$ are not
$\alpha$-invariant. 
Hence they are not $K$-invariant and the same can be shown if $K$ is replaced
by a conjugate group $\mu K \mu^{-1}\ (\mu \in \widetilde N)$. In a similar
way, one can construct examples where $N$ is \linebreak
$K$-invariant but $G$ is not $K$-invariant.
\\[0mm plus .4mm]
When $K$ is abelian (see Corollary\;\ref{cor410}\,(a)), one can show similar
statements as in \cite{Se}\;Th.\,1,\,p.\,143. Put 
$\widetilde M' =
\{x\in \widetilde N\!: k\circ x = x \text{ for all } k \in K\},\
M' = N^K \cap \widetilde M'$. Proposition\;\ref{pro47}\,(e) implies
$\widetilde N = N_{\R}\widetilde M'$. One can choose $K$ such that $N$ and $G$
are $K$-invariant and $N^K = NM'$ if and only if there exists a
nilpotent subgroup $L'$ of $G$ such that
$G = NL'$ and $N$ is $s(L')$-invariant \,(where as in \ref{di25},
$s(x) \in \Aut(N_{\R})$ denotes the semisimple part of the
automorphism $\iota_x$ of $N_{\R}$).
However, even under these stronger assumptions there is no uniqueness in
general. Similarly as above, one can construct non-isomorphic splittings 
$N^K \rtimes K$ of this type.
\\[0mm plus .5mm]
As mentioned before, \cite{Se} and \cite{Au} assumed that $G/N$ is torsion
free. But it is easy to modify the examples above to meet this requirement.
For example, the first one came from an action of $\Z^2_2$ on $\R^2$ (in
fact on $\Q^2$). This gives rise to a faithful action of $\Z^2$ on
$\R^2 \times \Z^4$ when combining with a faithful action of $\Z^2$ on $\Z^4$
by semisimple matrices (of course, this leads outside the scope of groups of
polynomial growth).
\vspace{1mm plus 1.2mm}
\item[\bf (d)] On faithful representations. In (b), we mentioned examples
concerning the conditions in (i),\,(ii) of Remark\;\ref{rem411}\,(a). Now
we consider the first example of~(a), \,$G = \C \rtimes \Z$\,. A natural
choice of a faithful representation would be
$\pi(z,n) =
\begin{pmatrix}\alpha^n & z\\0& 1 \end{pmatrix}
\vspace{-1mm}\in \GL(2,\C)\ \,(\subseteq \GL(4,\R)\,)$. But
$\pi(G)$ is not closed, the (real) Zariski closure gives
$\Bigl\{
\begin{pmatrix}\beta & z\\0 & 1\end{pmatrix}
\!: \beta, z \in \C,\ |\beta| = 1\Bigr\}\;\cong \C \rtimes K$
\vspace{-1.2mm} with 
$K = \{\beta \in \C\!: |\beta| = 1\}$\,.
Write
$\alpha = \alpha_1^2$\,, take $r\in\R$ with $\lvert r\rvert\neq 0,1$ and put 
\vspace{.5mm}$\pi_r(z,n) =
\begin{pmatrix}(r\alpha_1)^{n}& z \\ 0 & (r/\alpha_1)^n\end{pmatrix}
\linebreak\in \GL(2,\C)\ \,(\subseteq \GL(4,\R)\,)$.\vspace{.3mm}
Now $\pi_r(G)$ is closed but not distal, the (real) Zariski closure gives
$\Bigl\{
\begin{pmatrix}\beta & z\\0 & \gamma\end{pmatrix}
\!: \beta,\gamma,z \in \C,\ \beta\,\gamma\in\R^*\Bigr\}\;
\cong \C \rtimes K\times \R^*$
\vspace{-1mm}
(with $\R^*=\R\setminus\{0\}$ non-compact\vspace{-6mm}).
$$\hspace{-2.4cm}\text{Put}\qquad\qquad\pi_{alg}(z,n) =
\begin{pmatrix}\alpha^n & z&0&0\\0& 1&0&0\\0&0&1&n\\0&0&0&1 \end{pmatrix}
\vspace{-1.5mm}\in \GL(4,\C)\ \,(\subseteq \GL(8,\R)\,)\,.
$$
$\pi_{alg}(G)$ satisfies all the properties (i)-(iii) of
Remark\;\ref{rem411}\,(a). The (real) Zariski closure (which gives the
algebraic hull, isomorphic to the version in (a)\,) is\vspace{-2.5mm}
$$\Biggl\{\
\begin{pmatrix}\beta & z&0&0\\0& 1&0&0\\0&0&1&t\\0&0&0&1 \end{pmatrix}
\!: \beta, z \in \C,\ |\beta| = 1,\ t\in\R\,\Biggr\}\,.\vspace{-2.5mm}
$$
To get an example of
a faithful representation of $G$ satisfying (i),(ii), but not (iii), of
Remark\;\ref{rem411}\,(a),
take (with $\alpha_1$ as above) \quad $\pi'(z,n) =$
\\[1mm]
$\begin{pmatrix}\alpha_1^n & z&0&0\\0& \alpha_{1}^{-n}&0&0\\0&0&1&n\\0&0&0&1
\end{pmatrix}\,,$ giving\vspace{.5mm}
\;$\Biggl\{\ \!
\begin{pmatrix}\beta & z&0&0\\0& \beta^{-1}&0&0 \\ 0&0&1&t \\ 0&0&0&1
\end{pmatrix}
\!: \beta, z \in \C,\ |\beta| = 1,\ t\in\R\,\Biggr\}$ \linebreak
as (real) Zariski closure\,.  Then the corresponding action of
\,$K \cong \{\beta \in \C\!: |\beta| = 1\}$ on \,
$\widetilde N \cong \C\times\R$ \,is \,$\beta \circ (z,t) = (\beta^{2}z,t)$\,.
Thus $K$ does not act faithfully.
\\
The last example excludes a possible converse in
Proposition\;\ref{pro47}\,(c)\,:
when $G$ has no non-trivial compact normal subgroups, $K$ need not act
faithfully on $\widetilde N$.
\vspace{1mm}
\item[\bf (e)] In Corollary\;\ref{cor410}\,(d), the condition of the existence
of an abelian almost-\-supplementary group $H$ cannot be dropped (i.e., for
$\widetilde N$
to be abelian, it is not enough that $N$ is an ${FC}^-_{G\,}$-group). Let
$\alpha, \beta \in \R$ be $\Q$-linearly independent. Consider
($\mathsf H_{\Z}$ denotes the discrete Heisenberg group) the group
$G = \mathsf H_{\Z}\ltimes \C$ 
given by $\Z^{3} \times \C$ topologically, with multiplication
\;
$(k, l, m, z)\:(k', l', m', z') = 
(k + k', l + l',\linebreak m + m' + l\,k',e^{i(k'\alpha + l'\beta)}z + z')$.
Then $N = \{(0,0,m,z)\! : m \in \Z,\, z \in \C\}$ is an 
${FC}^-$-group (observe that the action of $\mathsf{H}_{\Z}$ on $\C$ is
semisimple). But (similarly as in (a)\,)
\,$\widetilde N = (\mathsf H_{\Z})_{\R} \times \C\ (=\mathsf H\times\C),\
K \cong \{\gamma \in \C\!: |\gamma| = 1\},\
\widetilde M = (\mathsf H_{\Z})_{\R}$\,, thus
$\widetilde N$ is not abelian. In the notation of Corollary\;\ref{cor48},
one has $L = \mathsf H_{\Z}$\,.
\vspace{1mm}
\item[\bf (f)] In the notation of Proposition\;\ref{pro47}\,(b),
$G$ need not be $K_1$-invariant (but, as mentioned earlier,
$G_{an} = G\,K_{\!1}$ is always a group).
Let $G = \C \rtimes \Z$ be the first example of (a) and define
$\sigma \in \Aut(\widetilde G \times \widetilde G)$ by
\,$\sigma (x,y) = (y,x)$. Let
$W\ (\cong \Z_2)$ be the subgroup generated by $\sigma$ and put
$\widetilde G' = (\widetilde G \times \widetilde G) \rtimes W$. Let 
$j\!: G \to \widetilde G$ denote the embedding and consider the
subgroup $G'\ (\,\cong (G\times G)\rtimes W)$ of $\widetilde G$ generated by
$j(G)\times j(G)$ and $\sigma$. Then $K' = (K \times K) \rtimes W$ gives
a compact component for $\widetilde G',\ (K')^0 = K \times K$ and taking 
$x \in K$ such that $x^2 \notin \{\alpha^n : n \in \Z\},\ G'$ is not
invariant under the inner automorphism of $\widetilde G'$ defined by
$(x,x)\in K'$. Similarly for subgroups conjugate to $K'$.
\\
Taking $\mu \in \widetilde M \times \widetilde M$ such that
$\sigma(\mu)\,\mu^{-1} \notin \Z^2$, one gets
\,$\mu K' \mu^{-1} \cap G' = K \times K$\,. It follows that
$\mu K' \mu^{-1} \cap G'$ is not a maximal compact subgroup of $G'$, hence
the corresponding statement
of Proposition\;\ref{pro47}\,(d) does not extend to the non-connected case.
\end{Exs}


\end{document}